\documentclass[11pt,reqno]{amsart}

\usepackage[T1]{fontenc}
\usepackage{lmodern}
\usepackage{amsmath,amssymb,mathtools,mathrsfs}
\usepackage{microtype}
\usepackage[colorlinks=true,linkcolor=blue,citecolor=blue,urlcolor=blue]{hyperref}

\numberwithin{equation}{section}
\allowdisplaybreaks

\newtheorem{theorem}{Theorem}[section]
\newtheorem{proposition}[theorem]{Proposition}
\newtheorem{lemma}[theorem]{Lemma}
\newtheorem{corollary}[theorem]{Corollary}
\theoremstyle{definition}
\newtheorem{definition}[theorem]{Definition}
\newtheorem{notation}[theorem]{Notation}
\theoremstyle{remark}
\newtheorem{remark}[theorem]{Remark}

\DeclareMathOperator{\pardeg}{pardeg}

\hypersetup{
  pdftitle={On possible values of the signature of flat unitary bundles over compact surfaces},
  pdfauthor={Inkang Kim, Pierre Pansu, Xueyuan Wan},
  pdfsubject={Possible signatures of flat unitary bundles over compact surfaces},
  pdfkeywords={signature, flat Hermitian bundle, surface group representation, Toledo invariant, rho invariant, parabolic Higgs bundle, Milnor--Wood inequality}
}

\title[Possible signatures of flat unitary bundles]
{On possible values of the signature of flat unitary bundles over compact surfaces}

\author{Inkang Kim}
\address{School of Mathematics, Korea Institute for Advanced Study (KIAS),
85 Hoegiro, Dongdaemun-gu, Seoul 02455, Republic of Korea}
\email{inkang@kias.re.kr}

\author{Pierre Pansu}
\address{Universit\'e Paris-Saclay, CNRS, Laboratoire de Math\'ematiques d'Orsay,
91405 Orsay Cedex, France}
\email{pierre.pansu@universite-paris-saclay.fr}

\author{Xueyuan Wan}
\address{Mathematical Science Research Center, Chongqing University of Technology,
Chongqing 400054, China}
\email{xwan@cqut.edu.cn}

\date{}

\makeatletter
\@namedef{subjclassname@2020}{\textup{2020} Mathematics Subject Classification}
\makeatother

\begin{document}

\begin{abstract}
We determine all possible signatures of flat Hermitian bundles over compact,
connected, oriented surfaces of positive genus with nonempty boundary. If $\Sigma_{g,n}$ has genus $g\geq 1$ and $n\geq 1$ boundary
components, then the signatures arising from representations
$\pi_1(\Sigma_{g,n})\to\mathrm{U}(p,q)$ are exactly the integers $m$ satisfying
\[
|m|
\leq
(p+q)(2g+n-2)
-(2g-2)|p-q|
-\min\{2,n|p-q|\}.
\]
Every such integer, including the two extremal values, is realized by a
block-diagonal representation whose image is contained, after possibly
interchanging $p$ and $q$, in
$\mathrm{U}(1,1)^{\times\min\{p,q\}}\times\mathrm{U}(|p-q|)$.
\end{abstract}

\thanks{Research by Inkang Kim is partially supported by RS-2026-25468457 and
KIAS Individual Grant (MG031408), Pierre Pansu is supported by Agence Nationale de la Recherche,
 ANR-22-CE40-0004 GOFR, and Xueyuan Wan is supported by the
National Key R\&D Program of China (Grant No. 2024YFA1013200) and the National Natural Science Foundation of China (Grant No. 12671100).}

\maketitle

\section{Introduction}
\label{sec:introduction}
Representation spaces of surface groups lie at a meeting point of topology,
differential geometry, complex geometry, and Lie theory. Among the numerical
invariants attached to a surface group representation, the signature is
distinguished by its simultaneously topological and analytic nature. Let
\(\Sigma\) be a compact oriented surface and let
\(\phi\colon\pi_1(\Sigma)\to\mathrm{U}(p,q)\) be a representation. The
associated flat Hermitian bundle \(\mathcal E_\phi\) determines a
nondegenerate Hermitian intersection form on the space
\(
\operatorname{Im}(
H^{1}(\Sigma,\partial\Sigma;\mathcal{E}_{\phi})
\to
H^{1}(\Sigma;\mathcal{E}_{\phi}))
\)
and its signature will be denoted by \(\operatorname{sign}(\phi)\).

For closed manifolds, signatures of flat bundles were studied systematically
by Lusztig and Meyer, who related them to indices of elliptic operators and to
characteristic classes \cite{Lusztig,MeyerLocal,MeyerSurface}. For a manifold with boundary, the
Atiyah--Patodi--Singer index theorem includes an additional term determined by
the boundary operator \cite{APS1,APS3}. In the case of surfaces,
Atiyah developed this point of view further and related the resulting
correction terms to his signature cocycle and to the logarithm of the Dedekind
\(\eta\)-function \cite{Atiyah}.

A parallel geometric invariant arises from the Hermitian symmetric space
associated with \(\mathrm{U}(p,q)\). The Toledo invariant was introduced in
the study of representations into complex hyperbolic groups
\cite{Toledo}. For representations into
\(\mathrm{PU}(1,1)\cong\mathrm{PSL}(2,\mathbb R)\) over a closed surface, it
agrees, up to the choice of orientation convention, with the Euler number.
The classical Milnor--Wood inequalities bound this number
\cite{Milnor,Wood}. More generally, bounds for the Toledo invariant are
controlled by the norm of the bounded K\"ahler class; see
\cite{DomicToledo,BIW}. Burger--Iozzi--Wienhard extended the Toledo invariant to
surfaces with boundary by means of bounded cohomology and proved that the
Milnor--Wood inequality remains valid in this setting \cite{BIW}. For closed
surfaces the Toledo invariant is discrete, whereas for surfaces with boundary
it may vary continuously. The signature, by contrast, remains integer-valued.

The authors related these invariants in \cite{KPWUnitary}. For every representation
\(\phi\colon\pi_1(\Sigma)\to\mathrm{U}(p,q)\), one has
\begin{equation}
\label{eq:introduction-signature-Toledo}
\operatorname{sign}(\phi)
=
-2T(\Sigma,\phi)
+
\rho_\phi(\partial\Sigma),
\end{equation}
where
\[
\rho_\phi(\partial\Sigma)
=
\sum_{c\subset\partial\Sigma}\rho\bigl(\phi(c)\bigr)
\]
is a boundary invariant depending on the boundary holonomies. The function
\(\rho\colon\mathrm{U}(p,q)\to\mathbb R\) is generally discontinuous, and its
value can be computed from the eigenvalues and unipotent parts of a boundary
holonomy. The signature also satisfies the Milnor--Wood type estimate
\cite[Theorem~6.5 and Equation~(6.6)]{KPWUnitary}
\begin{equation}
\label{eq:introduction-coarse-MW}
\left|\operatorname{sign}(\phi)\right|
\leq
(p+q)|\chi(\Sigma)|.
\end{equation}
On representation spaces with prescribed boundary holonomies, the signature
can help distinguish connected components. This approach has been applied to
representations into
\(\mathrm{SL}(2,\mathbb R)\) and \(\mathrm{PSL}(2,\mathbb R)\)
\cite{KWComponents}.

A natural problem is therefore to determine exactly which integers occur as
signatures of representations into \(\mathrm{U}(p,q)\). The extremal case also connects this problem with the rigidity of
maximal Toledo representations. For surfaces of negative Euler
characteristic and $0<p<q$, Burger--Iozzi--Wienhard proved that maximal
Toledo representations preserve a maximal tube type subdomain,
whose stabilizer is conjugate to
$\mathrm{U}(p,p)\times\mathrm{U}(q-p)\subset\mathrm{U}(p,q)$
\cite{BIW}.
Related rigidity results for uniform lattices in higher-dimensional
complex hyperbolic spaces were obtained by Koziarz and Maubon
\cite{KoziarzMaubon2017}.
This suggests asking whether the maximal signature is likewise
attained by representations into
$\mathrm{U}(p,p)\times\mathrm{U}(q-p)$.
For surfaces with boundary, the rho correction in
\eqref{eq:introduction-signature-Toledo} prevents the Toledo bound
alone from determining the maximal signature.
Our main theorem answers this question affirmatively for surfaces
of positive genus with nonempty boundary.
In fact, we prove a stronger realization statement: every possible
signature is attained by a representation into the smaller
block-diagonal subgroup
$\mathrm{U}(1,1)^{\times p}\times\mathrm{U}(q-p)$;
see Theorem~\ref{thm:introduction-main}(3).
Thus passing from these direct sums to arbitrary
$\mathrm{U}(p,q)$-representations does not enlarge the range
of signatures.

The analogous
problem for symplectic groups was solved in \cite{KPWValues}: for a surface with
nonempty boundary and negative Euler characteristic, every integer permitted
by the signature Milnor--Wood inequality is realized by a representation into
\(\mathrm{Sp}(2p,\mathbb R)\). The same work settled the balanced unitary
case \(\mathrm{U}(p,p)\), the compact case \(\mathrm{U}(p)\), and the
general unitary case over genus-zero surfaces. The remaining case was
\(\mathrm{U}(p,q)\) with \(p,q>0\) and \(p\neq q\), over surfaces of
positive genus with nonempty boundary.

We settle this remaining case and combine it with the known and elementary
cases to give the complete range for every compact oriented surface. For
positive genus and nonempty boundary, the answer distinguishes the balanced
case \(p=q\) from the unbalanced case \(p\neq q\). In the balanced case,
the general signature Milnor--Wood bound is attained. In the unbalanced
case, the sharp bound is smaller by an explicit quantity depending on
\(|p-q|\), the genus, and the number of boundary components.

Let \(\Sigma_{g,n}\) be a compact, connected, oriented surface of genus \(g\)
with \(n\geq0\) boundary components, and write
\(
\chi_{g,n}:=\chi(\Sigma_{g,n})=2-2g-n.
\)
For \(p,q\geq0\), put
\(
a:=\min\{p,q\},
\) \(
b:=\max\{p,q\},
\)
and let
\[
\mathscr S_{p,q}(\Sigma_{g,n})
:=
\left\{
\operatorname{sign}(\phi)
\;\middle|\;
\phi\in
\operatorname{Hom}\bigl(\pi_1(\Sigma_{g,n}),\mathrm{U}(p,q)\bigr)
\right\}.
\]
For \(n\geq1\) and \(r\geq0\), define
\begin{equation}
\label{eq:introduction-delta}
\delta(n,r)
:=\min\{2,nr\}\end{equation}
and, when \(g,n\geq1\), set
\begin{equation}
\label{eq:introduction-M}
M_{g,n}(p,q)
:=
2a(2g-2)+n(a+b)-\delta(n,b-a).
\end{equation}

Our main theorem is the following.

\begin{theorem}
\label{thm:introduction-main}
Let \(p,q\geq0\) with \(p+q>0\).
\begin{enumerate}
\item If \(g=0\) and \(n\leq1\), then
\[
\mathscr S_{p,q}(\Sigma_{0,n})=\{0\}.
\]
If \(g=0\) and \(n\geq2\), then
\begin{equation}
\label{eq:introduction-main-genus-zero}
\mathscr S_{p,q}(\Sigma_{0,n})
=
\bigl[-(p+q)(n-2),(p+q)(n-2)\bigr]\cap\mathbb Z.
\end{equation}

\item If \(n=0\) and \(g\geq1\), then
\begin{equation}
\label{eq:introduction-main-closed}
\mathscr S_{p,q}(\Sigma_{g,0})
=
\bigl[-2a|\chi_{g,0}|,2a|\chi_{g,0}|\bigr]\cap4\mathbb Z.
\end{equation}

\item If \(g,n\geq1\), then
\begin{equation}
\label{eq:introduction-main-positive-genus}
\mathscr S_{p,q}(\Sigma_{g,n})
=
\bigl[-M_{g,n}(p,q),M_{g,n}(p,q)\bigr]\cap\mathbb Z.
\end{equation}
Moreover, every value in this interval is realized by a block-diagonal
representation whose image is contained, after possibly interchanging the
positive and negative subspaces, in
\[
\mathrm{U}(1,1)^{\times a}\times\mathrm{U}(b-a)
\subseteq
\mathrm{U}(a,b),
\]
where the compact factor acts on a definite subspace.
\end{enumerate}
\end{theorem}

Except for the divisibility condition on closed surfaces, the possible
signatures therefore form a full symmetric interval of integers. To see the
meaning of the positive-genus endpoint, write \(r=b-a\). Then
\[
M_{g,n}(p,q)
=
2a(2g-2+n)+nr-\min\{2,nr\}.
\]
The first term is the sharp bound for the balanced part
\(\mathrm{U}(1,1)^{\times a}\); the remaining terms give the sharp
bound for the definite part \(\mathrm{U}(r)\). The main theorem says that coupling
these parts in a general \(\mathrm{U}(p,q)\)-representation cannot
increase the bound. When
\(p=q=a\), the defect vanishes and
\[
M_{g,n}(a,a)=2a(2g-2+n)=2a|\chi_{g,n}|,
\]
so the usual signature Milnor--Wood inequality is sharp. When \(b>a\), the
endpoint may be rewritten as
\begin{equation}
\label{eq:introduction-improvement}
M_{g,n}(p,q)
=
(p+q)|\chi_{g,n}|-(b-a)(2g-2)-\delta(n,b-a).
\end{equation}
The improvement therefore consists of a term \((b-a)(2g-2)\) and a
boundary correction of at most \(2\). In particular, even for genus \(1\),
the boundary correction makes the coarse bound strictly larger than the
sharp bound whenever \(p\neq q\).

To realize all values, we take orthogonal direct sums of representations into
\(\mathrm{U}(1,1)\) and a compact unitary group. The main work is to prove
that no other representation has a larger signature. For a
reductive representation, parabolic nonabelian Hodge theory produces a
normalized logarithmic \(\mathrm{U}(p,q)\)-Higgs bundle
\[
(E=V\oplus W,\Phi),
\qquad
\Phi=
\begin{pmatrix}0&\beta\\ \gamma&0\end{pmatrix},
\]
on the compactification of the punctured surface
\cite{Corlette,Simpson,BGPMiR20}. Two residue operators enter the proof. The ordinary residue of \(\gamma\)
enters the estimates for its image and its saturation. The residue
on the associated graded bundle describes the local monodromy. These
operators need not have the same rank. The Euler characteristic of the
\(L^2\) Dolbeault complex accounts for their difference and gives the
sharp estimate. For a non-reductive representation, we then pass to a
suitable Levi quotient. This degeneration preserves the Toledo invariant,
and an estimate for the change in the boundary rho invariant completes an
induction on the positive rank.

Section~\ref{sec:signature-unitary} recalls the invariants and inequalities
used in the paper. Section~\ref{sec:special-signature-values} treats the
rank-one, genus-zero, compact, and closed-surface cases.
Section~\ref{sec:positive-genus-signatures} states the positive-genus result
and outlines its proof. Section~\ref{sec:reductive} proves the sharp bound
for reductive representations, and Section~\ref{non-reductive} extends it
to arbitrary representations. Section~\ref{subsec:realization-positive-genus}
realizes every integer in the resulting interval.
Appendix~\ref{app:L2-minimal-extension} provides the local analytic and
filtered comparison used in the reductive case.

\section{Background on signatures of flat unitary bundles}
\label{sec:signature-unitary}

This section fixes the conventions and notation used throughout the paper.
We recall the signature of a flat Hermitian bundle, the Toledo invariant of
Burger--Iozzi--Wienhard \cite{BIW}, and the boundary rho invariant introduced
in \cite{KPWUnitary}. We then state the signature--Toledo formula and the
Milnor--Wood inequalities needed below.

Let \(\Sigma=\Sigma_{g,n}\) be a compact, connected, oriented surface of
genus \(g\) with \(n\geq 0\) boundary components, so that
\(\chi(\Sigma)=2-2g-n\). Let \(E=\mathbb{C}^{p+q}\), where \(p,q\geq 0\)
and \(p+q>0\), and let \(\Omega\) be a nondegenerate Hermitian form on
\(E\) of signature \((p,q)\). Throughout, Hermitian forms are taken to be
linear in the first variable. We denote by
\[
\mathrm{U}(E,\Omega)
:=
\bigl\{
A\in\mathrm{GL}(E)
\;\big|\;
\Omega(Av,Aw)=\Omega(v,w)
\text{ for all }v,w\in E
\bigr\}
\]
the group of complex linear automorphisms of \(E\) preserving \(\Omega\).
After choosing an \(\Omega\)-orthonormal basis of \(E\), this group is
naturally identified with the group \(\mathrm{U}(p,q)\).

Given a representation
\(\phi\colon\pi_1(\Sigma)\to\mathrm{U}(E,\Omega)\), we denote by
\(\mathcal{E}_{\phi}:=\widetilde{\Sigma}\times_{\phi}E\) the associated
flat Hermitian vector bundle over \(\Sigma\). The Hermitian form
\(\Omega\) induces a parallel nondegenerate Hermitian pairing on
\(\mathcal{E}_{\phi}\), which will again be denoted by \(\Omega\).

\subsection{The twisted signature}
\label{subsec:twisted-signature}

Consider the image of relative twisted cohomology in absolute twisted
cohomology,
\[
\widehat{H}^{1}(\Sigma;\mathcal{E}_{\phi})
 :=
\operatorname{Im}\!\left(
H^{1}(\Sigma,\partial\Sigma;\mathcal{E}_{\phi})
\longrightarrow
H^{1}(\Sigma;\mathcal{E}_{\phi})
\right).
\]
For \(x,y\in \widehat{H}^{1}(\Sigma;\mathcal{E}_{\phi})\), choose a class
\(\widetilde{x}\in H^{1}(\Sigma,\partial\Sigma;\mathcal{E}_{\phi})\)
whose image is \(x\), and define
\begin{equation}
\label{eq:twisted-intersection-form}
Q_{\phi}(x,y)
:=
\left\langle
\Omega\bigl(\widetilde{x}\smile y\bigr),
[\Sigma,\partial\Sigma]
\right\rangle.
\end{equation}
Here the cup product is sesquilinear: \(\Omega\) contracts the coefficients
linearly in the first variable and conjugate-linearly in the second. Also,
\([\Sigma,\partial\Sigma]\) denotes the relative fundamental class.
The expression in \eqref{eq:twisted-intersection-form} is independent of
the choice of \(\widetilde{x}\). By Poincar\'e--Lefschetz duality,
\(Q_{\phi}\) is nondegenerate. Moreover, \(Q_{\phi}\) is skew-Hermitian,
and hence \(iQ_{\phi}\) is a nondegenerate Hermitian form; see
\cite[Section~2.1]{KPWUnitary}.

The signature of the flat Hermitian bundle
\((\mathcal{E}_{\phi},\Omega)\), or equivalently of the representation
\(\phi\), is defined by
\begin{equation}
\label{eq:def-signature}
\operatorname{sign}(\phi)
:=
\operatorname{sign}(iQ_{\phi}).
\end{equation}
Thus, if
\(\widehat{H}^{1}(\Sigma;\mathcal{E}_{\phi})
 =\mathscr{H}^{+}\oplus\mathscr{H}^{-}\),
where \(iQ_{\phi}\) is positive definite on \(\mathscr{H}^{+}\) and
negative definite on \(\mathscr{H}^{-}\), then
\(\operatorname{sign}(\phi)=\dim\mathscr{H}^{+}-\dim\mathscr{H}^{-}\).

The signature is invariant under conjugation of the representation and is
additive under orthogonal direct sums. More precisely, if
\(\phi=\phi_{1}\oplus\phi_{2}\) is induced by an orthogonal decomposition
of the underlying Hermitian space, then
\(\operatorname{sign}(\phi)
 =\operatorname{sign}(\phi_{1})+\operatorname{sign}(\phi_{2})\).

\subsection{The Toledo invariant}
\label{subsec:toledo-invariant}

The symmetric space associated with \(\mathrm{U}(p,q)\) is the classical
Hermitian symmetric space
\(\mathcal{X}_{p,q}
 =\mathrm{U}(p,q)/(\mathrm{U}(p)\times\mathrm{U}(q))\).
Its rank is \(\min\{p,q\}\). Put \(G=\mathrm{U}(p,q)\), and suppose
first that \(pq>0\). Let \(\omega_{\mathcal X}\) be the invariant
K\"ahler form whose metric has minimal holomorphic sectional curvature
\(-1\). We use the normalized bounded K\"ahler class
\(\kappa_G^b\in H_{\mathrm{cb}}^2(G;\mathbb R)\), represented by the cocycle
\[
(g_0,g_1,g_2)\longmapsto
\frac{1}{2\pi}\int_{\triangle(g_0x,g_1x,g_2x)}\omega_{\mathcal X},
\qquad x\in\mathcal X_{p,q}.
\]
The factor \(1/(2\pi)\) fixes the normalization in all formulas below.

Since every connected component of \(\partial\Sigma\) has amenable
fundamental group, the canonical map
\(
j_{\partial\Sigma}\colon
H_{b}^{2}(\Sigma,\partial\Sigma;\mathbb{R})
\to
H_{b}^{2}(\Sigma;\mathbb{R})
\)
is an isometric isomorphism. Following Burger--Iozzi--Wienhard
\cite[Section~1.1]{BIW}, the Toledo invariant of \(\phi\) is defined by
\begin{equation}
\label{eq:def-toledo}
T(\Sigma,\phi)
:=
\left\langle
j_{\partial\Sigma}^{-1}\phi_{b}^{*}(\kappa_{G}^{b}),
[\Sigma,\partial\Sigma]
\right\rangle.
\end{equation}
In this pairing, the relative bounded cohomology class is first mapped to
ordinary relative cohomology. We also use the identification
\(H_{b}^{2}(\pi_{1}(\Sigma);\mathbb{R})
 \cong H_{b}^{2}(\Sigma;\mathbb{R})\).
When \(pq=0\), the symmetric space \(\mathcal{X}_{p,q}\) is a point, so
\(\kappa_{G}^{b}=0\) and consequently \(T(\Sigma,\phi)=0\).

The Toledo invariant also admits an equivalent description in terms of a
compactly supported Chern--Weil form on the interior of \(\Sigma\); see
\cite[Section~3]{KPWUnitary}. The definition \eqref{eq:def-toledo} is the one used below.

\subsection{The boundary rho invariant}
\label{subsec:rho-invariant}

The relation between the signature and the Toledo invariant involves a
correction determined entirely by the boundary holonomies. Let
\(L\in\mathrm{U}(p,q)\), and let \(\mathcal{E}_{L}\to S^{1}\) be the flat
bundle with holonomy \(L\). Choose a compatible complex structure \(J\) on
\(\mathcal{E}_{L}\), a fixed point
\(W\in\overline{\mathcal{X}}_{p,q}\) of \(L\), and the corresponding
\(L\)-invariant primitive \(\alpha_{W}\) of the K\"ahler form. If
\(\widetilde{J}\colon\mathbb{R}\to\mathcal{X}_{p,q}\) is the associated
equivariant map and \(A_{J}\) is the induced self-adjoint boundary
operator, set
\begin{equation}
\label{eq:def-rho}
\rho(L)
:=
-\tfrac{1}{\pi}\int_{S^{1}}\widetilde{J}^{\,*}\alpha_{W}
+\eta(A_{J}).
\end{equation}
The right-hand side of \eqref{eq:def-rho} is independent of the choices of
\(J\) and \(W\). It therefore defines a conjugacy-invariant, generally
discontinuous, real-valued function on \(\mathrm{U}(p,q)\); see
\cite[Definition~4.1, Lemma~4.3, and Corollary~4.4]{KPWUnitary}.

Write
\(\partial\Sigma=c_{1}\sqcup\cdots\sqcup c_{n}\), with each boundary
component equipped with the induced boundary orientation. The rho
invariant of the boundary representation is
\(\rho_{\phi}(\partial\Sigma)
 :=\sum_{j=1}^{n}\rho(\phi(c_{j}))\).
With the above normalizations, the signature, the Toledo invariant, and
the boundary rho invariant are related by
\begin{equation}
\label{eq:signature-toledo-rho}
\operatorname{sign}(\phi)
 =
-2T(\Sigma,\phi)+\rho_{\phi}(\partial\Sigma).
\end{equation}
This is the signature--Toledo formula established in
\cite[Theorem~1]{KPWUnitary}. Analytically, the rho invariant in
\eqref{eq:signature-toledo-rho} records the boundary correction arising
from the Atiyah--Patodi--Singer index theorem; compare \cite{Atiyah}.
In particular, when \(\Sigma\) is closed, the boundary term disappears and
\(\operatorname{sign}(\phi)=-2T(\Sigma,\phi)\).

\subsection{Milnor--Wood inequalities}
\label{subsec:milnor-wood-unitary}

The Toledo invariant satisfies the Milnor--Wood inequality
\begin{equation}
\label{eq:mw-toledo}
\left|T(\Sigma,\phi)\right|
\leq
\operatorname{rank}(\mathcal{X}_{p,q})\,|\chi(\Sigma)|
=
\min\{p,q\}\,|\chi(\Sigma)|.
\end{equation}
For closed surfaces, this estimate originates in the computation of the
norm of the K\"ahler class by Domic--Toledo \cite{DomicToledo}. For \(\mathrm{U}(p,q)\), a direct Higgs-bundle proof is given by
Bradlow--Garc\'ia-Prada--Gothen
\cite[Section~3.4, especially Lemma~3.24, Corollary~3.27,
and Remark~3.29]{BGG03}.  The non-emptiness results
\cite[Theorem~6.1 and Remark~6.2]{BGG03} also show that the bound is
sharp in the Higgs-bundle moduli problem. Its
extension to compact surfaces with boundary is due to
Burger--Iozzi--Wienhard; see \cite[Theorem~1(1)]{BIW}.

The signature satisfies a different Milnor--Wood type inequality:
\begin{equation}
\label{eq:mw-signature}
\left|\operatorname{sign}(\phi)\right|
\leq
(p+q)|\chi(\Sigma)|.
\end{equation}
This follows from \cite[Theorem~6.5 and Equation~(6.6)]{KPWUnitary}.
The estimate \eqref{eq:mw-signature} is not a formal consequence of the Toledo inequality \eqref{eq:mw-toledo}, since the
boundary term in \eqref{eq:signature-toledo-rho} need not vanish. For
closed surfaces, however, the rho invariant is absent, and
\eqref{eq:mw-toledo} yields the sharper estimate
\(\left|\operatorname{sign}(\phi)\right|
 \leq2\min\{p,q\}|\chi(\Sigma)|\).

\section{Previously known and elementary cases}
\label{sec:special-signature-values}

We begin with \(\mathrm{U}(1,1)\), where all required signatures can be
realized in the subgroup
\(\mathrm{SU}(1,1)\cong\mathrm{SL}(2,\mathbb R)\). We give the rank-one
construction in detail because it will also supply the higher-rank
examples. We then give short proofs for genus-zero surfaces and compact
unitary groups. Finally, we treat closed surfaces, where the absence of a
boundary term forces the signature to be divisible by \(4\).

Let \(\Sigma_{g,n}\) be a compact, connected, oriented surface of genus \(g\)
with \(n\geq 0\) boundary components, and write
\(\chi(\Sigma_{g,n})=2-2g-n\). When \(n\geq1\), we orient each boundary
component by the orientation induced from \(\Sigma_{g,n}\), and use the
standard presentation
\begin{equation}
\label{eq:standard-surface-presentation}
\pi_1(\Sigma_{g,n})
=
\left\langle
a_1,b_1,\ldots,a_g,b_g,c_1,\ldots,c_n
\;\middle|\;
\prod_{j=1}^{g}[a_j,b_j]\prod_{i=1}^{n}c_i=1
\right\rangle.
\end{equation}
We use \([A,B]=ABA^{-1}B^{-1}\). Each \(c_i\) follows the induced
boundary orientation. For \(n=0\), omit the boundary generators.
Throughout this section, all intervals of signature values contain only
integers.

\subsection{The rank-one case}
\label{subsec:U11}

We choose the isomorphism
\(\mathrm{SU}(1,1)\cong\mathrm{SL}(2,\mathbb R)\) so that the unitary
signature agrees with the symplectic signature used in \cite{KPWValues}.
The rho invariant then agrees as well, whereas the two Toledo conventions
have opposite signs: \cite{KPWValues} writes the signature formula as
\(\operatorname{sign}=2T+\rho\). This convention allows us to use its
rank-one constructions directly.
A representation
\(\phi\colon\pi_1(\Sigma_{g,n})\to\mathrm{SL}(2,\mathbb{R})\) is called
\emph{boundary paraelliptic} if
\(\operatorname{tr}(\phi(c_i))\in[-2,2]\) for every \(i\). Thus each boundary
holonomy is elliptic, parabolic, or central.

\begin{theorem}[{\cite[Theorem~3.3]{KPWValues}}]
\label{thm:rank-one-paraelliptic}
Let \(\Sigma=\Sigma_{g,n}\) satisfy \(n\geq 1\) and
\(\chi(\Sigma)\leq 0\). Then every integer in
\(
[-2|\chi(\Sigma)|,\,2|\chi(\Sigma)|]
\)
is the signature of a boundary paraelliptic representation
\(\pi_1(\Sigma)\to\mathrm{SL}(2,\mathbb{R})\).
\end{theorem}

\begin{proof}
The Milnor--Wood inequality \eqref{eq:mw-signature} gives
\(
\bigl|\operatorname{sign}(\phi)\bigr|
\leq 2|\chi(\Sigma)|
\)
for every representation
\(\phi\colon\pi_1(\Sigma)\to\mathrm{SL}(2,\mathbb R)\). It remains to
realize every integer in this interval. Put \(N:=|\chi(\Sigma)|\). If
\(N=0\), then \(\Sigma\) is an annulus and the trivial representation realizes
the only possible value, namely \(0\). We therefore assume \(N\geq1\).

\smallskip
\noindent\emph{Step 1: even signatures.}
By \cite[Theorem~4]{KPWValues}, for every integer \(a\) with
\(0\leq a\leq N\), there exists a boundary elliptic representation of
signature
\(
2N-4a.
\)
Thus all even integers congruent to \(2N\) modulo \(4\) are realized.

Consider an even integer in the other congruence class. It can be written as
\(m=2N-2b\), where \(b\in\{1,3,\ldots,2N-1\}\). Suppose first that
\(n\geq2\). Set \(N':=N-1=|\chi(\Sigma_{g,n-1})|\) and
\(a:=(b-1)/2\). Then \(m=2N'-4a\), with \(0\leq a\leq N'\). If
\(N'>0\), the boundary elliptic realization theorem provides a representation
of \(\pi_1(\Sigma_{g,n-1})\) with signature \(m\); when \(N'=0\), the same
conclusion follows from an \(\mathrm{SO}(2)\)-representation of the annulus.
Removing an additional open disc introduces a boundary component with trivial
holonomy and leaves the signature unchanged. This yields a boundary
paraelliptic representation of \(\pi_1(\Sigma_{g,n})\) with signature \(m\).

Now suppose that \(n=1\). Then \(N=2g-1\), and the even values not already
covered are precisely the multiples \(m=4e\) with \(|e|\leq g-1\). By
Goldman's theorem \cite[Theorem~B]{Goldman}, there is a representation of the
closed surface group into \(\mathrm{PSL}(2,\mathbb R)\) with Euler number
\(2e\). Since this Euler number is even, the representation lifts to
\(\mathrm{SL}(2,\mathbb R)\). On a closed surface the rho term is absent, so
its signature is \(4e=m\). Removing a disc produces the required
representation of \(\pi_1(\Sigma_{g,1})\) with trivial boundary holonomy.
Hence every even integer in the Milnor--Wood interval is realized.

\smallskip
\noindent\emph{Step 2: odd signatures when \(n\geq2\).}
By \cite[Subsection~3.1]{KPWValues}, there are pair-of-pants representations
\(\eta_+\) and \(\eta_-\) with signatures \(1\) and \(-1\), respectively,
such that two boundary holonomies are elliptic, one with angle in $(0,\pi)$ and the other with angle in $(\pi,2\pi)$, and the third is parabolic.
Every odd integer in \([-2N,2N]\) has one of the forms
\[
m=2N-4a-1 \quad (0\leq a\leq N-1)
\text{ or }
m=2N-4a+1 \quad (1\leq a\leq N).
\]
In the first case, choose on \(\Sigma_{g,n-1}\) a boundary elliptic
representation of signature \(2(N-1)-4a\) and glue it to \(\eta_+\). In the
second case, choose one of signature \(2(N-1)-4(a-1)\) and glue it to
\(\eta_-\). The elliptic conjugacy classes may be adjusted so that the
holonomies along the gluing curves are inverse to one another. Additivity of the signature under gluing
then gives the prescribed value \(m\).

\smallskip
\noindent\emph{Step 3: odd signatures when \(n=1\).}
We argue by induction on \(g\). For \(g=1\), punctured-torus representations
with parabolic boundary holonomy and signatures \(\pm1\) are constructed in
\cite[Subsection~6.2]{KPWValues}. Assume now that \(g\geq2\) and that all odd
integers in
\(
[-2(2g-3),\,2(2g-3)]
\)
are realized on \(\Sigma_{g-1,1}\). Sending the generators of an additional
handle to the identity extends each of these representations to
\(\Sigma_{g,1}\) without changing the signature. Since
\(|\chi(\Sigma_{g,1})|=2g-1\), the only remaining positive odd values are
\(4g-5\) and \(4g-3\).

By \cite[Lemma~4.13]{KWComponents}, for each
\(\varepsilon\in\{-1,1\}\) there is a pair-of-pants representation
\(\eta_\varepsilon\) with two hyperbolic boundary holonomies of negative
trace, one parabolic boundary holonomy, and signature \(\varepsilon\). Choose
maximal boundary hyperbolic representations on \(\Sigma_{1,1}\) and
\(\Sigma_{g-1,1}\), of signatures \(2\) and \(2(2g-3)\), respectively.
Their boundary holonomies have negative trace by
\cite[Lemma~3.7]{KWComponents}, and their eigenvalues may be chosen
to match those of the two hyperbolic boundary components of
\(\eta_\varepsilon\). Gluing the three pieces gives a boundary parabolic
representation of \(\Sigma_{g,1}\) with signature
\(
2+\varepsilon+2(2g-3)=4g-4+\varepsilon.
\)
Thus \(4g-5\) and \(4g-3\) are realized. Composing with the standard
sign-reversing automorphism of \(\mathrm{SL}(2,\mathbb R)\) realizes their
negatives; see \cite[Subsection~1.4]{KPWValues}. This completes the proof.
\end{proof}

\begin{remark}
\label{rem:rank-one-proof-correction}
The gluing argument in \cite[Section~3.2]{KPWValues} starts with a
genus-zero representation and attaches \(g\) elliptic punctured-torus
representations. To justify such a gluing, one must control the elliptic
conjugacy class on each attaching boundary; the existence of some
elliptic holonomy does not suffice. In particular, that argument does not
by itself establish the extremal values. The proof above obtains these
values from the boundary elliptic realization theorem and supplies the
boundary matching needed for the odd values.
\end{remark}

The preceding theorem immediately gives the corresponding result for
\(\mathrm{U}(1,1)\).

\begin{corollary}
\label{cor:signature-values-U11}
Let \(\Sigma=\Sigma_{g,n}\) satisfy \(n\geq1\) and
\(\chi(\Sigma)\leq0\). Then
\begin{equation}
\label{eq:U11-complete-range}
\left\{
\operatorname{sign}(\phi)
\;\middle|\;
\phi\in
\operatorname{Hom}
\bigl(\pi_1(\Sigma),\mathrm{U}(1,1)\bigr)
\right\}
=
[-2|\chi(\Sigma)|,\,2|\chi(\Sigma)|]\cap\mathbb{Z}.
\end{equation}
Moreover, every value in this interval is realized by a representation whose
image is contained in \(\mathrm{SU}(1,1)\) and which corresponds, under
\(\mathrm{SU}(1,1)\cong\mathrm{SL}(2,\mathbb{R})\), to a boundary
paraelliptic representation.
\end{corollary}

Taking orthogonal direct sums of \(a\) such representations realizes every
integer in
\([-2a|\chi(\Sigma)|,2a|\chi(\Sigma)|]\). Together with
\eqref{eq:mw-signature}, this gives the complete range for
\(\mathrm{U}(a,a)\).

\subsection{Genus-zero surfaces}
\label{subsec:genus-zero-unitary}

We now assume that \(\Sigma=\Sigma_{0,n}\), where \(n\geq2\), so that
\(|\chi(\Sigma)|=n-2\).
The following elementary observation will also be useful for compact
unitary groups. Define
\[
\varepsilon(\theta)
:=
\begin{cases}
0, & \theta=0,\\[1mm]
1-\tfrac{\theta}{\pi}, & 0<\theta<2\pi.
\end{cases}
\]

\begin{lemma}
\label{lem:angle-realization}
Let \(N\geq2\) and let \(s\) be an integer satisfying
\(|s|\leq N-2\). Then there exist angles
\(\theta_1,\ldots,\theta_N\in[0,2\pi)\) such that
\begin{equation}
\label{eq:angle-realization}
\sum_{\nu=1}^{N}\theta_\nu\in2\pi\mathbb{Z},
\qquad
\sum_{\nu=1}^{N}\varepsilon(\theta_\nu)=s.
\end{equation}
\end{lemma}

\begin{proof}
Set \(r:=|s|+2\) and \(k:=(r-s)/2\). Thus \(k=1\) when
\(s\geq0\), and \(k=1-s\) when \(s<0\).
Then \(2\leq r\leq N\) and \(1\leq k<r\). Set
\(\theta_1=\cdots=\theta_r=2\pi k/r\) and
\(\theta_{r+1}=\cdots=\theta_N=0\). It follows that
\(
\sum_{\nu=1}^{N}\theta_\nu=2\pi k\) and \(
\sum_{\nu=1}^{N}\varepsilon(\theta_\nu)
=
r-2k
=
s.
\)
\end{proof}

\begin{proposition}[{\cite[Proposition~10.1]{KPWValues}}]
\label{prop:genus-zero-unitary}
Let \(p,q\geq0\) with \(p+q>0\). Then every integer in
\begin{equation}
\label{eq:genus-zero-range}
[-(p+q)(n-2),\,(p+q)(n-2)]
\end{equation}
is the signature of a representation
\(\pi_1(\Sigma_{0,n})\to\mathrm{U}(p,q)\). Moreover, the representation may
be chosen with image contained in
\[
\mathrm{U}(1)^p\times\mathrm{U}(1)^q
\subset
\mathrm{U}(p)\times\mathrm{U}(q)
\subset
\mathrm{U}(p,q).
\]
\end{proposition}

\begin{proof}
The Milnor--Wood inequality gives
\(
|\operatorname{sign}(\phi)|
\leq
(p+q)(n-2)
\)
for every
\(\phi\colon\pi_1(\Sigma_{0,n})\to\mathrm{U}(p,q)\). It remains to prove
realizability.

Consider diagonal boundary holonomies
\begin{equation}
\label{eq:diagonal-boundary-genus-zero}
\phi(c_i)
=
\operatorname{diag}
\bigl(
e^{i\theta_{i,1}},\ldots,e^{i\theta_{i,p+q}}
\bigr),
\qquad
1\leq i\leq n,
\end{equation}
where \(\theta_{i,j}\in[0,2\pi)\). These matrices define a representation of
\(\pi_1(\Sigma_{0,n})\) precisely when
\(
\sum_{i=1}^{n}\theta_{i,j}
\in
2\pi\mathbb{Z}
\), \(
1\leq j\leq p+q.
\)
For such a representation, the Toledo invariant vanishes and the signature
is
\begin{equation}
\label{eq:diagonal-signature-formula}
\operatorname{sign}(\phi)
=
\sum_{j=1}^{p}\sum_{i=1}^{n}\varepsilon(\theta_{i,j})
-
\sum_{j=p+1}^{p+q}\sum_{i=1}^{n}\varepsilon(\theta_{i,j});
\end{equation}
see \cite[Section~10.1]{KPWValues}.

Let \(m\) be an integer in the interval
\eqref{eq:genus-zero-range}, and set \(M:=n-2\). Since
\([-(p+q)M,(p+q)M]\) is the sum of \(p+q\) copies of \([-M,M]\), there
exist integers \(t_1,\ldots,t_{p+q}\in[-M,M]\) such that
\(m=t_1+\cdots+t_{p+q}\). Define
\[
s_j:=
\begin{cases}
t_j, & 1\leq j\leq p,\\
-t_j, & p+1\leq j\leq p+q.
\end{cases}
\]
For each \(j\), Lemma~\ref{lem:angle-realization}, applied with \(N=n\),
provides angles \(\theta_{1,j},\ldots,\theta_{n,j}\) satisfying
\(
\sum_{i=1}^{n}\theta_{i,j}\in2\pi\mathbb{Z},
\) and \(
\sum_{i=1}^{n}\varepsilon(\theta_{i,j})=s_j.
\)
Thus \eqref{eq:diagonal-boundary-genus-zero} defines a representation, and
\eqref{eq:diagonal-signature-formula} yields
\[
\operatorname{sign}(\phi)
=
\sum_{j=1}^{p}s_j-\sum_{j=p+1}^{p+q}s_j
=
\sum_{j=1}^{p+q}t_j
=
m.
\]
\end{proof}

\subsection{Compact unitary groups}
\label{subsec:compact-unitary}

We next consider representations into the compact unitary group
\(\mathrm{U}(p)\). Proposition~\ref{prop:genus-zero-unitary}, with \(q=0\),
already gives the complete answer in genus zero. We therefore assume
\(g\geq1\).

\begin{proposition}[{\cite[Proposition~10.3]{KPWValues}}]
\label{prop:positive-genus-Up}
Let \(\Sigma=\Sigma_{g,n}\) with \(g\geq1\) and \(n\geq1\), and let
\(p\geq1\). Then every representation
\(\phi\colon\pi_1(\Sigma)\to\mathrm{U}(p)\) satisfies
\begin{equation}
\label{eq:Up-signature-bound}
|\operatorname{sign}(\phi)|
\leq
\max\{0,np-2\}.
\end{equation}
Conversely, if \(np\geq2\), then every integer in
\([2-np,np-2]\) is realized as the signature of a representation into
\(\mathrm{U}(p)\).
\end{proposition}

\begin{proof}
The Toledo invariant vanishes for representations into \(\mathrm{U}(p)\). Write the eigenvalues of the \(i\)-th
boundary holonomy as
\(
\phi(c_i)
\sim
\operatorname{diag}
\bigl(
e^{i\theta_{i,1}},\ldots,e^{i\theta_{i,p}}
\bigr),
\) \(
\theta_{i,j}\in[0,2\pi).
\)
The boundary rho formula gives
\begin{equation}
\label{eq:Up-boundary-signature}
\operatorname{sign}(\phi)
=
\sum_{i=1}^{n}\sum_{j=1}^{p}
\varepsilon(\theta_{i,j});
\end{equation}
see \cite[Section~10.2]{KPWValues}.

Applying the determinant to the relation
\eqref{eq:standard-surface-presentation}, we obtain
\(
\prod_{i=1}^{n}\det\phi(c_i)=1.
\)
Consequently,
\(
\sum_{i=1}^{n}\sum_{j=1}^{p}\theta_{i,j}
=
2\pi k
\)
for some integer \(k\geq0\). Let \(r\) denote the number of nonzero angles
among the \(\theta_{i,j}\). Equation \eqref{eq:Up-boundary-signature} then
becomes
\begin{equation}
\label{eq:Up-r-minus-2k}
\operatorname{sign}(\phi)=r-2k.
\end{equation}

If \(r=0\), then \(k=0\), and the signature is \(0\). Otherwise, all
\(r\) nonzero angles lie strictly between \(0\) and \(2\pi\), so
\(0<k<r\). Since \(k\) is an integer, \(r\geq2\) and
\[
2-r\leq r-2k\leq r-2.
\]
Together with \(r\leq np\), this proves
\eqref{eq:Up-signature-bound}. In particular, when \(np=1\), the
nonzero-angle case is impossible, and every signature is \(0\).

We now prove realizability. Assume \(np\geq2\), and let
\(m\in[2-np,np-2]\). By Lemma~\ref{lem:angle-realization}, applied with
\(N=np\), there exist \(np\) angles, arranged as
\(\{\theta_{i,j}\}_{1\leq i\leq n,\,1\leq j\leq p}\), such that
\begin{equation}
\label{eq:Up-angle-choice}
\sum_{i=1}^{n}\sum_{j=1}^{p}\theta_{i,j}
\in2\pi\mathbb{Z},
\qquad
\sum_{i=1}^{n}\sum_{j=1}^{p}
\varepsilon(\theta_{i,j})
=
m.
\end{equation}
Set
\(
C_i
:=
\operatorname{diag}
\bigl(
e^{i\theta_{i,1}},\ldots,e^{i\theta_{i,p}}
\bigr),
\) \(
D:=C_1\cdots C_n.
\)
The first condition in \eqref{eq:Up-angle-choice} implies that
\(D\in\mathrm{SU}(p)\).

If \(p=1\), then \(D=1\), and we take \(A=B=1\). Suppose \(p\geq2\), and
write \(D=\operatorname{diag}(d_1,\ldots,d_p)\). Let
\(A\in\mathrm{U}(p)\) be the cyclic permutation matrix defined by
\(
Ae_j=e_{j+1}
\quad (1\leq j<p),
\) \(
Ae_p=e_1,
\)
and set
\(
B
=
\operatorname{diag}
\bigl(
1,d_2,d_2d_3,\ldots,d_2d_3\cdots d_p
\bigr).
\)
Since \(d_1\cdots d_p=1\), a direct computation gives
\begin{equation}
\label{eq:unitary-commutator}
[A,B]
=
D^{-1}.
\end{equation}

We define a representation on the generators in
\eqref{eq:standard-surface-presentation} by
\(
\phi(c_i)=C_i,
\) \(
\phi(a_1)=A,
\) \(
\phi(b_1)=B,
\)
and by
\(\phi(a_j)=\phi(b_j)=I_p\) for \(2\leq j\leq g\). Equation
\eqref{eq:unitary-commutator} shows that the defining surface relation is
satisfied. Finally, \eqref{eq:Up-boundary-signature} and
\eqref{eq:Up-angle-choice} give
\(\operatorname{sign}(\phi)=m\).
\end{proof}

On a closed surface, the signature vanishes because both the Toledo
invariant and the boundary term vanish. It also vanishes on a disc,
whose fundamental group is trivial. Combining these observations with
the preceding propositions gives the complete range for compact unitary
groups.

\begin{theorem}[{\cite[Theorem~2]{KPWValues}}]
\label{thm:complete-Up-range}
Let \(\Sigma=\Sigma_{g,n}\) be a compact, connected, oriented surface, and
let \(p\geq1\). The
possible values of the signature of representations
\(\pi_1(\Sigma)\to\mathrm{U}(p)\) are
\begin{equation}
\label{eq:complete-Up-range}
\begin{cases}
\{0\}, & n=0\text{ or }(g,n)=(0,1),\\[1mm]
[-p(n-2),\,p(n-2)]\cap\mathbb{Z},
& g=0,\ n\geq2,\\[1mm]
\bigl([2-np,\,np-2]\cap\mathbb{Z}\bigr)\cup\{0\},
& g\geq1,\ n\geq1.
\end{cases}
\end{equation}
\end{theorem}

\subsection{Closed surfaces}
\label{subsec:closed-surfaces}

We conclude this section with the closed-surface case. Here the boundary rho
term is absent, and the signature satisfies a divisibility condition that does
not occur for bordered surfaces.

Let \(\Sigma=\Sigma_g\) be a closed oriented surface of genus \(g\geq1\),
and assume \(0\leq p\leq q\) and \(p+q>0\). For a representation
\(\phi\colon\pi_1(\Sigma)\to\mathrm{U}(p,q)\), choose a compatible complex
structure and write \(\mathcal E_\phi=\mathcal E^+\oplus\mathcal E^-\) for
the corresponding positive and negative subbundles. By
\cite[(3.15)]{KPWUnitary},
\begin{equation}
\label{eq:closed-toledo-chern-weil}
T(\Sigma,\phi)
=
\int_\Sigma\bigl(c_1(\mathcal E^-)-c_1(\mathcal E^+)\bigr)
=
-2\int_\Sigma c_1(\mathcal E^+)
\in 2\mathbb Z.
\end{equation}
Indeed, flatness makes the real first Chern class of
\(\mathcal E_\phi\) vanish. Since \(H^2(\Sigma;\mathbb Z)\) is torsion-free,
\(c_1(\mathcal E^+)+c_1(\mathcal E^-)=0\) also holds integrally. Since
\(\partial\Sigma=\varnothing\), the signature--Toledo formula reduces to
\begin{equation}
\label{eq:closed-sign-divisible-by-four}
\operatorname{sign}(\phi)=-2T(\Sigma,\phi)\in 4\mathbb Z.
\end{equation}

Combining this with the closed Milnor--Wood inequality gives
\[
\mathscr S_{p,q}(\Sigma_g)
\subset
[2p\chi(\Sigma_g),\,2p|\chi(\Sigma_g)|]\cap 4\mathbb Z.
\]
The inclusion is sharp.

\begin{proposition}
\label{prop:closed-surface-signature-values}
Let \(\Sigma_g\) be a closed oriented surface of genus \(g\geq1\), and let
\(0\leq p\leq q\). Then
\begin{equation}
\label{eq:closed-signature-range-Upq}
\mathscr S_{p,q}(\Sigma_g)
=
[2p\chi(\Sigma_g),\,2p|\chi(\Sigma_g)|]\cap 4\mathbb Z.
\end{equation}
\end{proposition}

\begin{proof}
The inclusion ``\(\subset\)'' follows from
\eqref{eq:closed-sign-divisible-by-four} and the Milnor--Wood inequality. If
\(p=0\) or \(g=1\), both sides of
\eqref{eq:closed-signature-range-Upq} are \(\{0\}\), so there is nothing to
prove. Assume henceforth that \(p\geq1\) and \(g\geq2\).

Let
\(m\in[2p\chi(\Sigma_g),2p|\chi(\Sigma_g)|]\cap 4\mathbb Z\), and put
\(T_0:=-m/2\). Then \(T_0\in2\mathbb Z\) and
\(|T_0|\leq p|\chi(\Sigma_g)|\). We may write
\(T_0=t_1+\cdots+t_p\), where each
\(
t_j\in[-|\chi(\Sigma_g)|,|\chi(\Sigma_g)|]\cap2\mathbb Z.
\)
Goldman's theorem and the lifting criterion from
\(\mathrm{PSL}(2,\mathbb R)\) to \(\mathrm{SL}(2,\mathbb R)\) imply that
each \(t_j\) occurs as the Toledo invariant of a representation
\(
\psi_j\colon\pi_1(\Sigma_g)
\to
\mathrm{SU}(1,1)\cong\mathrm{SL}(2,\mathbb R);
\)
see \cite[Theorem~B and p.~604]{Goldman}. The block-diagonal representation
\[
\phi:=\psi_1\oplus\cdots\oplus\psi_p\oplus\mathbf1_{q-p}
\colon\pi_1(\Sigma_g)\longrightarrow\mathrm{U}(p,q)
\]
has Toledo invariant \(T_0\). Therefore
\(\operatorname{sign}(\phi)=-2T_0=m\), proving the reverse inclusion.
\end{proof}

\begin{remark}
\label{rem:closed-lifting-parity}
In rank one, the divisibility above is exactly Goldman's lifting obstruction:
a representation into
\(\mathrm{PU}(1,1)\cong\mathrm{PSL}(2,\mathbb R)\) lifts to
\(\mathrm{SU}(1,1)\cong\mathrm{SL}(2,\mathbb R)\) if and only if its Euler
class is even \cite[p.~604]{Goldman}. On a closed surface, the Euler number agrees with the rank-one Toledo
invariant up to the orientation convention. This sign has no effect on
parity, and the lifting criterion explains
\eqref{eq:closed-toledo-chern-weil}.
\end{remark}
\begin{remark}
\label{rem:closed-higgs-MW}
The closed-surface Milnor--Wood inequality also admits a direct
Higgs-bundle proof.  For a reductive representation, the associated
polystable \(\mathrm{U}(p,q)\)-Higgs bundle yields the bound by applying
semistability to the kernel and saturated image of a Higgs-field component;
see
\cite[Section~3.4, especially Lemma~3.24, Corollary~3.27,
and Remark~3.29]{BGG03}.  The corresponding non-emptiness results
\cite[Theorem~6.1 and Remark~6.2]{BGG03} also show that the bound is
sharp.
\end{remark}

\section{Positive-genus case}
\label{sec:positive-genus-signatures}

We now consider surfaces with positive genus and nonempty boundary. Replacing
the coefficient form \(\Omega\) by \(-\Omega\) regards the same holonomy as
a representation into \(\mathrm{U}(q,p)\) and multiplies the twisted
Hermitian intersection form by \(-1\). Consequently,
\(
\mathscr S_{p,q}(\Sigma)=-\mathscr S_{q,p}(\Sigma).
\)
Since the asserted ranges are symmetric, it is enough to treat \(q\geq p\).
The balanced and compact cases were settled in Section~\ref{sec:special-signature-values}, so it remains to prove
the sharp range for
\[
g\geq1,
\qquad
n\geq1,
\qquad
q>p>0.
\]
Let \(\Sigma=\Sigma_{g,n}\), with
\(\partial\Sigma=c_1\sqcup\cdots\sqcup c_n\), and recall
\(M_{g,n}(p,q)\) from \eqref{eq:introduction-M}.

\begin{theorem}
\label{thm:positive-genus-realization-bounds}
For \(g,n\geq1\) and \(q>p>0\),
\begin{equation}
\label{eq:positive-genus-exact-range}
\mathscr S_{p,q}(\Sigma_{g,n})
=
[-M_{g,n}(p,q),M_{g,n}(p,q)]\cap\mathbb Z.
\end{equation}
Every value in this interval is realized by a representation with image in
\[
\mathrm{U}(1,1)^{\times p}\times\mathrm{U}(q-p)
\subseteq
\mathrm{U}(p,q).
\]
\end{theorem}

The proof has three stages. We first establish the sharp upper bound for
reductive representations using the \(L^2\) Dolbeault
complex of the associated parabolic Higgs bundle. We then extend the estimate
to arbitrary representations using a controlled degeneration to a Levi factor.
Finally, we realize every integer in the asserted interval by orthogonal direct
sums.

\section{Reductive representations}
\label{sec:reductive}

The proof of the sharp upper bound uses the parabolic Higgs bundle associated
with a reductive representation. We first review the closed-surface argument
of \cite[Section~3.4]{BGG03}. We then express the signature as the Euler
characteristic of a two-term complex. For a surface with boundary, the same
method introduces explicit local corrections at the punctures. The purpose
of this section is to keep track of these corrections and of the degrees
of the kernel and image of the Higgs field.

\subsection{The closed-surface Milnor--Wood argument} 
\label{subsec:closed-surface-proof}

Let $\Sigma$ be a closed surface of genus $g\geq1$, and assume
$0<p\leq q$. A representation
$\phi\colon\pi_1(\Sigma)\to\mathrm U(p,q)$ determines a flat vector bundle
$E$ with a parallel Hermitian form of signature $(p,q)$. A smooth
$\phi$-equivariant map
$\widetilde\Sigma\to\mathrm U(p,q)/(\mathrm U(p)\times\mathrm U(q))$
is equivalent to a smooth rank-$p$ subbundle $V\subset E$ on which this
form is positive definite. Its orthogonal complement $W:=V^\perp$ is
negative definite, and $E=V\oplus W$. 
In the Higgs-bundle convention of
\cite[Definition~3.28 and Remark~3.29]{BGG03}, the Toledo invariant is
\[
\tau
=
2\frac{q\deg(V)-p\deg(W)}{p+q}.
\]
Since the total flat bundle satisfies
$\deg(V)+\deg(W)=0$, this becomes
\(
\tau=2\deg(V).
\)
Our bounded-cohomological normalization has the opposite sign; indeed,
\cite[(3.15)]{KPWUnitary} gives
\[
T(\Sigma,\phi)
=
\int_\Sigma\bigl(c_1(W)-c_1(V)\bigr)
=
-2\deg(V).
\]
Thus
\begin{align} \label{eqn:toledo-closed}
T(\phi):=T(\Sigma,\phi)=-2\mathrm{deg}(V).
\end{align}

Fix a Riemann surface structure on $\Sigma$. If $\phi$ is reductive, the
corresponding harmonic reduction gives a \emph{Higgs bundle structure}:

\begin{itemize}
  \item holomorphic structures on $V$ and $W$;
  \item a holomorphic Higgs field
$\Phi\in H^0(\Sigma,\operatorname{End}(E)\otimes K)$ that exchanges the two
summands,
  \begin{align*}
\Phi=\begin{pmatrix}
 0     &  \beta  \\
 \gamma     & 0 
\end{pmatrix},
\end{align*}
\end{itemize}
where $K$ is the canonical bundle of $\Sigma$,
$\beta\colon W\to V\otimes K$, and $\gamma\colon V\to W\otimes K$.
The associated Higgs bundle is polystable of degree zero. We use its
\emph{semistability}: every $\Phi$-invariant coherent subsheaf
$S\subset E$, that is, every subsheaf satisfying
$\Phi(S)\subset S\otimes K$, obeys
\begin{align*}
\mathrm{deg}(S)\le 0.
\end{align*}
 
 If $\gamma=0$, $V$ itself is a $\Phi$-invariant subsheaf, hence $\mathrm{deg}(V)\le 0$. 
 
Otherwise, let $N:=\ker\gamma$ and let $I\subset W$ be the saturation
of $\operatorname{im}(\gamma)\otimes K^{-1}$. Thus $I$ is the smallest
subbundle of $W$ containing the image. Both $N$ and $V\oplus I$ are
$\Phi$-invariant: $\gamma$ vanishes on $N$, $\gamma(V)\subset I\otimes K$,
and $\beta(I)\subset V\otimes K$. Semistability therefore gives
\begin{align} \label{eqn:stability}
\mathrm{deg}(N)\le 0,\quad \mathrm{deg}(I)\le -\mathrm{deg}(V).
\end{align}

The induced morphism $V/N\to I\otimes K$ is generically an isomorphism.
Taking its determinant gives a nonzero holomorphic section of
\begin{align*}
\mathrm{det}((V/N)^*)\otimes \mathrm{det}(I \otimes K)=\mathrm{det}((V/N)^*)\otimes \mathrm{det}(I) \otimes K^{\mathrm{rank}(I)}.
\end{align*}
Therefore this line bundle has nonnegative degree. In other words,
\begin{align}\label{eqn:isomorphism}
-\mathrm{deg}(V)+\mathrm{deg}(N)+\mathrm{deg}(I)+\mathrm{rank}(I)\mathrm{deg}(K)\ge 0.
\end{align}
Combining \eqref{eqn:stability} and \eqref{eqn:isomorphism}, and using
$\operatorname{rk}I\leq p$ and $\deg K=2g-2\geq0$, yields
\begin{align*}
2\mathrm{deg}(V)&\le \mathrm{rank}(I)\mathrm{deg}(K)\le p|\chi(\Sigma)|.
\end{align*}
In both cases, one concludes that
\begin{align*}
T(\phi)\ge p\chi(\Sigma).
\end{align*}
Complex conjugation of the representation reverses the Toledo invariant.
Applying the same argument to the conjugate representation gives
\begin{align*}
T(\phi)\le p|\chi(\Sigma)|.
\end{align*}

\subsection{A cohomological formula for the signature} 
\label{subsec:signature-closed-surface-proof}

For a closed surface, the signature equals minus twice the Toledo
invariant. The preceding argument therefore already gives its sharp bound.
On a punctured surface, however, the signature also contains a boundary
term. Although the kernel--image argument extends to parabolic Higgs
bundles \cite[Proposition~4.2]{GPLM}, a bound for the Toledo invariant alone
does not directly control that boundary term.

We instead express the signature in terms of Higgs cohomology. We explain
this first for a closed surface. The harmonic-bundle K\"ahler identities
identify flat harmonic forms with harmonic forms for
$D^0:=\bar\partial+\Phi$. Because $\Phi$ is off diagonal, its Dolbeault
complex splits into the two summands associated with
\[
\mathcal C_\gamma^\bullet=[V\xrightarrow{\gamma}W\otimes K],
\qquad
\mathcal C_\beta^\bullet=[W\xrightarrow{\beta}V\otimes K],
\]
placed in degrees zero and one. On harmonic one-forms, the Hermitian
intersection form is negative definite on the $\gamma$-summand and positive
definite on the $\beta$-summand. Thus its negative index is
$\dim\mathbb H^1(\mathcal C_\gamma^\bullet)$. Poincar\'e duality pairs
$\mathbb H^2(\mathcal C_\gamma^\bullet)$ with
$\mathbb H^0(\mathcal C_\beta^\bullet)$, so the degree-zero contributions
cancel when the signature is expressed using Euler characteristics. The
result is
\begin{align}\label{eqn:sign-chi}
\operatorname{sign}(\phi)
=-\chi(\Sigma;\mathcal E_\phi)
+2\chi(\mathcal C_\gamma^\bullet).
\end{align}
Here $\mathcal E_\phi$ is the flat local system,
$\chi(\Sigma;\mathcal E_\phi)=(p+q)\chi(\Sigma)$, and
$\chi(\mathcal C_\gamma^\bullet)=\chi(V)-\chi(W\otimes K)$.
The punctured-surface version, including the cancellation in degrees zero
and two, is proved in Lemma~\ref{lem:L2-Hodge-gamma} and
Corollary~\ref{eqn:sign-versus-chi-of-C-gamma} below.

Riemann--Roch gives
$\chi(V)=p(1-g)+\deg V$ and
$\chi(W\otimes K)=q(g-1)+\deg W$. Consequently,
\begin{align*}
\chi(\mathcal{C}^\bullet_\gamma)=(p+q)(1-g)+\deg V-\deg W.
\end{align*}
We now use the kernel and saturated image from
Subsection~\ref{subsec:closed-surface-proof}. Put
$I^0:=\operatorname{im}(\gamma)\otimes K^{-1}$,
$k:=\operatorname{rk}I$, and $Q:=W/I$. The difference between $I^0$ and
its saturation $I$ is measured by the torsion sheaf $I/I^0$; write
$\ell(I/I^0)\geq0$ for its total length. The exact sequences defining
$N$, $I^0$, and $Q$ give
\begin{align*}
\deg V=\deg N+\deg I-k\chi(\Sigma)-\ell(I/I^0),\quad \deg W=\deg I+\deg Q.
\end{align*}
Semistability gives $\deg N\leq0$. It also gives
$\deg(V\oplus I)\leq0$, and therefore
$\deg Q=-\deg(V\oplus I)\geq0$. Hence
\begin{align*}
\deg V-\deg W\le -k\chi(\Sigma).
\end{align*}
It follows that
\begin{align*}
\chi(\mathcal{C}^\bullet_\gamma)\le(p+q)(1-g)-k\chi(\Sigma)=(p+q-2k)(1-g),
\end{align*}
and \eqref{eqn:sign-chi} yields
\begin{align*}
\operatorname{sign}\le -2k\chi(\Sigma)\le 2p|\chi(\Sigma)|.
\end{align*}

\subsection{From boundary circles to marked points}
\label{subsec:signature-punctured-surface}

We now return to a surface with boundary. To repeat the preceding argument,
we need a Higgs complex that computes the cohomology carrying the
nondegenerate intersection form. Its local conditions at the punctures will
produce the boundary corrections to \eqref{eqn:sign-chi}. 

Let \(\overline\Sigma\) be the closed surface
obtained by capping off the boundary components. Choose marked points
\(x_1,\ldots,x_n\) in the caps, identify the interior of \(\Sigma\) with
\(\overline\Sigma\setminus\{x_1,\ldots,x_n\}\), and put
\(
D:=x_1+\cdots+x_n,
\qquad
\Sigma^\circ:=\overline\Sigma\setminus D.
\)

The nonabelian Hodge correspondence now supplies a parabolic Higgs bundle.
Its flags and weights record the growth of the harmonic metric at the
punctures. We will combine its parabolic degree inequalities with an
explicit coherent complex whose hypercohomology is identified, through
$L^2$-harmonic forms, with intersection cohomology.

\begin{remark}[Boundary and puncture orientations]
\label{rem:boundary-puncture-orientation}
Give \(\overline\Sigma\) the complex orientation agreeing with the given
orientation of \(\Sigma\). Let \(\delta_j\) be a small positively oriented
loop about \(x_j\) in a holomorphic coordinate. Under the identification
\(\Sigma^\circ\simeq\operatorname{int}\Sigma\), the loop \(\delta_j\) is
freely homotopic to \(c_j^{-1}\): the induced orientation on the boundary of
a deleted disc is opposite to its positive complex orientation. We therefore
write
\(
L_j:=\phi(c_j)\) and \(
M_j:=\phi(\delta_j)=L_j^{-1}.
\)
The rho invariant term is always \(\rho(L_j)\), whereas all
parabolic weights, residues, and monodromy weight filtrations below are the
local nonabelian-Hodge data attached to \(M_j\).
\end{remark}

\subsection{Normalized parabolic Higgs bundles}

\label{subsec:reductive-higgs-setup}

We first fix the convention for parabolic structures, following
\cite[Section~2]{GPLM}. A parabolic structure on a holomorphic vector
bundle $F$ over $(\overline\Sigma,D)$ consists, at each marked point
$x_j$, of a flag
\[
F|_{x_j}=F_{j,1}\supsetneq F_{j,2}\supsetneq\cdots
\supsetneq F_{j,\ell_j+1}=0
\]
and weights
$0\leq\alpha_{j,1}<\cdots<\alpha_{j,\ell_j}<1$.
The multiplicity of $\alpha_{j,i}$ is
$m_{j,i}:=\dim_{\mathbb C}(F_{j,i}/F_{j,i+1})$.
We write
\[
\operatorname{Gr}_{\alpha_{j,i},x_j}(F)
:=
F_{j,i}/F_{j,i+1},
\]
and put $\operatorname{Gr}_{\lambda,x_j}(F)=0$ when $\lambda$ is not
a weight of $F$ at $x_j$. Thus the graded pieces are subquotients of
the fiber, not the flag subspaces themselves. The parabolic degree is
\[
\operatorname{pardeg}(F)
=
\deg F+\sum_{j=1}^n\sum_{i=1}^{\ell_j}
\alpha_{j,i}m_{j,i}.
\]

For a holomorphic subbundle $S\subset F$, the induced parabolic
structure is defined by the intersections
$S|_{x_j}\cap F_{j,i}$. The quotient $T=F/S$ carries the quotient
parabolic structure, whose flag is obtained by taking the images
$\pi_j(F_{j,i})$ under $\pi_j\colon F|_{x_j}\to T|_{x_j}$.
In both constructions, we retain the weight $\alpha_{j,i}$ precisely
when the corresponding graded quotient is nonzero.

We denote the positive-weight flag step by
\begin{equation}\label{positive-weight flag}
F_{j,>0}:=
\begin{cases}
F_{j,2},&\alpha_{j,1}=0,\\
F|_{x_j},&\alpha_{j,1}>0.
\end{cases}
\end{equation}
If the only weight is zero, the first case gives $F_{j,>0}=0$.
Then
\[
\operatorname{Gr}_{0,x_j}(F)
=
F|_{x_j}/F_{j,>0},
\qquad
\dim F_{j,>0}
=
\sum_{\lambda\in(0,1)}
\dim\operatorname{Gr}_{\lambda,x_j}(F).
\]
For the induced quotient structure, one has
$\pi_j(F_{j,>0})=T_{j,>0}$.

Let \(G=\mathrm{U}(p,q)\), regarded as a real reductive algebraic group,
and let
\(
H_\phi
:=
\overline{\phi(\pi_1(\Sigma))}^{\,\mathrm{Zar}}
\)
be the real Zariski closure of the image of \(\phi\). We call \(\phi\)
\emph{reductive} if \(H_\phi\) is \(G\)-completely reducible: whenever
\(H_\phi\) is contained in a real parabolic subgroup \(P\subset G\), it is
contained in a Levi subgroup of \(P\). For real reductive groups, this
condition is equivalent to the closed-orbit condition for the conjugation
action; see \cite{BMR,RichardsonSlodowy}.

Throughout the rest of this section, we assume that
\(
\phi\colon\pi_1(\Sigma)\to\mathrm{U}(p,q)
\)
is reductive. The correspondence in \cite{BGPMiR20} is formulated using
meromorphic equivalence classes. For our degree and length calculations,
we need actual vector bundles on $\overline\Sigma$. We therefore first
choose the representative with standard weights in $[0,1)$ and verify that
its Higgs field is logarithmic.

Let
\(
H=\mathrm{U}(p)\times\mathrm{U}(q)
\)
be the standard maximal compact subgroup of \(G\). Its complexification and
the complexified isotropy representation are
\[
H^{\mathbb C}
=
\mathrm{GL}(p,\mathbb C)\times\mathrm{GL}(q,\mathbb C),
\quad
\mathfrak m^{\mathbb C}
=
\operatorname{Hom}(W,V)
\oplus
\operatorname{Hom}(V,W).
\]
Here $V\cong\mathbb C^p$ and $W\cong\mathbb C^q$ denote, respectively, the positive and negative subspaces in the standard orthogonal decomposition $\mathbb C^{p+q}=V\oplus W$ of the Hermitian space of signature $(p,q)$. The two factors of $H^{\mathbb C}$ act on $V$ and $W$ through their standard representations.

Let \(\mathbf C=(C_1,\ldots,C_n)\), where \(C_j\) is the conjugacy class of
the positively oriented puncture monodromy
\(M_j=\phi(c_j)^{-1}\). In the notation of \cite{BGPMiR20}, a
representation with prescribed puncture-monodromy conjugacy
classes belongs to
\(
\mathcal R(\mathbf C)
=
\mathcal S(0,\mathbf C);
\)
see \cite[Proposition~7.9]{BGPMiR20}. The local correspondence between
monodromy data and Higgs-bundle data is described in
\cite[Table~1 and Proposition~6.8]{BGPMiR20}, while the global
correspondence follows from
\cite[Proposition~6.3, Theorem~6.5, and Theorem~7.10]{BGPMiR20}.

The \emph{group-theoretic lattice} is the local lattice supplied by the
parabolic nonabelian-Hodge correspondence of \cite{BGPMiR20}, with
weights $\alpha_j^{\mathrm B}$ normalized so that the compact elliptic
part of the positively oriented puncture monodromy is
$\exp(2\pi\sqrt{-1}\alpha_j^{\mathrm B})$. The
\emph{standard lattice} is obtained from this lattice by the integral
Hecke transformation which replaces each group-theoretic weight
$\theta\in[0,1)$ by the usual vector-bundle weight
$\{-\theta\}\in[0,1)$. This transformation preserves parabolic degrees. The adapted harmonic
metric also gives polystability in the standard lattice. The ordinary
residue and the saturation defects, however, may change when the lattice
changes.

We shall use the following consequence in a form adapted to
\(\mathrm{U}(p,q)\). 

\begin{proposition}
\label{prop:normalized-logarithmic-representative}
Let
\(
\phi\colon\pi_1(\Sigma)\to\mathrm{U}(p,q)
\)
be reductive, with arbitrary boundary holonomy. Then the parabolic
nonabelian-Hodge correspondent of \(\phi\) has a distinguished
standard-vector-bundle representative
\[
(E=V\oplus W,\Phi)
\]
on \((\overline\Sigma,D)\) with the following properties.

\begin{enumerate}
\item
The bundles \(V\) and \(W\) are ordinary holomorphic vector bundles on
\(\overline\Sigma\), equipped with parabolic flags and standard
vector-bundle weights in \([0,1)\), and
\(
\operatorname{rk}V=p,
\) \(
\operatorname{rk}W=q.
\)

\item
The Higgs field is logarithmic and weakly parabolic, of the form
\[
\Phi
=
\begin{pmatrix}
0&\beta\\
\gamma&0
\end{pmatrix},
\quad
\beta\colon W\to
V\otimes K_{\overline\Sigma}(D),
\quad
\gamma\colon V\to
W\otimes K_{\overline\Sigma}(D).
\]
Thus its ordinary residue preserves the parabolic filtration, and its
strictly parabolic part vanishes on the associated graded.
\item The distinguished representative is obtained from the
group-theoretic lattice of \cite{BGPMiR20} by an integral Hecke
transformation to the standard vector-bundle lattice. The transformation
preserves the parabolic degree of every reduction, and the standard
representative is polystable. All kernels, images, saturations, ordinary
residues, and local lengths below are taken in this standard lattice.

\item
The parabolic degrees satisfy
\[
\pardeg(V)+\pardeg(W)=0.
\]
Here, for standard weights \(c_{j,\ell}\in[0,1)\), counted with
multiplicity,
\(
\pardeg(F):=\deg(F)+\sum_{j,\ell}c_{j,\ell}.
\)

\end{enumerate}
\end{proposition}

\begin{proof}
Put \(M_j:=\phi(c_j)^{-1}\), the positively oriented puncture monodromy.
After conjugating its compact elliptic factor into
\(\mathrm U(p)\times\mathrm U(q)\), write the group-theoretic weight of
\cite{BGPMiR20} as
\[
\alpha_j^{\mathrm B}
=
\operatorname{diag}
\bigl(
\theta^V_{j,1},\ldots,\theta^V_{j,p},
\theta^W_{j,1},\ldots,\theta^W_{j,q}
\bigr),
\qquad
0\leq\theta^V_{j,r},\theta^W_{j,s}<1.
\]
The convention of \cite[Remark~2.8]{BGPMiR20} associates
\(\alpha_j^{\mathrm B}\) to \(M_j\) through
\(\exp(2\pi\sqrt{-1}\alpha_j^{\mathrm B})\); it is not the standard vector-bundle convention. 
For \(0\leq\theta<1\), define
\[
\varepsilon(\theta):=
\begin{cases}
0,&\theta=0,\\
1,&0<\theta<1,
\end{cases}
\quad
a(\theta):=\varepsilon(\theta)-\theta=\{-\theta\}.
\]
The standard weights of \(V\) and \(W\) are
\(
a_{j,r}:=a(\theta^V_{j,r}),
\) \(b_{j,s}:=a(\theta^W_{j,s}).
\)
Equivalently, these are the eigenangles in \([0,1)\) of the compact
elliptic factor of the boundary holonomy \(\phi(c_j)\).

Let \(\varepsilon_j\) be obtained by applying \(\varepsilon\) entrywise to
\(\alpha_j^{\mathrm B}\), and put
\[
\widetilde\alpha_j
:=\alpha_j^{\mathrm B}-\varepsilon_j
=-\operatorname{diag}
\bigl(a_{j,1},\ldots,a_{j,p},b_{j,1},\ldots,b_{j,q}\bigr).
\]
Then
\(\exp(2\pi\sqrt{-1}\widetilde\alpha_j)
=\exp(2\pi\sqrt{-1}\alpha_j^{\mathrm B})\).
This is the integral Hecke transformation of
\cite[Section~3.3, Theorem~3.5]{BGPMiR20}.

If \(e^{\mathrm B}_{j,\ell}\) is a compatible local frame of
group-theoretic weight \(\theta_{j,\ell}\), the standard lattice is
generated by
\(
e_{j,\ell}
=z^{\varepsilon(\theta_{j,\ell})}e^{\mathrm B}_{j,\ell}.
\)
This Hecke modification is determined by the filtration and the integral
cocharacter \(-\varepsilon_j\), hence is independent of the compatible
splitting. The tame metric has the corresponding growth exponent
$a(\theta_{j,\ell})$: for a frame also adapted to the monodromy weight
filtration, its model estimate is
\[
\|e_{j,\ell}\|_h^2
\asymp
|z|^{2a(\theta_{j,\ell})}
|\log|z||^{\kappa_{j,\ell}}
\]
for an integer $\kappa_{j,\ell}$; see
\cite[Section~5.1, especially~(5.8)]{BGPMiR20}. The logarithmic factor
does not affect the parabolic weight, so the standard weights are
precisely the $a(\theta_{j,\ell})$.

We next verify logarithmicity. Suppose that a local Higgs-field coefficient
maps a source direction of group weight \(\theta_s\) to a target direction
of group weight \(\theta_t\), and set \(\mu=\theta_t-\theta_s\). In the
group-theoretic lattice,
\begin{equation}
\label{eq:parabolic-isotropy-sheaf}
{P}\!E^{\mathrm B}(\mathfrak m^{\mathbb C})
\cong
\bigoplus_{\mu}
\mathfrak m^{\mathbb C}_\mu
\otimes
\mathcal O_{\overline\Sigma}
\bigl(\lfloor-\mu\rfloor x_j\bigr);
\end{equation}
see \cite[(4.1)]{BGPMiR20}. Since \(-1<\mu<1\), its coefficient with
respect to \(dz/z\) has vanishing order at least
\(\mathbf 1_{\{\theta_t>\theta_s\}}\). In standard frames its vanishing
order is therefore at least
\(
\nu_{ts}
=\mathbf 1_{\{\theta_t>\theta_s\}}
+\varepsilon(\theta_s)-\varepsilon(\theta_t).
\)
Equivalently,
\[
\nu_{ts}=\mathbf 1_{\{a(\theta_t)<a(\theta_s)\}}.
\]
Thus every coefficient of $dz/z$ is holomorphic, and its residue vanishes
whenever the target standard weight is smaller than the source standard
weight. This proves both logarithmicity and preservation of the standard
parabolic filtration. The unequal-weight part of the residue is strictly
parabolic; the equal-weight part induces the graded residue. The Hecke transformation
preserves the intrinsic graded-residue orbit, but may change the ordinary
residue and the associated saturation lengths.

Finally, if \(F^{\mathrm B}\) is one of the two group-theoretic lattices and
\(F\) its standard Hecke transform, then
\(
\deg(F)=\deg(F^{\mathrm B})
-\sum_{j,\ell}\varepsilon(\theta_{j,\ell}).
\)
Consequently,
\[
\begin{aligned}
\pardeg(F)
&=\deg(F)+\sum_{j,\ell}
(\varepsilon(\theta_{j,\ell})-\theta_{j,\ell})=\deg(F^{\mathrm B})-\sum_{j,\ell}\theta_{j,\ell}.
\end{aligned}
\]
The last expression is the parabolic degree in the group-theoretic
convention of \cite[Remark~2.8 and Equation~(2.4)]{BGPMiR20}.

The same equality holds for every parabolic reduction, and we record the
reason because it is needed later. By \cite[Lemma~3.7]{BGPMiR20}, the
integral Hecke transformation gives a bijection
\(
\sigma^{\mathrm B}\leftrightarrow\sigma
\)
between reductions of the group-theoretic and standard representatives to
any fixed parabolic subgroup. Let \(\chi\) be an integral antidominant
character and let \(L^{\mathrm B}_{\sigma,\chi}\) and
\(L_{\sigma,\chi}\) be the associated character lines. Near \(x_j\), the
Hecke transformation changes a compatible local frame by
\(
e_{\sigma,\chi}
=
z^{m_j(\sigma,\chi)}e^{\mathrm B}_{\sigma,\chi},
m_j(\sigma,\chi)\in\mathbb Z.
\)
Consequently,
\(
\deg L_{\sigma,\chi}
=
\deg L^{\mathrm B}_{\sigma,\chi}
-
\sum_j m_j(\sigma,\chi),
\)
whereas the corresponding real character weight changes from
\(\theta_j(\sigma,\chi)\) to
\(m_j(\sigma,\chi)-\theta_j(\sigma,\chi)\). Therefore the two integer shifts
cancel and
\begin{equation}
\label{eq:Hecke-preserves-reduction-pardeg}
\pardeg_{\widetilde\alpha}(E,\sigma,\chi)
=
\pardeg_{\alpha^{\mathrm B}}
(E^{\mathrm B},\sigma^{\mathrm B},\chi).
\end{equation}
The equality extends by linearity to the real antidominant characters used
in the stability criterion. Thus the numerical stability inequalities are
unchanged.

For polystability, we use the metric characterization rather than infer
that every Levi reduction extends across the Hecke modification. The
harmonic metric on $\Sigma^\circ$ is unchanged; in the new frames it has
the adapted growth for $\widetilde\alpha_j$ and still satisfies the Hitchin
equations. The metric-to-polystability implication in
\cite[Theorem~5.1]{BGPMiR20} therefore proves polystability of the standard
representative.

The total determinant is a character of $\mathrm U(p,q)$ itself. At
stability parameter zero, the character degree identity preceding
\cite[Theorem~5.1]{BGPMiR20} gives it parabolic degree zero. Hence
\(
\pardeg(V)+\pardeg(W)=0.
\)
\end{proof}

\begin{remark}
\label{rem:Toledo-parabolic-degree}
Let $(E=V\oplus W,\Phi)$ be the parabolic
$\mathrm{U}(p,q)$-Higgs bundle associated with $\phi$. Consider the
determinant-ratio character
\[
\chi_0(h_+,h_-)
:=
\det(h_-)\det(h_+)^{-1},\,
(h_+,h_-)\in
\mathrm{GL}(p,\mathbb C)\times\mathrm{GL}(q,\mathbb C).
\]
Its associated character line is
\(
L_0:=\det(W)\otimes\det(V)^{-1}.
\)
If $a_{j,1},\ldots,a_{j,p}$ and
$b_{j,1},\ldots,b_{j,q}$ are the standard parabolic weights of $V$
and $W$ at $x_j$, respectively, put
\(
\lambda_j^0
:=
\sum_{s=1}^q b_{j,s}
-
\sum_{r=1}^p a_{j,r}.
\)
Then
\[
\begin{aligned}
\operatorname{pardeg}(E,\chi_0)
&:=
\deg L_0+\sum_{j=1}^n\lambda_j^0=
\operatorname{pardeg}(W)-\operatorname{pardeg}(V).
\end{aligned}
\]

With the boundary orientation and the standard weight convention fixed
above, the rotation-number boundary formula
\cite[(5.2)]{KPWUnitary}, together with the formula for the rotation number on
$\mathrm{U}(p)\times\mathrm{U}(q)$
\cite[Lemma~5.6]{KPWUnitary} and the ordinary clutching formula for $L_0$,
gives
\[
T(\Sigma,\phi)
=
\operatorname{pardeg}(E,\chi_0)
=
\operatorname{pardeg}(W)-\operatorname{pardeg}(V)=-2\operatorname{pardeg}(V).
\]
Equivalently, the parabolic Toledo invariant in the Higgs-bundle
normalization of \cite[Definition~4.1]{GPLM} is
\[
\tau_{\mathrm{par}}(E,\Phi)
=
\frac{2}{p+q}
\left(
q\,\operatorname{pardeg}(V)
-
p\,\operatorname{pardeg}(W)
\right)
=
2\operatorname{pardeg}(V),
\]
and hence, with our normalization,
\[
T(\Sigma,\phi)=-\tau_{\mathrm{par}}(E,\Phi).
\]
The hyperbolic and unipotent factors of the peripheral holonomies do not
alter this formula: the rotation number depends only on the compact
elliptic factor, while the character weights on the Higgs-bundle side are
determined by the same factor. When $D=\varnothing$, parabolic degrees
reduce to ordinary degrees, and the identity becomes
\(
T(\Sigma,\phi)=-2\deg(V),
\)
in agreement with~\eqref{eqn:toledo-closed}.
\end{remark}

\begin{remark}
\label{rem:normalized-representative-essential}
The Hecke modification cannot be omitted. For example, suppose locally that
a \(V\)-direction has group-theoretic weight \(\theta\in(0,1)\), a
\(W\)-direction has weight \(0\), and
\(
\gamma(e_V^{\mathrm B})=e_W^{\mathrm B}\frac{dz}{z}.
\)
In the group-theoretic lattice, this channel has nonzero ordinary residue.
In the standard lattice, \(e_V=ze_V^{\mathrm B}\) and
\(e_W=e_W^{\mathrm B}\), so
\(
\gamma(e_V)=ze_W\frac{dz}{z}.
\)
Its ordinary residue is zero, while its image has a length-one saturation
defect. Thus all kernels, images, saturations, ordinary residues, and local
lengths used below must be formed in the standard lattice of
Proposition~\ref{prop:normalized-logarithmic-representative}; they cannot be
combined with ordinary-residue data from the unmodified group-theoretic
lattice.
\end{remark}

From now on, \(V\), \(W\), \(\beta\), and \(\gamma\) always denote the
standard-lattice objects of
Proposition~\ref{prop:normalized-logarithmic-representative}, and every
subsequent sheaf-theoretic construction and ordinary residue is taken in
this single lattice.

\subsection{The \texorpdfstring{\(L^2\)}{L2} Dolbeault complex associated
with the tame harmonic bundle}
\label{subsec:L2-Dolbeault-complex}

We next construct the two-term complex that replaces
$[V\xrightarrow{\gamma}W\otimes K]$ in the presence of punctures. Its
local modifications are determined by the nilpotent residue at parabolic
weight zero. They are exactly the corrections needed for the signature
formula. 

Let \(\Sigma^\circ=\overline\Sigma\setminus D\), and let
\(\mathscr E=V\oplus W\) be the normalized logarithmic lattice of the tame
harmonic bundle associated with \(\phi\). Here a logarithmic lattice means a locally free \(\mathcal O_{\overline\Sigma}\)-extension of the holomorphic bundle on \(\Sigma^\circ\), inside its sheaf of meromorphic sections, for which the Higgs field extends as a logarithmic morphism
\(\Phi\colon\mathscr E\to\mathscr E\otimes K_{\overline\Sigma}(D)\).
The adapted harmonic metric determines a parabolic growth filtration, and the lattice is called normalized when all its parabolic weights are represented in \([0,1)\).
The \(U(p,q)\)-decomposition is preserved by this extension and gives \(\mathscr E=V\oplus W\).
Choose a complete K\"ahler metric on $\Sigma^\circ$ that agrees with a
Poincar\'e metric near each puncture. All $L^2$ norms below use this metric
and the adapted harmonic metric $h$ on $\mathscr E$. Put
\[
D^\lambda
=
\overline\partial_{\mathscr E}+\Phi
+
\lambda\bigl(\partial_{\mathscr E,h}+\Phi_h^\dagger\bigr),
\qquad \lambda\in\mathbb C.
\]
Thus \(D^0=\overline\partial_{\mathscr E}+\Phi\) is the Higgs
Dolbeault differential, whereas \(D^1\) is the flat connection defining
\(\mathcal E_\phi\). Let
\(\mathscr L_{(2)}^\bullet(D^\lambda)\) denote its analytic maximal graph-domain
\(L^2\)-complex and \(H_{(2)}^k(D^\lambda)\) the cohomology of global
sections.

Two comparisons used below must be distinguished. The algebraic complex
constructed below is a sheaf-theoretic model for
\(\mathscr L_{(2)}^\bullet(D^0)\). The comparison with the flat local system
is instead made on global cohomology, through common \(L^2\)-harmonic
representatives. The harmonic-bundle K\"ahler identities give
\(
\Delta_\lambda=(1+|\lambda|^2)\Delta_0;
\)
see \cite[Section~2]{Simpson} and
\cite[Section~6.2.2, especially Proposition~6.5 and
Corollary~6.6]{MochizukiL2Complexes}. Hence \(D^0\) and \(D^1\) have the same
\(L^2\)-harmonic forms. At \(\lambda=1\),
\cite[Theorem~1.3]{MochizukiL2Complexes}, together with the
Riemann--Hilbert correspondence, identifies flat \(L^2\)-cohomology with
intersection cohomology. We do not assert a sheaf-level quasi-isomorphism
between the Higgs complex at \(\lambda=0\) and the intersection complex.
The explicit \(W_0/W_{-2}\) formula from
\cite[Section~3.2, equation~(2)]{DPS} will be used only on the nilpotent
zero-primary part of the residue.

We now recall this algebraic description in the present notation. Fix a marked
point \(x_j\). Let \(\operatorname{Gr}_{0,j}(\mathscr E)\) be the parabolic
graded fiber of weight \(0\) at \(x_j\); if the weight \(0\) does not occur, we
put \(\operatorname{Gr}_{0,j}(\mathscr E)=0\). The graded residue preserves
this space. Write
\[
\left.
\operatorname{GrRes}_{x_j}(\Phi)
\right|_{\operatorname{Gr}_{0,j}(\mathscr E)}
=
S_j+N_j,
\qquad
[S_j,N_j]=0,
\]
where \(S_j\) is semisimple and \(N_j\) is nilpotent; see
\cite[Section~4.1]{BGPMiR20}. Let \(G_j:=\ker S_j\) be the zero-primary part
of \(S_j\). Then \(N_j\) preserves \(G_j\). We denote by
\(W_\bullet G_j=W_\bullet(N_j)\) the monodromy weight filtration of \(N_j\),
centered at zero, so that \(N_j(W_\ell G_j)\subset W_{\ell-2}G_j\).

Let \(i_j:\{x_j\}\hookrightarrow\overline\Sigma\) be the inclusion. Evaluation
at \(x_j\), followed by projection to the weight-zero graded fiber and then to
\(G_j\), gives a canonical map \(q_j:\mathscr E\to(i_j)_*G_j\). Define two
coherent subsheaves of \(\mathscr E\) by
\begin{equation}
\label{eq:def-L2-lattice-A}
\mathcal A
:=
\ker(
\mathscr E
\longrightarrow
\bigoplus_{j=1}^{n}(i_j)_*
\bigl(G_j/W_0G_j\bigr)
)
\end{equation}
and
\begin{equation}
\label{eq:def-L2-lattice-B}
\mathcal B
:=
\ker(
\mathscr E
\longrightarrow
\bigoplus_{j=1}^{n}(i_j)_*
\bigl(G_j/W_{-2}G_j\bigr)
),
\end{equation}
where the maps are induced by the \(q_j\)'s. Equivalently, near \(x_j\),
\(\mathcal A_{x_j}=\{s\in\mathscr E_{x_j}:q_j(s)\in W_0G_j\}\) and
\(\mathcal B_{x_j}=\{s\in\mathscr E_{x_j}:q_j(s)\in W_{-2}G_j\}\). Away from
\(D\), both sheaves are equal to \(\mathscr E\).

These two sheaves will define a coherent model for the analytic
$L^2$ complex. On a summand where the semisimple residue is nonzero, the
residue is invertible, so the logarithmic complex is locally contractible
near the puncture. On such a summand, the individual sheaves of this
model need not coincide with the sheaves of $L^2$ sections; the comparison
is a quasi-isomorphism of complexes. Since \(W_{-2}G_j\subset W_0G_j\), we have
\begin{equation}
\label{eq:L2-lattice-inclusions}
\mathscr E(-D)\subseteq\mathcal B\subseteq\mathcal A\subseteq\mathscr E.
\end{equation}
Indeed, sections of $\mathscr E(-D)$ vanish at every marked point, so
they satisfy both defining conditions. The quotients of $\mathscr E$ by
$\mathcal A$ and $\mathcal B$ are supported on $D$. Both kernels are
coherent and torsion-free, hence locally free because $\overline\Sigma$
is a smooth curve.

Let \(\vartheta=\operatorname{id}_V\oplus(-\operatorname{id}_W)\). Since the
Higgs field is off diagonal, \(\vartheta\Phi\vartheta^{-1}=-\Phi\). Hence the
zero-primary spaces $G_j$ are preserved by $\vartheta$. On $G_j$,
conjugation by $\vartheta$ takes $N_j$ to $-N_j$. Multiplying a nilpotent
operator by $-1$ leaves its monodromy weight filtration unchanged, so
$\vartheta$ preserves that filtration as well. It follows that
\[
\mathcal A=\mathcal A_V\oplus\mathcal A_W,
\qquad
\mathcal B=\mathcal B_V\oplus\mathcal B_W,
\]
where \(\mathcal A_V=\mathcal A\cap V\), \(\mathcal A_W=\mathcal A\cap W\),
\(\mathcal B_V=\mathcal B\cap V\), and \(\mathcal B_W=\mathcal B\cap W\).
Thus \(V(-D)\subseteq\mathcal B_V\subseteq\mathcal A_V\subseteq V\), and
similarly \(W(-D)\subseteq\mathcal B_W\subseteq\mathcal A_W\subseteq W\).

The differential has the required target. Indeed, if
$\Phi=R_j(z)\,dz/z$ near $x_j$, preservation of the parabolic filtration
and the primary decomposition gives
$q_j(R_j(z)s)=N_jq_j(s)$. Since
$N_j(W_0G_j)\subset W_{-2}G_j$, we obtain
$\Phi(\mathcal A)\subset\mathcal B\otimes K_{\overline\Sigma}(D)$.
We can therefore define the coherent complex
\begin{equation}
\label{eq:L2-total-Dolbeault-complex}
\mathcal C_{L^2}^{\bullet}
=
\left[
\mathcal A_V\oplus\mathcal A_W
\xrightarrow{\Phi}
(\mathcal B_V\oplus\mathcal B_W)\otimes K_{\overline\Sigma}(D)
\right].
\end{equation}
Since \(\Phi\) is off diagonal, this complex splits as
\begin{equation}
\label{eq:L2-gamma-beta-splitting}
\mathcal C_{L^2}^{\bullet}
=
\mathcal C_\gamma^\bullet
\oplus
\mathcal C_\beta^\bullet,
\end{equation}
where
\[
\mathcal C_\gamma^\bullet
=
\left[
\mathcal A_V
\xrightarrow{\gamma}
\mathcal B_W\otimes K_{\overline\Sigma}(D)
\right],
\qquad
\mathcal C_\beta^\bullet
=
\left[
\mathcal A_W
\xrightarrow{\beta}
\mathcal B_V\otimes K_{\overline\Sigma}(D)
\right].
\]

The comparison of this coherent complex with the analytic maximal
\(L^2\)-complex is the only analytic input needed below. The local lattice
comparison and the punctured-curve holomorphic \(L^2\) identification are
collected in Appendix~\ref{app:L2-minimal-extension}.

\begin{definition}
Let \(h_0(\phi):=\dim H^0(\Sigma;\mathcal E_\phi)\), let
\(h_\partial(\phi):=\dim H^0(\partial\Sigma;\mathcal E_\phi)\), and let
\(h^-(\phi)\) denote the negative index of the Hermitian intersection form
\(iQ_\phi\).
\end{definition}

\begin{lemma}
\label{lem:L2-Hodge-gamma}
One has
\begin{equation}
\label{eq:hminus-gamma-hypercohomology}
h^-(\phi)
=
\dim\mathbb H^1
\bigl(\overline\Sigma,\mathcal C_\gamma^\bullet\bigr),
\end{equation}
and
\begin{equation}
\label{eq:hminus-h0-Euler-gamma}
h^-(\phi)-h_0(\phi)
=
-\chi(\mathcal C_\gamma^\bullet).
\end{equation}
Here $\mathbb H^k(\overline\Sigma,\mathcal C_\gamma^\bullet)$
denotes the hypercohomology of the two-term complex
$\mathcal C_\gamma^\bullet$, computed as the cohomology of its
global Dolbeault total complex.
\end{lemma}

\begin{proof}
Use the common \(L^2\)-harmonic representatives from
Proposition~\ref{prop:appendix-general-residue-L2-model}. The orthogonal splitting
\(
\mathcal C_{L^2}^{\bullet}
=\mathcal C_\gamma^\bullet\oplus\mathcal C_\beta^\bullet
\)
induces the corresponding splitting of the \(D^0\)-harmonic spaces.

Let \(J\) act by \(i\) on \(V\) and by \(-i\) on \(W\). In a local
holomorphic coordinate \(z\), with our orientation convention,
\(*dz=-i\,dz\) and \(*d\overline z=i\,d\overline z\). In the analytic complex, the total degree-one $\gamma$-summand consists of
$L^2$ forms in
\[
\Omega^{0,1}(V)\oplus\Omega^{1,0}(W),
\]
with the domain inherited from the full analytic complex.
For \(v\in V\) and \(w\in W\),
\(
*J(d\overline z\otimes v)=-d\overline z\otimes v,
\) \(
*J(dz\otimes w)=-dz\otimes w.
\)
Thus \(*J=-1\) on the degree-one \(\gamma\)-summand and \(*J=+1\) on
the degree-one \(\beta\)-summand. A common harmonic representative \(u\)
satisfies, with the normalization used here,
\(
iQ_\phi([u],[u])=\langle *Ju,u\rangle_{L^2};
\)
this is the pointwise identity used in
\cite[Proposition~2.1]{KPWUnitary}. Hence the form is negative definite on
the harmonic \(\gamma\)-summand and positive definite on the harmonic
\(\beta\)-summand, proving
\eqref{eq:hminus-gamma-hypercohomology}.

For the Euler characteristic, we need duality only between degrees zero
and two. The parallel Hermitian form and integration give a nondegenerate
sesquilinear pairing
\[
IH^0(\overline\Sigma;\mathcal E_\phi)
\times IH^2(\overline\Sigma;\mathcal E_\phi)
\longrightarrow\mathbb C;
\]
see \cite[Section~6.2.4]{MochizukiL2Complexes}. On common harmonic
representatives, the degree-zero $\gamma$-summand takes values in $V$,
whereas its degree-two summand takes values in $W$. The roles of $V$ and
$W$ are reversed for the $\beta$-summand. Since $V\perp W$, the pairing
between degrees zero and two vanishes on equal summands. Its two
cross-pairings are therefore nondegenerate: a class annihilating the
opposite summand also annihilates its own summand, and hence annihilates
the entire complementary-degree cohomology. Consequently,
\[
\dim\mathbb H^2(\mathcal C_\gamma^\bullet)
=
\dim\mathbb H^0(\mathcal C_\beta^\bullet).
\]
We use only this dimension equality. The pairing in degree one behaves
differently: it is definite on each summand, as established above.

In degree zero, Proposition~\ref{prop:appendix-general-residue-L2-model} gives
\[
\mathbb H^0(\mathcal C_{L^2}^{\bullet})
\cong IH^0(\overline\Sigma;\mathcal E_\phi)
\cong H^0(\Sigma;\mathcal E_\phi).
\]
The direct-sum decomposition therefore gives
\[
h_0(\phi)
=\dim\mathbb H^0(\mathcal C_\gamma^\bullet)
+\dim\mathbb H^0(\mathcal C_\beta^\bullet).
\]
Consequently,
\[
\begin{aligned}
-\chi(\mathcal C_\gamma^\bullet)
&=
-\dim\mathbb H^0(\mathcal C_\gamma^\bullet)
+\dim\mathbb H^1(\mathcal C_\gamma^\bullet)
-\dim\mathbb H^2(\mathcal C_\gamma^\bullet)\\
&=
h^-(\phi)
-
\dim\mathbb H^0(\mathcal C_\gamma^\bullet)
-
\dim\mathbb H^0(\mathcal C_\beta^\bullet)\\
&=
h^-(\phi)-h_0(\phi).
\end{aligned}
\]
This proves \eqref{eq:hminus-h0-Euler-gamma}.
\end{proof}

\begin{corollary}
\label{eqn:sign-versus-chi-of-C-gamma}
The signature is given by
\[
\operatorname{sign}(\phi)
=
-\dim E\cdot\chi(\Sigma)
+2\chi(\mathcal C_\gamma^\bullet)
-h_\partial(\phi).
\]
\end{corollary}

\begin{proof}
The intersection form is nondegenerate on
\[
\widehat H^1(\Sigma;\mathcal E_\phi)
:=\operatorname{Im}\bigl(
H^1(\Sigma,\partial\Sigma;\mathcal E_\phi)
\longrightarrow H^1(\Sigma;\mathcal E_\phi)
\bigr).
\]
Its dimension is
\begin{equation}
\label{eqn:dimension-parabolic-cohomology}
\dim\widehat H^1(\Sigma;\mathcal E_\phi)
=
-\dim E\cdot\chi(\Sigma)
-h_\partial(\phi)
+2h_0(\phi);
\end{equation}
see \cite[(6.2)]{KPWUnitary}. The signature of a nondegenerate Hermitian
form is its dimension minus twice its negative index. Therefore
Lemma~\ref{lem:L2-Hodge-gamma} gives
\[
\begin{aligned}
\operatorname{sign}(\phi)
&=\dim\widehat H^1(\Sigma;\mathcal E_\phi)-2h^-(\phi)\\
&=-\dim E\cdot\chi(\Sigma)-h_\partial(\phi)
  +2\bigl(h_0(\phi)-h^-(\phi)\bigr)\\
&=-\dim E\cdot\chi(\Sigma)
  +2\chi(\mathcal C_\gamma^\bullet)-h_\partial(\phi).
\end{aligned}
\]
\end{proof}

\subsection{Local lattice corrections}

\begin{notation}\label{eq:local-L2-lattice-lengths}
For each marked point \(x_j\), set
\begin{equation}
a_j:=\operatorname{length}_{x_j}(V/\mathcal A_V),
\qquad
b_j:=\operatorname{length}_{x_j}(W/\mathcal B_W).
\end{equation}
\end{notation}

These lengths measure how much the $L^2$ conditions shrink the original
lattices. To use them in linear algebra, we distinguish stalks from fibers.
Fix a point
$x_j\in D$, and choose a local coordinate $z$ centered at $x_j$.  Let
$\mathcal O_j=\mathcal O_{\overline\Sigma,x_j}$.  We write $W_j$ for the stalk at $x_j$ of the sheaf of
holomorphic sections of $W$.  Thus $W_j$ is a free $\mathcal O_j$-module of
rank $q$.  The fiber of $W$ at $x_j$ is the complex vector space
$W|_{x_j}=W_j/zW_j$, where $zW_j$ is the submodule of germs of sections
vanishing at $x_j$.
With this convention, an inclusion of sheaves such as
$W(-D)\subset\mathcal B_W\subset W$ gives, at the stalk $x_j$, the inclusions
$zW_j\subset(\mathcal B_W)_j\subset W_j$.  Hence
$P_j:=(\mathcal B_W)_j/zW_j$ is naturally a subspace of the fiber
$W|_{x_j}=W_j/zW_j$.  In particular, $W_j/zW_j$ has dimension $q$.
 Since
$zW_j\subset(\mathcal B_W)_j$, the quotient $W_j/(\mathcal B_W)_j$ is killed
by $z$, and therefore its length is equal to its dimension as a complex vector
space.  From the exact sequence
\[
0\to(\mathcal B_W)_j/zW_j\to W_j/zW_j\to W_j/(\mathcal B_W)_j\to0,
\]
we get
$\dim P_j=q-b_j$.  

\subsection{The Riemann--Roch calculation}

We apply ordinary Riemann--Roch to the locally free sheaves
$\mathcal A_V$ and $\mathcal B_W$. Their degrees differ from those of
$V$ and $W$ by the local lengths just defined. Thus
\begin{align} \label{eqn:RR}
\chi(\mathcal A_V)&=p(1-g)+\deg V-\sum_{j=1}^n a_j, \\ 
\chi\bigl(\mathcal B_W\otimes K_{\overline\Sigma}(D)\bigr)
&=
q(1-g)+\deg W-\sum_{j=1}^{n}b_j+q(2g-2+n).
\end{align}
Hence
\begin{equation}
\label{eq:Euler-Cgamma-degree}
\chi(\mathcal C_\gamma^\bullet)
=
\deg V-\deg W
-
\sum_{j=1}^{n}(a_j+q-b_j)
-
(p+q)(g-1).
\end{equation}

\subsection{Kernel and image of the Higgs field} 

We now express $\deg V-\deg W$ in terms of the kernel and saturated
image of the Higgs field, as in
Subsection~\ref{subsec:closed-surface-proof}. Apply this construction to the
logarithmic
morphism
\(
\gamma\colon V\to
W\otimes K_{\overline\Sigma}(D);
\)
compare \cite[Proposition~4.2]{GPLM}. 

\begin{notation}
Let
\(
N:=\ker\gamma\subset V
\)
be the sheaf-theoretic kernel of $\gamma$, and define
\(
I^0
:=
\operatorname{im}(\gamma)
\otimes K_{\overline\Sigma}(D)^{-1}
\subset W.
\)
Let
\(
I:=\operatorname{Sat}_{W}(I^0)
\)
be the saturation of \(I^0\) in \(W\), and set
\(
Q:=W/I.
\) 
\end{notation}

The \emph{saturation} of $I^0\subset W$ is the inverse image in $W$ of
the torsion subsheaf of $W/I^0$. Equivalently, it is the smallest subsheaf
$I\subset W$ containing $I^0$ for which $W/I$ is torsion-free.
On a smooth curve, $I$ is a holomorphic subbundle of $W$, and $I/I^0$
is a torsion sheaf supported at finitely many points.

Since \(\overline\Sigma\) is a smooth curve, every torsion-free coherent
sheaf is locally free. The quotient \(V/N\), being isomorphic to the image
of \(\gamma\), is torsion-free; hence \(N\) is a saturated holomorphic
subbundle of \(V\). Likewise, \(I\) is locally free, and the definition of
saturation implies that \(Q\) is torsion-free and therefore locally free.
In particular, the sequence $0\to I\to W\to Q\to0$ is exact on fibers.

\begin{notation}
Let $P_{Q,j}$ be the image of the subspace
$P_j=(\mathcal B_W)_j/zW_j\subset W|_{x_j}$ under
$W|_{x_j}\to Q|_{x_j}$, and put
$d_{Q,j}:=\dim_{\mathbb C}P_{Q,j}$.
\end{notation}

\begin{notation}
Set
\begin{equation}
\label{eq:ranks-N-I-Q-positive}
k:=\operatorname{rk}I
=
\operatorname{rk}\gamma,
\qquad
\operatorname{rk}N=p-k,
\qquad
m:=\operatorname{rk}Q=q-k.
\end{equation}
\end{notation}

Since \(k\leq p\), we have
\begin{equation}
\label{eq:m-lower-positive}
m=q-k\geq q-p.
\end{equation}

The subbundle \(N\) is equipped with the parabolic structure induced from
\(V\), and \(V/N\) carries the corresponding quotient parabolic structure.
The morphism induced by \(\gamma\) gives an isomorphism of coherent sheaves
\(
V/N
\cong
I^0\otimes K_{\overline\Sigma}(D).
\)
Transporting the quotient parabolic structure through this isomorphism and
then twisting by \(K_{\overline\Sigma}(D)^{-1}\) gives a parabolic structure
on \(I^0\), which we denote by \(I^0_{\mathrm{quot}}\).

Independently, the inclusion \(I\subset W\) equips \(I\) with the parabolic
structure induced from \(W\); we denote this parabolic bundle by \(I_W\).
The quotient \(Q=W/I\) is endowed with the quotient parabolic structure.
The additivity of parabolic degree in the exact sequence
\(
0\to I
\to W
\to Q
\to 0
\)
gives
\begin{equation}
\label{eq:pardeg-W-I-Q}
\pardeg(W)
=
\pardeg(I_W)+\pardeg(Q),
\end{equation}
see \cite[Section~2]{GPLM}.

\begin{notation}
We write
\[
\ell_{x_j}(I/I^0)
:=
\operatorname{length}_{\mathcal O_{\overline\Sigma,x_j}}
\bigl((I/I^0)_{x_j}\bigr).
\]
We also set
\begin{equation}
\label{eq:interior-saturation-defect-positive}
\ell_{\mathrm{int}}(I/I^0)
:=
\sum_{x\in\overline\Sigma\setminus D}
\operatorname{length}_{\mathcal O_{\overline\Sigma,x}}
\bigl((I/I^0)_x\bigr)
\geq0.
\end{equation}
\end{notation}

The exact sequence
\(
0\to N
\to V
\to
I^0\otimes K_{\overline\Sigma}(D)
\to 0
\)
gives
\[
\deg V
=
\deg N+\deg I^0+k(2g-2+n),
\]
while
\(
\deg I^0
=
\deg I
-
\ell_{\mathrm{int}}(I/I^0)
-
\sum_{j=1}^{n}\ell_{x_j}(I/I^0).
\)

Hence
\begin{align*}
\deg V
=
\deg N+\deg I
-\ell_{\mathrm{int}}(I/I^0)
-\sum_{j=1}^{n}\ell_{x_j}(I/I^0)
+k(2g-2+n).
\end{align*}
Moreover,
\begin{align*}
\deg W=\deg I+\deg Q.
\end{align*}
Substituting these identities into
\eqref{eq:Euler-Cgamma-degree} gives
\begin{align*}
\chi(\mathcal C_\gamma^\bullet)
&=
\deg N-\deg Q
-
\ell_{\mathrm{int}}(I/I^0)
-
\sum_{j=1}^{n}\ell_{x_j}(I/I^0)\\
&-
\sum_{j=1}^{n}(a_j+q-b_j-k)
-
(p+q-2k)(g-1).
\end{align*}
Substitution into
Corollary~\ref{eqn:sign-versus-chi-of-C-gamma} gives the exact formula
\begin{align} \label{eq:sign deg}
\begin{split}
\operatorname{sign}(\phi)
&= 
2k(2g-2)+(p+q)n-h_\partial(\phi)+2\deg N-2\deg Q\\
&-2\ell_{\mathrm{int}}(I/I^0)-2
\sum_{j=1}^{n}(\ell_{x_j}(I/I^0)+a_j+q-b_j-k).
\end{split}
\end{align}

\subsection{Local inequalities}

The boundary terms in \eqref{eq:sign deg} will be estimated in two steps.
First, the vanishing of the ordinary residue forces a saturation defect:
\[
\ell_{x_j}(I/I^0)\geq k-r_j^{\mathrm{full}}.
\]
Second, the $L^2$-lattice corrections satisfy
\[
a_j+q-b_j\geq r_j^{\mathrm{full}}+d_{Q,j}.
\]
These are Corollary~\ref{eq:local-length-full-residue} and
Lemma~\ref{lem:full-residue-lattice-inequality}, respectively. Together they give
\begin{align} \label{eq:aj+dj-k}
\ell_{x_j}(I/I^0)+a_j+q-b_j-k\ge d_{Q,j}.
\end{align}
Thus the local boundary term is nonnegative, and also controls the part of
the target lattice that survives in $Q$. We now define the ordinary residue
rank $r_j^{\mathrm{full}}$ used in these estimates.

\subsection{Ordinary and graded residues}
\label{subsec:full-graded-residues}

Two residue maps enter the proof. The \emph{ordinary logarithmic residue}
records the full coefficient of $dz/z$ at the marked point. The
\emph{graded residue} retains only the maps between equal parabolic
weights. We use the ordinary residue to estimate the saturation defect and
the graded residue to describe the boundary monodromy. Both are formed in
the standard lattice fixed above.

Fix a marked point \(x_j\), choose a local coordinate \(z\) centered at
\(x_j\), and choose a holomorphic frame compatible with the standard
parabolic filtration. All local data are attached to the positively
oriented puncture monodromy \(M_j=L_j^{-1}\). Let
\(
\widetilde\alpha_j
=-\operatorname{diag}
\bigl(a_{j,1},\ldots,a_{j,p},b_{j,1},\ldots,b_{j,q}\bigr)
\)
be the group-theoretic lift belonging to the standard lattice. Thus the
nonnegative numbers \(a_{j,r},b_{j,s}\) are the standard vector-bundle
weights, whereas \(\widetilde\alpha_j\) is the group-theoretic weight in the
local normal form.

After a compatible local parabolic gauge, the model of
\cite[Section~4.1 and Equation~(5.7)]{BGPMiR20}, transported by the integral
Hecke transformation, has the form
\begin{equation}
\label{eq:BGPM-local-model}
\Phi
=
\left(
s_j+\operatorname{Ad}(z^{\widetilde\alpha_j})Y_j+\psi_j
\right)\frac{dz}{z}.
\end{equation}
Here \(s_j\) is semisimple, \(Y_j\) is nilpotent,
\([s_j,Y_j]=0\), and the equal-weight part \(s_j+Y_j\) represents
\(\operatorname{GrRes}_{x_j}(\Phi)\). The residue contribution
$\psi_j(0)$ is strictly parabolic. All eigenvalues of
\(\operatorname{ad}(\widetilde\alpha_j)\) lie in \((-1,1)\). Since \(Y_j\)
is fixed by \(\exp(2\pi\sqrt{-1}\widetilde\alpha_j)\), its
\(\operatorname{ad}(\widetilde\alpha_j)\)-weights are integral and hence
zero. Therefore
\(
\operatorname{Ad}(z^{\widetilde\alpha_j})Y_j=Y_j.
\)
The term $\psi_j(0)$ may nevertheless be nonzero. Its nonzero blocks map
a weight-$a$ direction to a weight-$b$ direction with $b>a$, so they vanish
on the associated graded. This is why the ordinary and graded residues
need not coincide.

Recall the induced logarithmic morphism
$\bar\gamma\colon V/N\to I\otimes K_{\overline\Sigma}(D)$,
which is generically an isomorphism. 
Restriction to $x_j$ gives a complex-linear map
\(
\bar\gamma|_{x_j}\colon
(V/N)|_{x_j}
\to
I|_{x_j}\otimes K_{\overline\Sigma}(D)|_{x_j}.
\)
The canonical residue isomorphism
$\operatorname{res}_{x_j}\colon
K_{\overline\Sigma}(D)|_{x_j}\to\mathbb C$
sends the fiber class of $a(z)\,dz/z$ to $a(0)$.
We define the full, or ordinary, residue of $\gamma$ by
\begin{equation}
\label{eq:full-residue-map-positive}
R_j^{\mathrm{full}}
:=
\bigl(
\operatorname{id}_{I|_{x_j}}
\otimes\operatorname{res}_{x_j}
\bigr)
\circ\bigl(\bar\gamma|_{x_j}\bigr)
\colon
(V/N)|_{x_j}\longrightarrow I|_{x_j}.
\end{equation}
The definition is independent of the local coordinate and frames.

The parabolic structures on \(V/N\) and \(I_W\) determine filtrations of the
source and target of \(R_j^{\mathrm{full}}\). Since \(\gamma\) is a
parabolic logarithmic morphism, \(R_j^{\mathrm{full}}\) respects these
filtrations and therefore induces a map on their associated graded spaces:
\(
R_j^{\mathrm{gr}}\colon
\operatorname{Gr}_{x_j}(V/N)
\to
\operatorname{Gr}_{x_j}(I_W).
\)
Composing with the graded quotient from $V$ and the graded inclusion into
$W$ recovers the $\gamma$-component of
$\operatorname{GrRes}_{x_j}(\Phi)$. In frames adapted to the filtrations,
$R_j^{\mathrm{full}}$ is block triangular: its diagonal blocks give
$R_j^{\mathrm{gr}}$, and its remaining blocks strictly increase the
parabolic weight. A choice of splitting is needed to isolate those
remaining blocks. We therefore use only the intrinsic ranks of
$R_j^{\mathrm{full}}$ and $R_j^{\mathrm{gr}}$.

The following elementary estimate will be applied to the full residue.

\begin{lemma}
\label{lem:length-residual-rank-positive}
Let
\(
f\colon A\to B
\)
be an injective morphism between free \(\mathbb C[[z]]\)-modules of the same
rank \(k\), with finite-length cokernel. Let
\(
\overline f\colon A/zA\to B/zB
\)
be the induced map on the residue spaces. Then
\begin{equation}
\label{eq:length-residual-rank-DVR}
\operatorname{length}_{ \mathbb C[[z]]}(\operatorname{coker}f)
\geq
k-\operatorname{rank}\overline f.
\end{equation}
\end{lemma}

\begin{proof}
By the ordinary Smith normal form, there are bases of \(A\) and \(B\) in
which
\(
f
=
\operatorname{diag}
\bigl(
z^{\mu_1},\ldots,z^{\mu_k}
\bigr),
\mu_a\geq0.
\)
Consequently,
\(
\operatorname{length}_{ \mathbb  C[[z]]}(\operatorname{coker}f)
=
\sum_{a=1}^{k}\mu_a.
\)
After reduction modulo \(z\), precisely those diagonal entries with
\(\mu_a=0\) remain nonzero. Thus
\(
\operatorname{rank}\overline f
=
\#\{a:\mu_a=0\},
\)
and hence
\(
k-\operatorname{rank}\overline f
=
\#\{a:\mu_a>0\}
\leq
\sum_{a=1}^{k}\mu_a.
\)
\end{proof}

To apply Lemma~\ref{lem:length-residual-rank-positive}, choose a local coordinate $z$ at $x_j$ and write
$\bar\gamma=f_j\,dz/z$. Here
$f_j\colon(V/N)_j\to I_j$ is an injective morphism of free
$\mathcal O_j$-modules of rank $k$, with image $I_j^0$ and cokernel
$(I/I^0)_j$.
After completion, this gives an injective morphism
\(
\widehat f_j\colon
\widehat{(V/N)_j}\to\widehat{I_j}
\)
of free $\mathbb C[[z]]$-modules of rank $k$.
Its reduction modulo $z$ is precisely $R_j^{\mathrm{full}}$, under
the natural identifications
\(
\widehat{(V/N)_j}/z\widehat{(V/N)_j}
\cong (V/N)|_{x_j}
\) and \(
\widehat{I_j}/z\widehat{I_j}
\cong I|_{x_j}.
\)
Since completion preserves the length of a finite-length module,
Lemma~\ref{lem:length-residual-rank-positive} gives the following corollary.

\begin{corollary}\label{eq:local-length-full-residue}
\begin{equation}\label{local estimate length I/I^0}
\ell_{x_j}(I/I^0)
\geq
k-r_j^{\mathrm{full}},
\qquad
r_j^{\mathrm{full}}
:=
\operatorname{rank}R_j^{\mathrm{full}}.
\end{equation}

\end{corollary}

The ordinary residue is essential in this estimate. At the level of local
parabolic bundles, take line bundles of weights $0\leq a<b<1$ and the map
$\gamma(e_V)=e_W\,dz/z$. This map has ordinary residue rank one and graded
residue rank zero, while its image is already saturated. Thus the
local algebraic estimate would fail if the ordinary residue rank were
replaced by the graded residue rank: it would read $0\geq1$ in this
example.

\subsection{Low-dimensional examples}

The rank-one and rank-two cases illustrate the roles of the residue rank,
the saturation length, and the lattice corrections.

\subsubsection{One-dimensional representations}

First consider the negative-definite convention
\begin{align*}
\phi\colon\pi_1(\Sigma)\longrightarrow U(0,1),\qquad E=W.
\end{align*}
The Higgs field is zero because the symmetric space of $U(1)$ is a point.
Let $\alpha_j\in[0,1)$ be the standard parabolic weight at
$x_j$, and put $D_0:=\sum_{\alpha_j=0}x_j$.
At a zero-weight marked point the nilpotent residue vanishes, and its
monodromy weight filtration satisfies $W_0G_j=G_j$ and $W_{-2}G_j=0$.
At a positive-weight marked point there is no lattice correction. Hence
\begin{align*}
\mathcal A_W=W,\qquad\mathcal B_W=W(-D_0),
\end{align*}
and the $L^2$-Dolbeault complex is simply
\begin{align*}
\mathcal C_{L^2}^{\bullet}=\left[W\xrightarrow{\,0\,}W(-D_0)\otimes K_{\overline\Sigma}(D)\right].
\end{align*}
In this convention,
\begin{align*}
I=0,\qquad Q=W,\qquad r_j^{\mathrm{full}}=0,
\end{align*}
and 
\begin{align*}
a_j=0,\quad b_j=\begin{cases}1,&\alpha_j=0\\0,&\alpha_j>0\end{cases},\quad 
d_{Q,j}=1-b_j.
\end{align*}
Since, for a $U(1)$-representation, $\alpha_j=0$ exactly when
$\phi(c_j)=1$, one also has
$h_{\partial,j}(\phi)=b_j$,
\begin{align*}
h_\partial(\phi)=\#\{j:\phi(c_j)=1\}.
\end{align*}
Finally,
\begin{align*}
h_0(\phi)=\begin{cases}1,&\phi\text{ is the trivial character},\\0,&\text{otherwise}.\end{cases}
\end{align*}
The positive-definite $U(1,0)$ convention is obtained by interchanging $V$ and $W$.

\subsubsection{Two-dimensional representations}
For a $U(1,1)$-representation, the associated Higgs bundle has the form
\begin{align*}
E=V\oplus W,\quad  \Phi=\begin{pmatrix}0&\beta\\\gamma&0\end{pmatrix},
\end{align*}
where $V$ and $W$ are parabolic line bundles and
\begin{align*}
\beta\colon W\longrightarrow V\otimes K_{\overline\Sigma}(D),\qquad\gamma\colon V\longrightarrow W\otimes K_{\overline\Sigma}(D).
\end{align*}
Thus
\begin{align*}
\mathcal C_\gamma^\bullet=\left[\mathcal A_V\xrightarrow{\gamma}\mathcal B_W\otimes K_{\overline\Sigma}(D)\right]
\end{align*}
is a two-term complex of line bundles.

The sheaf-theoretic data depend on whether $\gamma$ vanishes identically.

If $\gamma=0$,
then 
\begin{align*}
k=0,\qquad N=V,\qquad I=0,\qquad Q=W,\qquad r_j^{\mathrm{full}}=0,\qquad d_{Q,j}=1-b_j.
\end{align*}
If $\gamma\neq 0$, then it has generic rank one, so
\begin{align*}
k=1,\qquad N=0,\qquad I=W,\qquad Q=0,\qquad d_{Q,j}=0.
\end{align*}

In the second case, writing locally
$\gamma=f_j(z)\frac{dz}{z}$
in compatible frames, one has
\begin{align*}
\ell_{x_j}(I/I^0)=\operatorname{ord}_{x_j}(f_j),\quad r_j^{\mathrm{full}}=\begin{cases}1,&f_j(0)\neq0,\\0,&f_j(0)=0.\end{cases}
\end{align*}
Thus the saturation length is simply the vanishing order of the local
coefficient of $\gamma$, while the ordinary residue rank only records
whether this coefficient vanishes at $x_j$.

We explain the identity and unipotent cases in the table below.
Put $M_j=\phi(c_j)^{-1}$, the positively oriented puncture monodromy.
In both cases its elliptic factor is the identity, so all parabolic
weights are zero and $\operatorname{Gr}_{0,j}(E)=E|_{x_j}$.
Write $\operatorname{GrRes}_{x_j}(\Phi)=S_j+N_j$.
In our correspondence the local-system weights are zero, so
$S_j=\tau(S_j)$, where $\tau$ is the compact conjugation.
The local monodromy formula is
$M_j\sim\exp(-4\pi\mathrm{i}S_j)
\exp\bigl(2\pi\mathrm{i}(N_j-H_j-X_j)\bigr)$;
see \cite[Proposition~6.3, Table~1, and Proposition~6.8]{BGPMiR20}.
Here $\sim$ denotes conjugacy, and $(H_j,X_j,N_j)$ is a compatible
$\mathfrak{sl}_2$-triple in the centralizer of $S_j$;
when $N_j=0$, we take $H_j=X_j=0$.
The first factor is hyperbolic, with Hermitian exponent
$-4\pi\mathrm{i}S_j$, while the second is unipotent since
$N_j-H_j-X_j=\operatorname{Ad}(e^{-X_j})N_j$.
Consequently, identity monodromy gives $S_j=N_j=0$, whereas
nontrivial unipotent monodromy gives $S_j=0$ and $N_j\neq0$.
Thus $G_j=E|_{x_j}$ in both cases.

Choose a local coordinate $z$ centered at $x_j$, and write
$E_j=V_j\oplus W_j$ for the stalk of $E$.
The map $q_j:E_j\to G_j$ is then evaluation at $x_j$.
For identity monodromy, the centered weight filtration of $N_j=0$
satisfies $W_0G_j=G_j$ and $W_{-2}G_j=0$.
The defining conditions therefore give
$\mathcal A_{x_j}=E_j$ and $\mathcal B_{x_j}=zE_j$.
In particular, $\mathcal B_{x_j}$ is not the zero stalk:
its sections are required to vanish at $x_j$.
Hence $(\mathcal A_V)_j=V_j$ and $(\mathcal B_W)_j=zW_j$,
giving $(a_j,b_j)=(0,1)$, since
$\operatorname{length}_{\mathcal O_j}(W_j/zW_j)=1$.

For nontrivial unipotent monodromy, $N_j$ is a nonzero nilpotent
operator on the two-dimensional space $G_j$.
Choose a basis $v_+,v_-$ with $N_jv_+=v_-$ and $N_jv_-=0$.
Its centered weight filtration is
$W_\ell G_j=0$ for $\ell\leq-2$,
$W_{-1}G_j=W_0G_j=\mathbb C v_-=\operatorname{Im}N_j$,
and $W_\ell G_j=G_j$ for $\ell\geq1$.
The only nonzero graded pieces occur in degrees $1$ and $-1$,
and $N_j$ induces an isomorphism between them.
These degrees are the monodromy filtration indices, distinct from
the parabolic weights, which remain zero.

If $N_j$ maps $V|_{x_j}$ nontrivially to $W|_{x_j}$, then
$W_0G_j=W|_{x_j}$.
Thus a section belongs to $\mathcal A_{x_j}$ exactly when its
$V$-component vanishes at $x_j$, giving
$\mathcal A_{x_j}=zV_j\oplus W_j$.
Since $W_{-2}G_j=0$, we still have $\mathcal B_{x_j}=zE_j$,
and hence $(a_j,b_j)=(1,1)$.
If instead $N_j$ maps $W|_{x_j}$ nontrivially to $V|_{x_j}$,
then $\mathcal A_{x_j}=V_j\oplus zW_j$, while
$\mathcal B_{x_j}=zE_j$ is unchanged.
This gives $(a_j,b_j)=(0,1)$.

For the usual $SU(1,1)$ boundary types, the lattice corrections and the
dimension $h_{\partial,j}=\dim\ker(\phi(c_j)-I)$ are as follows:
\[
\begin{array}{c|c|c|c}
\text{local type}&a_j&b_j&h_{\partial,j}\\
\hline
\text{no weight-zero block}&0&0&0\\
\text{weight zero, invertible semisimple residue}&0&0&0\\
\phi(c_j)=I&0&1&2\\
\text{nontrivial unipotent, }N_j\colon V\to W&1&1&1\\
\text{nontrivial unipotent, }N_j\colon W\to V&0&1&1
\end{array}
\]
Here ``no weight-zero block'' includes, for example, a nontrivial elliptic boundary holonomy.
The second row describes the positive-trace hyperbolic case. A negative-trace hyperbolic element has elliptic factor $-I$ and hence belongs to the first row, since its standard parabolic weight is nonzero. 
A parabolic element with eigenvalue $-1$ also belongs to the first row.

The global fixed-section terms have the elementary descriptions
\begin{align*}
h_0(\phi)=\dim\bigcap_{\eta\in\pi_1(\Sigma)}\ker\bigl(\phi(\eta)-I\bigr)
\end{align*}
and
\begin{align*}
h_\partial(\phi)=\sum_{j=1}^n\dim\ker\bigl(\phi(c_j)-I\bigr).
\end{align*}
Thus, for an irreducible $U(1,1)$-representation, $h_0(\phi)=0$, while
each boundary contribution is 0, 1, or 2.

Finally,
\begin{align*}
h^-(\phi)=\dim\mathbb H^1\left(\overline\Sigma,\mathcal C_\gamma^\bullet\right).
\end{align*}
If $\gamma\neq0$, this morphism of line bundles is injective as a map of
sheaves, even at its zeros. Its cokernel is a torsion sheaf, and therefore
\begin{align*}
h^-(\phi)=\operatorname{length}\operatorname{coker}\left(\mathcal A_V\xrightarrow{\gamma}\mathcal B_W\otimes K_{\overline\Sigma}(D)\right).
\end{align*}
If $\gamma=0$, then
\begin{align*}
h^-(\phi)=h^1(\mathcal A_V)+h^0\bigl(\mathcal B_W\otimes K_{\overline\Sigma}(D)\bigr).
\end{align*}
In either case, \eqref{eq:hminus-h0-Euler-gamma} and Riemann--Roch give
\begin{align*}
h^-(\phi) - h_0(\phi) = \deg W - \deg V + \sum_{j=1}^n (a_j + 1-b_j) + 2(g - 1).
\end{align*}

These examples separate three kinds of local information. The number
$r_j^{\mathrm{full}}$ records whether the local coefficient of $\gamma$
vanishes, and $\ell_{x_j}(I/I^0)$ records its vanishing order. The
corrections $a_j$ and $b_j$ instead come from the zero-weight nilpotent
part of the graded residue.

\subsection{Upper bound on the ordinary residue rank}

\begin{lemma}
\label{lem:full-residue-lattice-inequality}
For every \(j\), one has
\begin{equation}
\label{eq:full-residue-lattice-inequality}
r_j^{\mathrm{full}} \le a_j+q-b_j- d_{Q,j}.
\end{equation}
\end{lemma}

\begin{proof}
Fix a marked point \(x_j\), and choose a local coordinate \(z\)
centered at \(x_j\). Let
\(\mathcal O_j=\mathcal O_{\overline\Sigma,x_j}\), and write
\(V_j\) and \(W_j\) for the stalks of \(V\) and \(W\) at \(x_j\).
Their elements are germs of local holomorphic sections. Evaluation
at \(x_j\) identifies the fibers with
\(
V|_{x_j}=V_j/zV_j,
\) \(
W|_{x_j}=W_j/zW_j.
\)

The inclusions
\(V(-D)\subseteq\mathcal A_V\subseteq V\) and
\(W(-D)\subseteq\mathcal B_W\subseteq W\) imply
\[
zV_j\subseteq(\mathcal A_V)_j\subseteq V_j,
\qquad
zW_j\subseteq(\mathcal B_W)_j\subseteq W_j.
\]
Since \(zV_j\subseteq(\mathcal A_V)_j\), two germs of sections
of \(V\) with the same value at \(x_j\) either both belong to
\((\mathcal A_V)_j\) or neither does. The same holds for
\((\mathcal B_W)_j\subseteq W_j\).
The values at \(x_j\) of germs belonging to these lattices
form the following subspaces of the fibers:
\[
\overline A_j
:=(\mathcal A_V)_j/zV_j\subseteq V|_{x_j},
\qquad
P_j
:=(\mathcal B_W)_j/zW_j\subseteq W|_{x_j}.
\]
Both \(V_j/(\mathcal A_V)_j\) and
\(W_j/(\mathcal B_W)_j\) are annihilated by \(z\), so their
\(\mathcal O_j\)-lengths equal their complex dimensions.
Consequently, the definitions of \(a_j\) and \(b_j\) give
\[
\operatorname{codim}_{V|_{x_j}}\overline A_j=a_j,
\qquad
\dim_{\mathbb C}P_j=q-b_j.
\]

Concretely, choose a local holomorphic frame
\(w_1,\ldots,w_q\) of \(W\) such that
\(w_1(x_j),\ldots,w_{q-b_j}(x_j)\) form a basis of \(P_j\).
In this frame,
\[
W_j
=\bigoplus_{\ell=1}^{q}\mathcal O_j w_\ell,\quad 
(\mathcal B_W)_j
=
\bigoplus_{\ell=1}^{q-b_j}\mathcal O_j w_\ell
\;\oplus\!
\bigoplus_{\ell=q-b_j+1}^{q}z\mathcal O_j w_\ell.
\]
In other words, the last \(b_j\) components of a section in
\((\mathcal B_W)_j\) must vanish at \(x_j\), whereas the first
\(q-b_j\) components are unrestricted. This explains why the
space of its possible values has dimension \(q-b_j\).

We next examine what happens to these values under the quotient
map \(W\to Q=W/I\). Since \(I\) is saturated in \(W\), the
quotient \(Q\) is locally free. Hence the sequence of fibers
\[
0\longrightarrow I|_{x_j}
\longrightarrow W|_{x_j}
\xrightarrow{\;\pi_j\;}Q|_{x_j}
\longrightarrow0
\]
is exact, and we may regard \(I|_{x_j}\) as a subspace of
\(W|_{x_j}\). By definition,
\(P_{Q,j}=\pi_j(P_j)\) and
\(d_{Q,j}=\dim_{\mathbb C}P_{Q,j}\).
The vectors of \(P_j\) that disappear under this projection
are precisely those belonging to \(P_j\cap I|_{x_j}\).
Thus
\[
0\longrightarrow P_j\cap I|_{x_j}
\longrightarrow P_j
\xrightarrow{\;\pi_j\;}P_{Q,j}
\longrightarrow0
\]
is exact, and therefore
\[
\dim_{\mathbb C}\bigl(P_j\cap I|_{x_j}\bigr)
=q-b_j-d_{Q,j}.
\]

Finally, write \(\gamma=f_j(z)\,dz/z\) in local holomorphic
frames. Its full ordinary residue is the linear map
\[
\widetilde R_j^{\mathrm{full}}
=f_j(0)\colon V|_{x_j}\longrightarrow W|_{x_j}.
\]
It factors as
\[
V|_{x_j}
\twoheadrightarrow(V/N)|_{x_j}
\xrightarrow{\;R_j^{\mathrm{full}}\;}
I|_{x_j}
\hookrightarrow W|_{x_j}.
\]
The first map is surjective and the last is injective, so
\(\operatorname{rank}\widetilde R_j^{\mathrm{full}}
=r_j^{\mathrm{full}}\).

The inclusion
\(\gamma(\mathcal A_V)\subseteq
\mathcal B_W\otimes K_{\overline\Sigma}(D)\)
means that
\(f_j((\mathcal A_V)_j)\subseteq(\mathcal B_W)_j\).
Evaluating at \(x_j\), and recalling that the residue takes
values in \(I|_{x_j}\), we obtain
\[
\widetilde R_j^{\mathrm{full}}(\overline A_j)
\subseteq P_j\cap I|_{x_j}.
\]
Now extend a basis of \(\overline A_j\) to a basis of
\(V|_{x_j}\). Since \(\overline A_j\) has codimension \(a_j\),
this adds exactly \(a_j\) vectors, whose images can increase
the residue rank by at most \(a_j\). Hence
\[
\begin{aligned}
r_j^{\mathrm{full}}
&\leq
\dim_{\mathbb C}
\widetilde R_j^{\mathrm{full}}(\overline A_j)+a_j\\
&\leq
\dim_{\mathbb C}\bigl(P_j\cap I|_{x_j}\bigr)+a_j\\
&=a_j+q-b_j-d_{Q,j}.
\end{aligned}
\]
This proves the claimed inequality.
\end{proof}

\subsection{Semistability} 

Substituting \eqref{eq:aj+dj-k} into \eqref{eq:sign deg}, and using
$\ell_{\mathrm{int}}(I/I^0)\geq0$, gives
\begin{equation} \label{eq:sign deg d-Q}
\begin{aligned}
\operatorname{sign}(\phi)
&\le 
2k(2g-2)+(p+q)n+2\deg N\\
&\quad-h_\partial(\phi)-2(\deg Q+\sum_{j=1}^{n}d_{Q,j}).
\end{aligned}
\end{equation}
We now use semistability to control the two degree terms.

We use the polystability of the standard-lattice representative. By
Proposition~\ref{prop:normalized-logarithmic-representative}, the integral
Hecke transformation preserves the parabolic degree of every reduction, so
the usual standard-vector-bundle semistability inequalities apply. The
total parabolic degree is zero. The saturated parabolic subbundle
\(N\oplus0\subset V\oplus W\) is \(\Phi\)-invariant because
\(\gamma(N)=0\). Hence semistability gives
\begin{equation}
\label{eq:pardeg-N-nonpositive}
\pardeg(N)\leq0.
\end{equation}
Here the assertion is automatic if \(N=0\).
Similarly, \(V\oplus I\) is a saturated \(\Phi\)-invariant parabolic
subbundle. Indeed,
\(
\gamma(V)
\subset
I^0\otimes K_{\overline\Sigma}(D)
\subset
I\otimes K_{\overline\Sigma}(D),
\)
whereas
\(
\beta(I)
\subset
V\otimes K_{\overline\Sigma}(D).
\)
Therefore
\begin{equation}
\label{eq:pardeg-VI-nonpositive}
\pardeg(V)+\pardeg(I_W)\leq0.
\end{equation}
If \(V\oplus I=V\oplus W\), the left-hand side is zero.

The exact sequence
\(
0\to I
\to W
\to Q
\to0
\)
gives
\[
\pardeg(W)
=
\pardeg(I_W)+\pardeg(Q).
\]
Since \(\pardeg(V)+\pardeg(W)=0\), we obtain
\[
\pardeg(V)+\pardeg(I_W)+\pardeg(Q)=0.
\]
Together with \eqref{eq:pardeg-VI-nonpositive}, this implies
\begin{equation}
\label{eq:pardeg-Q-nonnegative}
\pardeg(Q)\geq0.
\end{equation}
Here $Q$ is used as a quotient parabolic bundle; no inclusion of $Q$ into
$E$ is needed.

\begin{notation}
For a parabolic bundle $F$ with weights in $[0,1)$, let
$\operatorname{wt}_{x_j}(F)$ be the sum of its weights at $x_j$, counted
with multiplicity, and set
\[
\operatorname{wt}(F):=\sum_{j=1}^n\operatorname{wt}_{x_j}(F)\geq0.
\]
Then $\pardeg(F)=\deg(F)+\operatorname{wt}(F)$.
\end{notation}

Because the weights of $N$ are nonnegative,
\begin{align}\label{eq:semistabilityN}
\deg(N)=\pardeg(N)-\operatorname{wt}(N)\le 0
\end{align} 
and similarly $-\deg Q\leq\operatorname{wt}(Q)$. Combining these
inequalities with \eqref{eq:sign deg d-Q} gives
\begin{align*}
\operatorname{sign}(\phi)
\le 
2k(2g-2)+(p+q)n-h_\partial(\phi)+2(\operatorname{wt}(Q)-\sum_{j=1}^n d_{Q,j}).
\end{align*}
We will show that $\sum_jd_{Q,j}\geq\operatorname{wt}(Q)$. Since
$k\leq p$ and $g\geq1$, this already gives the desired bound apart from
the correction $\delta(n,q-p)\leq2$. The next proposition supplies that
correction by combining the quotient degree with the boundary fixed
spaces.

\subsection{The final boundary estimate}

\begin{proposition}\label{prop:residual-quotient-L2-bound}
For a reductive representation in the setting above, one has
\begin{align*}
h_\partial(\phi)+2(\deg(Q)+\sum_{j=1}^{n}d_{Q,j})\ge \delta(n, q-p).
\end{align*}
\end{proposition}
\begin{proof}
For each marked point $x_j$, put
\[
s_{Q,j}
:=
\sum_{\lambda\in(0,1)}
\dim\operatorname{Gr}_{\lambda,x_j}(Q)
=
\dim Q_{j,>0},
\quad
s_Q:=\sum_{j=1}^n s_{Q,j}.
\]
Thus $s_Q$ counts the nonzero parabolic weights of $Q$, with
multiplicity. Here $Q_{j,>0}$ is defined by \eqref{positive-weight flag}.

We first prove $d_{Q,j}\geq s_{Q,j}$.
Recall that
$P_j=(\mathcal B_W)_j/zW_j\subseteq W|_{x_j}$.
The condition defining $\mathcal B$ is imposed only after projection
to the weight-zero graded fiber of $E$. Since the parabolic
decomposition $E=V\oplus W$ respects the flags, every vector in
$W_{j,>0}$ maps to zero under this projection. Any local holomorphic
lift of such a vector therefore belongs to $(\mathcal B_W)_j$.
Consequently,
\(
W_{j,>0}\subseteq P_j.
\)
Let $\pi_j\colon W|_{x_j}\to Q|_{x_j}$ be the quotient map.
By the definition of the quotient parabolic structure,
$\pi_j(W_{j,>0})=Q_{j,>0}$. Hence
\(
Q_{j,>0}\subseteq\pi_j(P_j)=P_{Q,j},
\)
and taking dimensions gives $d_{Q,j}\geq s_{Q,j}$.

The inequality $d_{Q,j}\geq s_{Q,j}$ implies
\begin{equation}
\sum_{j=1}^n d_{Q,j}-\operatorname{wt}(Q)\geq0.
\end{equation}
Indeed, let $\alpha_1,\ldots,\alpha_{s_Q}\in(0,1)$ be the nonzero
parabolic weights of $Q$, counted with multiplicity over all marked
points. Then
\begin{equation}\label{eq:sum-dQj-ge-sQ}
\begin{aligned}
\sum_{j=1}^n d_{Q,j}-\operatorname{wt}(Q)
&=
\sum_{j=1}^n(d_{Q,j}-s_{Q,j})
+\sum_{a=1}^{s_Q}(1-\alpha_a)\\
&\geq
\sum_{a=1}^{s_Q}(1-\alpha_a)
=
s_Q-\operatorname{wt}(Q)
\geq0.
\end{aligned}
\end{equation}
Every term in the first line is nonnegative, and $1-\alpha_a>0$.
Thus $s_Q-\operatorname{wt}(Q)=0$ if and only if $s_Q=0$.
Moreover,
$\sum_jd_{Q,j}-\operatorname{wt}(Q)=0$
if and only if $s_Q=0$ and $d_{Q,j}=0$ for every $j$.

Set $B_Q:=\deg Q+\sum_jd_{Q,j}$. This is an integer. Semistability gives
$\operatorname{pardeg}(Q)\geq0$, so
\[
B_Q
=
\operatorname{pardeg}(Q)
+\sum_{j=1}^n d_{Q,j}-\operatorname{wt}(Q)
\geq0.
\]
Thus $B_Q$ is nonnegative, and equality forces
$\operatorname{pardeg}(Q)=0$, $s_Q=0$, and $d_{Q,j}=0$ for every $j$.

We distinguish two cases. If $s_Q>0$, the weight estimate above is strict,
so $B_Q>0$. Since $B_Q$ is an integer, $B_Q\geq1$. Therefore
\begin{equation}
h_\partial(\phi)
+
2(\deg Q+\sum_{j=1}^n d_{Q,j})
\geq2
\geq\delta(n,q-p).
\end{equation}

If $s_Q=0$, all parabolic weights of $Q$ vanish. The zero-weight estimate
in Lemma~\ref{lem:zero-weight-Q-fixed-vectors} below gives
$\operatorname{rk}Q\leq\dim\ker(\phi(c_j)-I)$ for every $j$.
Summing over the boundary components, we obtain
$h_\partial(\phi)\geq n\operatorname{rk}Q$.
Since $\operatorname{rk}Q=q-k\geq q-p$ and $B_Q\geq0$, it follows that
\begin{equation}\label{eq:Q-bound-zero-weight}
\begin{aligned}
h_\partial(\phi)
+
2(\deg Q+\sum_{j=1}^n d_{Q,j})
&\geq h_\partial(\phi)\geq n\operatorname{rk}Q\\
&\geq n(q-p)\geq\delta(n,q-p).
\end{aligned}
\end{equation}
This proves the proposition.
\end{proof}

For each boundary component \(c_j\), put
\(h_{\partial,j}(\phi):=\dim\ker(\phi(c_j)-I)\), so that
\(h_\partial(\phi)=\sum_jh_{\partial,j}(\phi)\).

\begin{lemma}
\label{lem:zero-weight-Q-fixed-vectors}
If all parabolic weights of \(Q\) at \(x_j\) are zero, then
\begin{equation}
\label{eq:zero-weight-Q-fixed-vectors}
\operatorname{rk}Q\le h_{\partial,j}(\phi).
\end{equation}
\end{lemma}

\begin{proof}
Fix a marked point $x_j$ at which all parabolic weights of $Q$ are zero.
All vector spaces and dimensions in the proof are over $\mathbb C$.

\smallskip
\noindent
\emph{Step 1: factor through the cokernel of the graded residue.}

For a parabolic bundle $F$, let $F_{j,>0}\subseteq F|_{x_j}$ denote
its positive-weight flag step, see \eqref{positive-weight flag}.
Thus
\(
\operatorname{Gr}_{0,x_j}(F)
=
F|_{x_j}/F_{j,>0}.
\)
We use this canonical quotient throughout the argument.

Put
\[
V_{j,0}:=\operatorname{Gr}_{0,x_j}(V),
\quad
W_{j,0}:=\operatorname{Gr}_{0,x_j}(W),
\quad
E_{j,0}:=V_{j,0}\oplus W_{j,0}.
\]
Let $\pi\colon W\to Q=W/I$ be the quotient morphism, and let
$\pi_j\colon W|_{x_j}\to Q|_{x_j}$ be its map on fibers.
Since $Q$ carries the quotient parabolic structure,
\(
\pi_j(W_{j,>0})=Q_{j,>0}.
\)
By hypothesis, $Q_{j,>0}=0$ and
$\operatorname{Gr}_{0,x_j}(Q)=Q|_{x_j}$.
Consequently, $\pi_j$ induces a surjection
\[
\pi_{j,0}:=\operatorname{Gr}_{0,x_j}(\pi)
\colon W_{j,0}\twoheadrightarrow Q|_{x_j}.
\]

Let
\[
T_j
:=
\left.\operatorname{GrRes}_{x_j}(\Phi)\right|_{E_{j,0}}
=
\begin{pmatrix}
0&R_{\beta,j}^{\mathrm{gr}}\\
R_{\gamma,j}^{\mathrm{gr}}&0
\end{pmatrix},
\]
where
$R_{\gamma,j}^{\mathrm{gr}}\colon V_{j,0}\to W_{j,0}$
and
$R_{\beta,j}^{\mathrm{gr}}\colon W_{j,0}\to V_{j,0}$
are the weight-zero graded residue maps.

By the definition of $I$, the image of $\gamma$ is contained in
$I\otimes K_{\overline\Sigma}(D)$. Hence
\[
\bigl(
\pi\otimes\operatorname{id}_{K_{\overline\Sigma}(D)}
\bigr)\circ\gamma=0.
\]
Taking residues at $x_j$ and then passing to the weight-zero associated
graded gives
\[
\pi_{j,0}\circ R_{\gamma,j}^{\mathrm{gr}}=0.
\]
This passage is valid because the residue preserves the parabolic
filtration; the terms that strictly increase the weight disappear on the
associated graded.
Therefore $\pi_{j,0}$ descends to a surjection
\[
\operatorname{coker}R_{\gamma,j}^{\mathrm{gr}}
=
W_{j,0}/\operatorname{Im}R_{\gamma,j}^{\mathrm{gr}}
\twoheadrightarrow Q|_{x_j}.
\]
In particular,
\[
\operatorname{rk}Q
=
\dim Q|_{x_j}
\leq
\dim\operatorname{coker}R_{\gamma,j}^{\mathrm{gr}}.
\]

\smallskip
\noindent
\emph{Step 2: compare with the kernel on the zero-primary summand.}

The off-diagonal form of $T_j$ gives
\[
\operatorname{Im}T_j
=
\operatorname{Im}R_{\beta,j}^{\mathrm{gr}}
\oplus
\operatorname{Im}R_{\gamma,j}^{\mathrm{gr}}
\subseteq V_{j,0}\oplus W_{j,0}.
\]
Consequently,
\[
\operatorname{coker}T_j
\cong
\operatorname{coker}R_{\beta,j}^{\mathrm{gr}}
\oplus
\operatorname{coker}R_{\gamma,j}^{\mathrm{gr}}.
\]
Since $T_j$ is an endomorphism of a finite-dimensional vector space,
its kernel and cokernel have the same dimension. Thus
\[
\operatorname{rk}Q
\leq
\dim\operatorname{coker}R_{\gamma,j}^{\mathrm{gr}}
\leq
\dim\operatorname{coker}T_j
=
\dim\ker T_j.
\]

As in the definition of the $L^2$ complex, write the additive Jordan
decomposition as
\[
T_j=S_j+N_j,
\qquad
[S_j,N_j]=0,
\]
where $S_j$ is semisimple and $N_j$ is nilpotent, and put
$G_j:=\ker S_j$.
For every eigenvalue $\lambda$ of $S_j$, its eigenspace
$E_{j,0}^{\lambda}:=\ker(S_j-\lambda\operatorname{id})$
is preserved by $N_j$, and
\[
T_j|_{E_{j,0}^{\lambda}}
=
\lambda\operatorname{id}
+
N_j|_{E_{j,0}^{\lambda}}.
\]
This operator is invertible when $\lambda\neq0$, because
$N_j|_{E_{j,0}^{\lambda}}$ is nilpotent.
On $G_j$, the operator $T_j$ is simply $N_j|_{G_j}$.
It follows that
\[
\ker T_j=\ker(N_j|_{G_j}),
\qquad
\operatorname{rk}Q\leq\dim\ker(N_j|_{G_j}).
\]

\smallskip
\noindent
\emph{Step 3: identify the corresponding monodromy fixed vectors.}

Recall that $L_j=\phi(c_j)$ is the boundary holonomy and
$M_j=L_j^{-1}$ is the positively oriented puncture monodromy.
The correspondence was chosen with parabolic-local-system weight zero:
$\mathcal S(0,\mathbf C)=\mathcal R(\mathbf C)$; see
\cite[Proposition~7.9 and Theorem~7.10]{BGPMiR20}. By
\cite[Proposition~6.3]{BGPMiR20}, this weight is $s_j-\tau(s_j)$, where
$\tau$ is conjugation with respect to the compact real form. Thus
$s_j=\tau(s_j)$, and the Levi group in that proposition is the whole
group. Its monodromy formula therefore describes the actual conjugacy
class of $M_j$.
More explicitly, the proof of that proposition allows an additional
unipotent term only in negative eigenspaces of
$\operatorname{ad}(s_j-\tau(s_j))$. There are no such eigenspaces here.

We apply that formula to the summand corresponding to $G_j$.
The compact elliptic factor is the identity on this summand because
its parabolic Higgs-bundle weight is zero.
Since $S_j|_{G_j}=0$ and $s_j=\tau(s_j)$, both $s_j$ and $\tau(s_j)$ vanish on the
corresponding summand. Hence the semisimple-residue factor is also
the identity.

Set $Y_{j,0}:=N_j|_{G_j}$.
The local correspondence
\cite[Section~5.1, equation~(6.2), and Proposition~6.3]{BGPMiR20}
therefore gives an $M_j$-invariant summand on which the monodromy
is conjugate to
\[
U_{j,0}=\exp(Z_{j,0}),
\qquad
Z_{j,0}
=
2\pi\sqrt{-1}\,
\bigl(Y_{j,0}-H_{j,0}-X_{j,0}\bigr).
\]
Here $(H_{j,0},X_{j,0},Y_{j,0})$ are the normal
$\mathfrak{sl}_2$ data in the local model.
The normal triple is chosen in the centralizer of the compact factor
and the semisimple residue, so it preserves the summand under
consideration. When $Y_{j,0}=0$, we take
$H_{j,0}=X_{j,0}=0$.

The relations
$[X_{j,0},Y_{j,0}]=H_{j,0}$ and
$[H_{j,0},X_{j,0}]=2X_{j,0}$ imply
\[
Y_{j,0}-H_{j,0}-X_{j,0}
=
e^{-X_{j,0}}Y_{j,0}e^{X_{j,0}}.
\]
Thus $Z_{j,0}$ is nilpotent and
\[
\dim\ker Z_{j,0}=\dim\ker Y_{j,0}.
\]
Furthermore,
\[
\exp(Z_{j,0})-\operatorname{id}
=
Z_{j,0}
\left(
\sum_{a\geq0}\frac{Z_{j,0}^{a}}{(a+1)!}
\right).
\]
The sum is finite because $Z_{j,0}$ is nilpotent, and the operator
in parentheses is invertible and commutes with $Z_{j,0}$.
Therefore
\[
\dim\ker(U_{j,0}-\operatorname{id})
=
\dim\ker Z_{j,0}
=
\dim\ker(N_j|_{G_j}).
\]
Since $U_{j,0}$ represents the restriction of $M_j$ to an invariant
summand, we obtain
\[
\dim\ker(N_j|_{G_j})
\leq
\dim\ker(M_j-\operatorname{id}).
\]
Finally, $M_j=L_j^{-1}$ implies
$\ker(M_j-\operatorname{id})=\ker(L_j-\operatorname{id})$.
Combining the preceding inequalities gives
\[
\operatorname{rk}Q
\leq
\dim\ker(N_j|_{G_j})
\leq
\dim\ker(L_j-\operatorname{id})
=
h_{\partial,j}(\phi),
\]
as claimed.
\end{proof}

We can now conclude the reductive case.

\begin{corollary}\label{cor:global-residue-compensation}
For $g,n\geq1$, every reductive representation
$\phi\colon\pi_1(\Sigma_{g,n})\to U(p,q)$ satisfies
\begin{align*}
\operatorname{sign}(\phi)\le M_{g,n}(p,q)
\end{align*}
This is the upper bound in Theorem~\ref{thm:introduction-main}(3).
\end{corollary}

\begin{proof}
As above, assume $p\leq q$. By \eqref{eq:sign deg d-Q},
\eqref{eq:semistabilityN}, and
Proposition~\ref{prop:residual-quotient-L2-bound},
\[
\operatorname{sign}(\phi)
\leq 2k(2g-2)+n(p+q)-\delta(n,q-p).
\]
Since $k\leq p$ and $2g-2\geq0$, the right-hand side is at most
$M_{g,n}(p,q)$. In particular, this argument includes $g=1$.
Complex conjugation gives the corresponding lower bound. Reversing the
sign of the preserved Hermitian form interchanges $p$ and $q$ and reverses
the signature, so the same upper bound also holds when $p>q$.
\end{proof}
\subsection{The equality case for reductive representations}
\label{subsec:reductive-equality-case}

We now describe what equality in the sharp signature bound implies for a
reductive representation. Throughout this subsection, assume
\(g,n\geq1\) and \(q>p>0\), and write
\[
d=q-p,
\qquad \delta=\delta(n,d)=\min\{2,nd\}.
\]
Then
\(M_{g,n}(p,q)=2p(2g-2)+n(p+q)-\delta\).

Let
\[
(E=V\oplus W,\Phi),
\qquad
\Phi=\begin{pmatrix}0&\beta\\ \gamma&0\end{pmatrix},
\]
be the normalized parabolic Higgs bundle associated with a reductive
representation
\(\phi:\pi_1(\Sigma_{g,n})\to\mathrm U(p,q)\).
We keep the kernel--image notation from the preceding proof:
\[
N=\ker\gamma,
\quad
I^0=\operatorname{im}(\gamma)\otimes
K_{\overline\Sigma}(D)^{-1},
\quad I=\operatorname{Sat}_W(I^0),
\quad Q=W/I,
\]
and \(k=\operatorname{rk}I=\operatorname{rk}\gamma\).

For each marked point, define the local excess
\begin{equation}
\label{eq:local-equality-excess}
\varepsilon_j^{\mathrm{loc}}
:=\ell_{x_j}(I/I^0)+a_j+q-b_j-k-d_{Q,j}.
\end{equation}
The local estimates
\eqref{local estimate length I/I^0} and
\eqref{eq:full-residue-lattice-inequality} give
\[
\ell_{x_j}(I/I^0)\geq k-r_j^{\mathrm{full}},
\qquad
a_j+q-b_j-r_j^{\mathrm{full}}\geq d_{Q,j}.
\]
Thus the excess is a sum of two nonnegative quantities:
\begin{equation}
\label{eq:local-excess-decomposition}
\begin{aligned}
\varepsilon_j^{\mathrm{loc}}
={}&\left[\ell_{x_j}(I/I^0)-(k-r_j^{\mathrm{full}})\right]\\
&+\left[a_j+q-b_j-r_j^{\mathrm{full}}-d_{Q,j}\right]
\geq0.
\end{aligned}
\end{equation}

We also set \(B_Q=\deg Q+\sum_jd_{Q,j}\). The parabolic degree and
weight estimates proved above imply
\begin{equation}
\label{eq:BQ-nonnegative}
B_Q=\operatorname{pardeg}(Q)
+\left(\sum_{j=1}^n d_{Q,j}-\operatorname{wt}(Q)\right)\geq0.
\end{equation}
The boundary estimate used in the proof of the upper bound is
\begin{equation}
\label{eq:boundary-equality-defect}
h_\partial(\phi)+2B_Q\geq\delta.
\end{equation}
These quantities measure exactly how far the signature is from its upper
bound.

\begin{proposition}
\label{prop:exact-reductive-defect}
For every reductive representation \(\phi\) under the assumptions above,
\begin{align}
M_{g,n}(p,q)-\operatorname{sign}(\phi)
={}&2(p-k)(2g-2)-2\deg N
+2\ell_{\mathrm{int}}(I/I^0)\nonumber\\
&+2\sum_{j=1}^n\varepsilon_j^{\mathrm{loc}}
+\bigl(h_\partial(\phi)+2B_Q-\delta\bigr).
\label{eq:exact-reductive-defect}
\end{align}
Every term on the right is nonnegative.
\end{proposition}

\begin{proof}
The degree formula \eqref{eq:sign deg} reads
\begin{align*}
\operatorname{sign}(\phi)
={}&2k(2g-2)+n(p+q)-h_\partial(\phi)
+2\deg N-2\deg Q\\
&-2\ell_{\mathrm{int}}(I/I^0)
-2\sum_{j=1}^n
\bigl(\ell_{x_j}(I/I^0)+a_j+q-b_j-k\bigr).
\end{align*}
Substitute
\(\ell_{x_j}(I/I^0)+a_j+q-b_j-k
=d_{Q,j}+\varepsilon_j^{\mathrm{loc}}\)
and collect the terms involving \(Q\). This gives
\begin{align*}
\operatorname{sign}(\phi)
={}&2k(2g-2)+n(p+q)-h_\partial(\phi)+2\deg N\\
&-2B_Q-2\ell_{\mathrm{int}}(I/I^0)
-2\sum_{j=1}^n\varepsilon_j^{\mathrm{loc}}.
\end{align*}
Subtracting from \(M_{g,n}(p,q)\) proves the identity.

For nonnegativity, use \(k\leq p\) and \(g\geq1\). Semistability gives
\(\operatorname{pardeg}(N)\leq0\), and the normalized weights are
nonnegative, so
\[
\deg N=\operatorname{pardeg}(N)-\operatorname{wt}(N)\leq0.
\]
The saturation length is nonnegative by definition. The remaining terms
are nonnegative by \eqref{eq:local-excess-decomposition} and
\eqref{eq:boundary-equality-defect}.
\end{proof}

The identity gives a complete numerical characterization of equality.

\begin{corollary}
\label{cor:reductive-equality-characterization}
A reductive representation \(\phi\) satisfies
\(\operatorname{sign}(\phi)=M_{g,n}(p,q)\) if and only if all the following
conditions hold:
\begin{enumerate}
\item \((p-k)(2g-2)=0\);
\item \(\deg N=0\);
\item \(\ell_{\mathrm{int}}(I/I^0)=0\);
\item at each marked point \(x_j\),
\[
\ell_{x_j}(I/I^0)=k-r_j^{\mathrm{full}},
\qquad
a_j+q-b_j-r_j^{\mathrm{full}}=d_{Q,j};
\]
\item
\[
h_\partial(\phi)
+2\left(\deg Q+\sum_{j=1}^n d_{Q,j}\right)
=\delta(n,q-p).
\]
\end{enumerate}
\end{corollary}

\begin{proof}
A sum of nonnegative terms is zero exactly when every term is zero.
Apply this observation to \eqref{eq:exact-reductive-defect}.
For the local term, \eqref{eq:local-excess-decomposition} shows that
\(\varepsilon_j^{\mathrm{loc}}=0\) is equivalent to the two equalities
in item~(4).
\end{proof}

\begin{remark}
\label{rem:local-equality-meaning}
The first local equality has a concrete interpretation. After completing
the stalks at \(x_j\), put the coefficient matrix of \(\gamma\), with
respect to the logarithmic frame \(dz/z\), in Smith normal form. Its
nonzero diagonal entries are
\[
z^{\mu_{j,1}},\ldots,z^{\mu_{j,k}},
\qquad \mu_{j,a}\in\mathbb Z_{\geq0}.
\]
Then
\[
\ell_{x_j}(I/I^0)=\sum_{a=1}^k\mu_{j,a},
\qquad
r_j^{\mathrm{full}}=\#\{a:\mu_{j,a}=0\}.
\]
Consequently, the first equality in item~(4) of
Corollary~\ref{cor:reductive-equality-characterization} holds exactly
when every \(\mu_{j,a}\) is either \(0\) or \(1\). In an extremal
representation, each nonzero local vanishing order of the image is
therefore exactly one.
\end{remark}

\begin{remark}
\label{rem:extremal-kernel-defect}
Extremality gives \(\deg N=0\). Since
\[
\operatorname{pardeg}(N)=\deg N+\operatorname{wt}(N)\leq0,
\qquad \operatorname{wt}(N)\geq0,
\]
it follows that both \(\operatorname{pardeg}(N)\) and
\(\operatorname{wt}(N)\) vanish. If \(g\geq2\), the equality
\((p-k)(2g-2)=0\) also gives \(k=p\), hence \(N=0\).
For \(g=1\), this numerical term gives no rank information. However,
if \(N\ne0\), then \(N\oplus0\) is a proper invariant parabolic
subbundle of degree zero. Polystability makes it a Higgs direct summand,
so the associated representation is reducible.
\end{remark}

A zero parabolic degree for \(Q\) gives a more precise splitting.

\begin{lemma}
\label{lem:pardeg-Q-zero-splitting}
If \(\operatorname{pardeg}(Q)=0\), then
\[
(E,\Phi)\simeq
\bigl(V\oplus I,\Phi|_{V\oplus I}\bigr)\oplus(Q,0).
\]
The splitting can be chosen orthogonal for the harmonic metric, with
the summand representing \(Q\) contained in \(W\). In particular,
\(\phi\) is reducible.
\end{lemma}

\begin{proof}
The saturated subbundle \(S=V\oplus I\) is \(\Phi\)-invariant:
\(\gamma(V)\subseteq I\otimes K_{\overline\Sigma}(D)\), and
\(\beta(I)\subseteq V\otimes K_{\overline\Sigma}(D)\).
Additivity of parabolic degree gives
\[
\operatorname{pardeg}(S)=-\operatorname{pardeg}(Q)=0.
\]
The equality case of the harmonic-metric proof of polystability implies
that the second fundamental form of \(S\), and the off-diagonal Higgs
block relative to \(S\oplus S^{\perp_h}\), vanish. Thus this orthogonal
decomposition is a holomorphic parabolic Higgs splitting, preserved by
the associated flat connection.

Because \(V\) and \(W\) are orthogonal and \(S\) contains \(V\), its
complement \(S^{\perp_h}\) lies in \(W\) and maps isomorphically to
\(Q=W/I\). The quotient Higgs field is zero: \(\Phi(W)\) takes values
in \(V\otimes K_{\overline\Sigma}(D)\), which vanishes in the quotient
by \(S\). Hence the complementary summand is \((Q,0)\).
Both summands have positive rank, since \(p>0\) and
\(\operatorname{rk}Q=q-k\geq q-p>0\). The flat splitting is therefore
nontrivial.
\end{proof}

\begin{corollary}
\label{cor:pardeg-Q-group-reduction}
If \(\operatorname{pardeg}(Q)=0\), then, after conjugation in
\(\mathrm U(p,q)\),
\[
\phi=\phi_1\oplus\phi_Q,
\qquad
\phi_1:\pi_1(\Sigma_{g,n})\to\mathrm U(p,k),
\qquad
\phi_Q:\pi_1(\Sigma_{g,n})\to\mathrm U(q-k),
\]
where \(\phi_Q\) has compact image. If also
\(\operatorname{sign}(\phi)=M_{g,n}(p,q)\), then \(k=p\), so the image
is conjugate into \(\mathrm U(p,p)\times\mathrm U(q-p)\), and
\[
\operatorname{sign}(\phi_1)=M_{g,n}(p,p),
\qquad
\operatorname{sign}(\phi_Q)=M_{g,n}(0,q-p).
\]
\end{corollary}

\begin{proof}
By Lemma~\ref{lem:pardeg-Q-zero-splitting}, the harmonic splitting has
one summand of signature \((p,k)\) and one negative definite summand
of rank \(q-k\). The flat connection preserves both summands, giving
the asserted group reduction. On \(Q\), the Higgs field vanishes, so
the flat connection is unitary.

Suppose the signature is maximal. Additivity and the upper bounds for
the two summands give
\[
M_{g,n}(p,q)
\leq M_{g,n}(p,k)+M_{g,n}(0,q-k).
\]
If \(k<p\), direct substitution in the formula for \(M_{g,n}\) gives
\begin{align*}
&M_{g,n}(p,q)-M_{g,n}(p,k)-M_{g,n}(0,q-k)\\
&\qquad=2(p-k)(2g-2)+\delta(n,p-k)
+\delta(n,q-k)-\delta(n,q-p).
\end{align*}
Here \(p-k\geq1\) and \(q-k\geq2\). Thus
\(\delta(n,p-k)\geq1\), \(\delta(n,q-k)=2\), and
\(\delta(n,q-p)\leq2\). Since \(g\geq1\), the displayed difference
is at least one, a contradiction. Therefore \(k=p\).

Finally,
\[
M_{g,n}(p,p)+M_{g,n}(0,q-p)=M_{g,n}(p,q).
\]
The signatures of the summands add to the right-hand side and are
bounded by the two terms on the left. Both individual bounds must
therefore be attained.
\end{proof}

The boundary equality in
Corollary~\ref{cor:reductive-equality-characterization} leaves very few
possibilities.

\begin{corollary}
\label{cor:extremal-boundary-alternatives}
Assume \(\operatorname{sign}(\phi)=M_{g,n}(p,q)\).
\begin{enumerate}
\item If \((n,q-p)=(1,1)\), then
\[
h_\partial(\phi)=1,
\qquad B_Q=0,
\qquad k=p,
\qquad N=0,
\qquad \operatorname{rk}Q=1.
\]
All parabolic weights of \(Q\) vanish, and
\(\deg Q=\operatorname{pardeg}(Q)=d_{Q,1}=0\).
In particular, \(\phi\) is reducible.

\item If \(n(q-p)\geq2\), then
\[
(h_\partial(\phi),B_Q)=(2,0)
\quad\text{or}\quad
(h_\partial(\phi),B_Q)=(0,1).
\]
In the first case, necessarily
\[
n(q-p)=2,
\qquad k=p,
\qquad N=0,
\qquad \operatorname{rk}Q=q-p.
\]
All parabolic weights of \(Q\) vanish, and
\[
\deg Q=\operatorname{pardeg}(Q)=0,
\qquad d_{Q,j}=0,
\qquad h_{\partial,j}(\phi)=\operatorname{rk}Q
\quad\text{for every }j.
\]
Consequently, \(\phi\) is reducible.
In the second case, \(Q\) has at least one nonzero parabolic weight;
the equality \(B_Q=1\) alone does not force
\(\operatorname{pardeg}(Q)=0\).
\end{enumerate}
\end{corollary}

\begin{proof}
Both \(B_Q\) and \(h_\partial(\phi)\) are nonnegative integers, and
extremality gives
\[
h_\partial(\phi)+2B_Q=\delta(n,q-p).
\]
If the right-hand side is one, the only possibility is
\((h_\partial(\phi),B_Q)=(1,0)\). If it is two, the possibilities are
\((2,0)\) and \((0,1)\).

Whenever \(B_Q=0\), the two nonnegative terms in
\eqref{eq:BQ-nonnegative} vanish. The equality criterion in
\eqref{eq:sum-dQj-ge-sQ} then gives
\[
\operatorname{pardeg}(Q)=\deg Q=0,
\qquad \operatorname{wt}(Q)=0,
\qquad d_{Q,j}=0\quad\text{for every }j.
\]
The zero-weight estimate \eqref{eq:zero-weight-Q-fixed-vectors} implies
\[
h_\partial(\phi)\geq n\operatorname{rk}Q\geq n(q-p).
\]
For \((n,q-p)=(1,1)\), the left-hand side is one, so
\(\operatorname{rk}Q=1\) and \(k=p\). In the case \(\delta=2\) and
\(B_Q=0\), we have
\[
2=h_\partial(\phi)\geq n\operatorname{rk}Q
\geq n(q-p)\geq2.
\]
Every inequality is an equality. This gives \(n(q-p)=2\), \(k=p\),
and \(h_{\partial,j}(\phi)=\operatorname{rk}Q\) at each boundary
component. In both cases, \(N=0\), and reducibility follows from
Lemma~\ref{lem:pardeg-Q-zero-splitting}.

Finally, if \((h_\partial(\phi),B_Q)=(0,1)\) and all weights of \(Q\)
were zero, the same zero-weight estimate would give
\(h_\partial(\phi)\geq n\operatorname{rk}Q>0\), a contradiction.
\end{proof}

For stable Higgs bundles, only the last alternative can occur.

\begin{corollary}
\label{cor:stable-extremal-Higgs-bundles}
Assume that the associated parabolic Higgs bundle is stable; this holds,
in particular, when \(\phi\) is irreducible. If
\(\operatorname{sign}(\phi)=M_{g,n}(p,q)\), then
\[
k=p,
\qquad N=0,
\qquad \operatorname{pardeg}(Q)>0.
\]
Moreover,
\[
\delta(n,q-p)=2,
\qquad h_\partial(\phi)=0,
\qquad \deg Q+\sum_{j=1}^n d_{Q,j}=1.
\]
In particular, no stable extremal Higgs bundle exists when
\((n,q-p)=(1,1)\).
\end{corollary}

\begin{proof}
Extremality gives \(\deg N=\operatorname{pardeg}(N)=0\).
If \(N\ne0\), the proper invariant subbundle \(N\oplus0\) has the
same slope as \(E\), contrary to stability. Hence \(N=0\) and \(k=p\).
Since \(\operatorname{rk}Q=q-p>0\), the invariant subbundle \(V\oplus I\)
is proper. Stability gives
\[
-\operatorname{pardeg}(Q)=\operatorname{pardeg}(V\oplus I)<0.
\]
This excludes every alternative with \(B_Q=0\) in
Corollary~\ref{cor:extremal-boundary-alternatives}. The remaining
alternative has \(\delta=2\), \(h_\partial(\phi)=0\), and \(B_Q=1\).
\end{proof}

\section{From reductive to arbitrary representations}\label{non-reductive}

\subsection{Isotropic filtrations and reductive splittings}
\label{subsec:isotropic-reduction-positive}

We now extend the upper bound to representations that are not reductive.
By the definition of reductivity used above, the image of such a
representation lies in a proper real parabolic subgroup. The parabolic
subgroups of \(\mathrm{U}(p,q)\) stabilize isotropic flags; see
\cite[Chapter~VII, Section~7]{Knapp}. Thus the representation preserves a
nonzero isotropic subspace.

Taking an isotropic quotient reduces both indices of the Hermitian form by
the same amount. The difficulty is that this operation can change the
signature: the discarded extension data may contribute to the boundary rho
invariant. We first explain when the signature is preserved, then estimate
the possible change in general.

\begin{lemma}
\label{lem:isotropic-quotient-positive}
Let \((\mathcal E,\Omega)\) be a flat Hermitian local system of signature
\((p,q)\) over \(\Sigma\), and suppose that
\(\mathcal F\subset\mathcal E\) is a flat isotropic sub-local system of
complex rank \(s\). Then \(\mathcal F^\perp\) is a flat sub-local system, and
the quotient
\begin{equation}
\label{eq:isotropic-quotient-positive}
\mathcal E_0
:=
\mathcal F^\perp/\mathcal F
\end{equation}
inherits a nondegenerate flat Hermitian form \(\Omega_0\) of signature
\((p-s,q-s)\). In particular,
\(
(q-s)-(p-s)=q-p.
\)
\end{lemma}

\begin{proof}
Since the holonomy of \(\mathcal E\) preserves both \(\Omega\) and
\(\mathcal F\), it also preserves the orthogonal complement
\(\mathcal F^\perp\), which is therefore a flat subbundle. Moreover, the
isotropy of \(\mathcal F\) implies
\(\mathcal F\subseteq\mathcal F^\perp\).
Recall that the radical of a Hermitian form \(B\) on a vector space \(U\) is
\(
\operatorname{rad}(B)
:=
\{u\in U:B(u,v)=0\text{ for every }v\in U\}.
\)
Fiberwise, the nondegeneracy of \(\Omega\) gives
\((\mathcal F^\perp)^\perp=\mathcal F\), and hence
\(
\operatorname{rad}\!\left(\Omega|_{\mathcal F^\perp}\right)
=
\mathcal F^\perp\cap(\mathcal F^\perp)^\perp
=
\mathcal F.
\)
Consequently, \(\Omega|_{\mathcal F^\perp}\) descends to a well-defined flat
Hermitian form on
\(\mathcal E_0:=\mathcal F^\perp/\mathcal F\), given by
\(
\Omega_0([u],[v]):=\Omega(u,v).
\)
The equality of the radical with \(\mathcal F\) shows that
\(\Omega_0\) is nondegenerate.

Fiberwise, choose an isotropic subspace \(F^{-}\) opposite to \(F\), paired
nondegenerately with \(F\), and a nondegenerate complement \(E_0\). This
gives a Witt decomposition
\[
(E,\Omega)
\cong
(F\oplus F^{-},\Omega_{\mathrm{hyp}})
\perp
(E_0,\Omega_0).
\]
The hyperbolic summand contributes \(s\) positive and \(s\) negative
directions. Hence \(\Omega_0\) has signature \((p-s,q-s)\).
\end{proof}

The signature \((p-s,q-s)\) in this lemma is the signature of the
Hermitian form on each fiber. It should not be confused with the signature
of the twisted intersection form on the surface. The latter need not be
preserved, as the following example shows.

\begin{remark}
\label{rem:failure-naive-isotropic-reduction}
Let \(\Sigma=\Sigma_{1,1}\) be the once-punctured torus, with standard
generators \(a,b,c\) satisfying \([a,b]c=1\). Fix \(\lambda>1\) and \(t>0\),
and define
\[
A=
\begin{pmatrix}
\lambda & 0\\
0 & \lambda^{-1}
\end{pmatrix},
\qquad
B=
\begin{pmatrix}
\lambda^{-1} & t\\
0 & \lambda
\end{pmatrix},
\qquad
C=[A,B]^{-1}.
\]
The assignment
\(\phi(a)=A\), \(\phi(b)=B\), and \(\phi(c)=C\) defines a representation
into \(\mathrm{SL}(2,\mathbb R)\cong\mathrm{SU}(1,1)\). Its image is
contained in a Borel subgroup and hence preserves an isotropic line. Since
this line is maximal isotropic, the quotient
\(\mathcal F^\perp/\mathcal F\) is zero-dimensional.

Writing \(u=t(\lambda-\lambda^{-1})>0\), a direct multiplication gives
\[
[A,B]=\begin{pmatrix}1&u\\0&1\end{pmatrix},
\qquad
C=\begin{pmatrix}1&-u\\0&1\end{pmatrix}.
\]
The image lies in an amenable Borel subgroup, so its Toledo invariant is
zero. The signed parabolic formula of \cite[Subsection~6.2]{KPWValues} gives
\(\rho(C)=1\). Hence the signature--Toledo formula gives
\(\operatorname{sign}(\phi)=1\).
The isotropic quotient has signature zero, so in this example
\[
\operatorname{sign}(\mathcal E,\Omega)
\neq
\operatorname{sign}
\bigl(\mathcal F^\perp/\mathcal F,\Omega_0\bigr).
\]
The entire discrepancy comes from the boundary rho invariant.
\end{remark}

A flat hyperbolic splitting removes this difficulty: the two isotropic
summands cancel in the intersection form.

\begin{lemma}
\label{lem:split-isotropic-reduction-positive}
Let \((\mathcal E,\Omega)\) be a flat Hermitian local system over
\(\Sigma\). Suppose that there is an orthogonal decomposition of flat local
systems
\begin{equation}
\label{eq:flat-hyperbolic-splitting}
(\mathcal E,\Omega)
\cong
(\mathcal F\oplus\mathcal F^{-},\Omega_{\mathrm{hyp}})
\perp
(\mathcal E_0,\Omega_0),
\end{equation}
where \(\mathcal F\) and \(\mathcal F^{-}\) are flat isotropic local systems
paired nondegenerately by \(\Omega_{\mathrm{hyp}}\). Then
\begin{equation}
\label{eq:split-isotropic-signature}
\operatorname{sign}(\mathcal E,\Omega)
=
\operatorname{sign}(\mathcal E_0,\Omega_0).
\end{equation}
\end{lemma}

\begin{proof}
By additivity of the twisted intersection form under orthogonal direct sums,
it is enough to show that the hyperbolic summand
\(\mathcal H:=\mathcal F\oplus\mathcal F^{-}\) has signature zero.

Consider the flat involution
\[
\tau\colon
\mathcal F\oplus\mathcal F^{-}
\longrightarrow
\mathcal F\oplus\mathcal F^{-},
\qquad
\tau(v,w)=(v,-w).
\]
Since both \(\mathcal F\) and \(\mathcal F^{-}\) are isotropic and are
paired nondegenerately with one another, one has
\(\tau^*\Omega_{\mathrm{hyp}}=-\Omega_{\mathrm{hyp}}\). The induced
automorphism of
\[
\widehat H^1(\Sigma;\mathcal H)
=\operatorname{Im}\!\left(
H^1(\Sigma,\partial\Sigma;\mathcal H)
\longrightarrow H^1(\Sigma;\mathcal H)
\right)
\]
is therefore an anti-isometry of the Hermitian intersection form:
\[
iQ_{\mathrm{hyp}}(\tau_*x,\tau_*y)
=
-iQ_{\mathrm{hyp}}(x,y).
\]
It follows that \(\tau_*\) exchanges the positive and negative subspaces.
Consequently,
\(\operatorname{sign}(\mathcal F\oplus\mathcal F^{-},
\Omega_{\mathrm{hyp}})=0\), and
\eqref{eq:split-isotropic-signature} follows.
\end{proof}

For reductive representations, the required flat splitting is supplied by
a Levi subgroup.

\begin{corollary}
\label{cor:reductive-isotropic-reduction-positive}
Let
\(\phi\colon\pi_1(\Sigma)\to\mathrm{U}(p,q)\) be reductive, and suppose that
its image preserves an isotropic subspace \(F\subset E\) of dimension \(s\).
Then the associated local system admits a flat orthogonal decomposition
\[
(\mathcal E,\Omega)
\cong
(\mathcal F\oplus\mathcal F^{-},\Omega_{\mathrm{hyp}})
\perp
(\mathcal E_0,\Omega_0),
\qquad
\mathcal E_0\cong\mathcal F^\perp/\mathcal F,
\]
where \((\mathcal E_0,\Omega_0)\) has signature
\((p-s,q-s)\). In particular,
\[
\operatorname{sign}(\phi)
=
\operatorname{sign}(\mathcal E_0,\Omega_0).
\]
\end{corollary}

\begin{proof}
Let \(P_F\subset\mathrm{U}(p,q)\) be the parabolic subgroup stabilizing
\(F\). Since \(H_\phi=\overline{\phi(\pi_1(\Sigma))}^{\,\mathrm{Zar}}\subset P_F\) and \(\phi\) is reductive,
\(H_\phi\) is contained in a Levi subgroup \(L_F\subset P_F\). Such a Levi
subgroup preserves an isotropic subspace \(F^{-}\) opposite to \(F\), as well
as a nondegenerate complement \(E_0\), and hence yields a decomposition
\[
(E,\Omega)
\cong
(F\oplus F^{-},\Omega_{\mathrm{hyp}})
\perp
(E_0,\Omega_0).
\]
Because \(H_\phi\subset L_F\), this decomposition passes to the associated
flat local systems. The conclusion now follows from
Lemma~\ref{lem:split-isotropic-reduction-positive}.
\end{proof}

\begin{remark}
\label{rem:no-general-reductive-reduction}
Corollary~\ref{cor:reductive-isotropic-reduction-positive} concerns a flat
splitting. It does not assert that an arbitrary isotropic quotient preserves
the signature. In the nonreductive case, we must still account for the
boundary contribution illustrated in
Remark~\ref{rem:failure-naive-isotropic-reduction}. The next subsection gives
the required estimate.
\end{remark}

\subsection{Controlled degeneration to a Levi factor}
\label{subsec:controlled-levi-degeneration}

The key construction is a family of conjugate representations in which the
off-diagonal extension blocks tend to zero. The limit preserves a flat
hyperbolic splitting. The Toledo invariant is continuous along this
degeneration, whereas the signature may jump. We will bound that jump by
estimating the change in the boundary rho invariants.

Let \(F\subset E\) be a \(\phi\)-invariant isotropic subspace of dimension
\(s\). Choose an isotropic subspace \(F^{-}\) opposite to \(F\), and let
\(E_0=(F\oplus F^{-})^\perp\). We then have a Witt decomposition
\begin{equation}
\label{eq:witt-decomposition-levi}
E=F\oplus E_0\oplus F^{-},
\qquad
\Omega=
\begin{pmatrix}
0&0&I_s\\
0&\Omega_0&0\\
I_s&0&0
\end{pmatrix},
\end{equation}
where \((E_0,\Omega_0)\) has signature \((p-s,q-s)\).

Let \(P_F\subset\mathrm{U}(E,\Omega)\) be the parabolic subgroup stabilizing
\(F\). With respect to \eqref{eq:witt-decomposition-levi}, every
\(L\in P_F\) can be written in the form
\begin{equation}
\label{eq:parabolic-block-form-levi}
L=
\begin{pmatrix}
A&X&Z\\
0&L_0&Y\\
0&0&(A^*)^{-1}
\end{pmatrix},
\end{equation}
where \(L_0\in\mathrm{U}(E_0,\Omega_0)\), and the off-diagonal blocks satisfy
the relations imposed by \(L^*\Omega L=\Omega\). This is the standard Levi
decomposition of the stabilizer of an isotropic subspace; see, for example,
\cite[Chapter~VII, Section~7]{Knapp}.

For \(t>0\), define
\[
a_t
:=
\operatorname{diag}
\bigl(tI_F,I_{E_0},t^{-1}I_{F^{-}}\bigr).
\]
Then \(a_t\in\mathrm{U}(E,\Omega)\), and
\begin{equation}
\label{eq:levi-contraction}
a_tLa_t^{-1}
=
\begin{pmatrix}
A&tX&t^2Z\\
0&L_0&tY\\
0&0&(A^*)^{-1}
\end{pmatrix}.
\end{equation}
Consequently, the limit
\[
L^{\mathrm{gr}}
:=
\lim_{t\to 0}a_tLa_t^{-1}
=
\operatorname{diag}
\bigl(A,L_0,(A^*)^{-1}\bigr)
\]
exists and belongs to a Levi subgroup of \(P_F\).

Applying this contraction simultaneously to all elements in the image of
\(\phi\), we obtain a family of representations
\(\phi_t:=a_t\phi a_t^{-1}\), \(t>0\), converging pointwise to a
representation
\begin{equation}
\label{eq:graded-Levi-representation}
\phi^{\mathrm{gr}}
=
\alpha\oplus\phi_0\oplus\alpha^{-*},
\end{equation}
where
\[
\phi_0\colon
\pi_1(\Sigma)\longrightarrow
\mathrm{U}(p-s,q-s),
\qquad
\alpha^{-*}(\gamma)
:=
\bigl(\alpha(\gamma)^*\bigr)^{-1}.
\]
We refer to \(\phi^{\mathrm{gr}}\) as the Levi degeneration of \(\phi\)
associated with \(F\). Here ``Levi'' describes the subgroup containing the
limit; the representations \(\alpha\) and \(\phi_0\) need not be reductive.

\begin{lemma}
\label{lem:Levi-Jordan-compatibility}
Let \(\pi_F\colon P_F\to L_F\) be the algebraic projection onto the Levi
subgroup determined by the Witt decomposition
\eqref{eq:witt-decomposition-levi}. Thus
\(L^{\mathrm{gr}}=\pi_F(L)\). Then:

\begin{enumerate}
\item
The algebraic Jordan decomposition is preserved by \(\pi_F\). In particular,
if \(L=L_sL_u=L_uL_s\), then
\[
L^{\mathrm{gr}}
=
\pi_F(L_s)\pi_F(L_u),
\]
where \(\pi_F(L_s)\) and \(\pi_F(L_u)\) are respectively the semisimple and
unipotent parts of \(L^{\mathrm{gr}}\).

\item
For every eigenvalue \(\lambda\), the \(\lambda\)-primary component of
\(L^{\mathrm{gr}}\) is the associated graded of the filtration induced by
\(0\subset F\subset F^\perp\subset E\) on the \(\lambda\)-primary component
of \(L\).

\item
If \(|\lambda|=1\), put
\[
E_\lambda:=\ker(L-\lambda I)^N=\ker(L_s-\lambda I)
\qquad(N\gg0).
\]
Then \(\Omega|_{E_\lambda}\) is nondegenerate, and its signature is
unchanged by passage to the associated graded.
Consequently, the contribution to the rho invariant from eigenvalues
\(\lambda\) with \(|\lambda|=1\) and \(\lambda\ne1\) is the same for
\(L\) and \(L^{\mathrm{gr}}\).
\end{enumerate}
\end{lemma}

\begin{proof}
The Levi projection
\(\pi_F:P_F\to \operatorname{GL}(F)\times \operatorname{U}(E_0,\Omega_0)\)
is an algebraic group homomorphism.  If \(L=L_sL_u=L_uL_s\) is the
multiplicative Jordan decomposition of \(L\), then \(\pi_F(L_s)\) is
semisimple, \(\pi_F(L_u)\) is unipotent, and these two elements commute.
By the uniqueness of the multiplicative Jordan decomposition, they are the
semisimple and unipotent parts of \(\pi_F(L)\).  This proves (1).

We next prove the compatibility with primary decompositions.  Since \(L\)
stabilizes \(F\) and preserves \(\Omega\), it also stabilizes \(F^\perp\).
Thus the filtration
\[
0\subset F\subset F^\perp\subset E
\]
is \(L\)-invariant.  With respect to the Witt decomposition
\(E=F\oplus E_0\oplus F^{-}\), the matrix of \(L\) is block upper triangular,
and the action induced by \(L\) on the associated graded of this filtration is
precisely
\[
L^{\mathrm{gr}}=\pi_F(L)=
\begin{pmatrix}
A&0&0\\
0&L_0&0\\
0&0&(A^*)^{-1}
\end{pmatrix}.
\]
For each eigenvalue \(\lambda\), let
\(E_\lambda=\ker(L-\lambda I)^N\) for \(N\gg0\). Equivalently,
\(E_\lambda=\ker(L_s-\lambda I)\). The primary projector is a polynomial in
\(L\), so it preserves \(F\) and \(F^\perp\). Thus the filtration restricts
to \(E_\lambda\), and passage to its associated graded removes precisely
the off-diagonal blocks. This gives the \(\lambda\)-primary component of
\(L^{\mathrm{gr}}\), proving (2).

For primary spaces one has the orthogonality relation
\[
\Omega(E_\lambda,E_\mu)=0
\qquad\text{unless}\qquad \lambda\overline\mu=1.
\]
This follows first for eigenvectors from \(L\)-invariance and then for
generalized eigenvectors by induction on the two nilpotence orders. If
\(|\lambda|=1\), the only primary space that can pair with
\(E_\lambda\) is \(E_\lambda\) itself; nondegeneracy of \(\Omega\) on
\(E\) therefore implies that \(\Omega|_{E_\lambda}\) is nondegenerate.

Put \(F_\lambda=F\cap E_\lambda\). Since the primary projectors preserve
\(F\), the preceding orthogonality also gives
\(
F^\perp\cap E_\lambda
=F_\lambda^{\perp,E_\lambda}.
\)
Hence
\(0\subset F_\lambda\subset F^\perp\cap E_\lambda\subset E_\lambda\)
is a self-dual filtration. Its associated graded is the orthogonal sum of
the nondegenerate middle quotient and the hyperbolic pair formed by the
first and last graded pieces. The hyperbolic pair has signature zero.
Consequently \(E_\lambda\) and its associated graded have the same Witt
class and the same signature.

For \(\lambda=e^{2\pi\sqrt{-1}\theta}\) with \(0<\theta<1\),
the contribution of this primary space to the rho invariant is
\[
(1-2\theta)\operatorname{sign}(\Omega|_{E_\lambda});
\]
see \cite[Theorem~2(3)]{KPWUnitary}. Both factors are unchanged by the Levi
projection. Summing over these eigenvalues proves the last assertion in
(3). This argument does not cover the eigenvalue \(1\), where the
nilpotent part can affect the rho invariant.
\end{proof}

\begin{lemma}
\label{lem:split-hyperbolic-levi-factor}
For the representation
\(\alpha\oplus\alpha^{-*}\) on \(F\oplus F^{-}\), one has
\begin{equation}
\label{eq:split-factor-vanishing}
\operatorname{sign}
\bigl(\alpha\oplus\alpha^{-*}\bigr)
=
T\bigl(\Sigma,\alpha\oplus\alpha^{-*}\bigr)
=
0.
\end{equation}
Moreover, for every boundary component \(c_j\),
\(
\rho\bigl(
\alpha(c_j)\oplus\alpha(c_j)^{-*}
\bigr)=0.
\)
Consequently,
\[
\operatorname{sign}(\phi^{\mathrm{gr}})
=
\operatorname{sign}(\phi_0),
\qquad
T(\Sigma,\phi^{\mathrm{gr}})
=
T(\Sigma,\phi_0).
\]
\end{lemma}

\begin{proof}
Define a flat involution
\[
\tau\colon F\oplus F^{-}\longrightarrow F\oplus F^{-},
\qquad
\tau(v,w)=(v,-w).
\]
The map \(\tau\) commutes with the holonomy
\(\alpha\oplus\alpha^{-*}\) and satisfies
\(\tau^*\Omega_{\mathrm{hyp}}=-\Omega_{\mathrm{hyp}}\). It therefore induces
an anti-isometry of the twisted Hermitian intersection form. Hence its
positive and negative indices are equal, and
\[
\operatorname{sign}
\bigl(\alpha\oplus\alpha^{-*}\bigr)=0.
\]

Naturality of the boundary operator under the flat bundle isomorphism
\(\tau\) identifies the rho invariant of
\((F\oplus F^{-},\Omega_{\mathrm{hyp}})\) with that of the same holonomy
equipped with \(-\Omega_{\mathrm{hyp}}\). Reversing the Hermitian form
reverses the sign of the rho invariant (both the eta term and the primitive
term in \eqref{eq:def-rho} change sign). Hence this rho invariant equals
its negative and is zero. Equivalently, in
\cite[Theorem~2]{KPWUnitary} the elliptic and signed-unipotent terms cancel
in dual pairs, while the hyperbolic terms vanish. Thus
\(
\rho\bigl(
\alpha(c_j)\oplus\alpha(c_j)^{-*}
\bigr)=0
\)
for every \(j\).
The signature--Toledo formula \cite[Theorem~1]{KPWUnitary} now gives
\(T(\Sigma,\alpha\oplus\alpha^{-*})=0\). The last two identities follow from
the additivity of the signature, the Toledo invariant, and the rho invariant
under orthogonal direct sums.
\end{proof}

We next record a finite-dimensional consequence of the unipotent rho
formula. It expresses the remaining boundary contribution as the signature
of a Hermitian matrix, so that we can estimate it by linear algebra.

\begin{lemma}
\label{lem:unipotent-rho-finite-signature}
Let \(B\in\mathfrak u(E,\Omega)\) be nilpotent, and define the possibly
degenerate Hermitian form
\(
h_B(u,v):=\Omega(\sqrt{-1}Bu,v).
\)
Then
\begin{equation}
\label{eq:unipotent-rho-finite-signature}
\rho\bigl(\exp(2\pi B)\bigr)
=
-\operatorname{sign}(h_B),
\end{equation}
where zero eigenvalues are omitted when taking the signature of \(h_B\).
\end{lemma}

\begin{proof}
The skew-adjointness of \(B\) with respect to \(\Omega\) gives
\(\Omega(Bu,v)=-\Omega(u,Bv)\). Since our Hermitian forms are linear in
the first variable, this identity shows that \(h_B\) is Hermitian.
Since a nilpotent endomorphism has trace zero,
\(B\in\mathfrak{su}(E,\Omega)\). By
\cite[Proposition~4.15]{KPWUnitary}, \((E,\Omega,B)\) is an
\(\Omega\)-orthogonal direct sum of \(B\)-invariant indecomposable blocks.
Both sides of \eqref{eq:unipotent-rho-finite-signature} are additive, so it
suffices to consider one Jordan block of dimension \(d\).

Choose a Jordan chain \(e_0,\ldots,e_{d-1}\) with
\(Be_j=e_{j-1}\) for \(j\geq1\), \(Be_0=0\), and put
\(c:=\Omega(e_0,e_{d-1})\). From
\(\Omega(Bu,v)+\Omega(u,Bv)=0\) one obtains
\[
\Omega(e_r,e_s)=0\quad(r+s<d-1),
\quad
\Omega(e_r,e_{d-1-r})=(-1)^r c.
\]
Nondegeneracy of \(\Omega\) implies \(c\ne0\). Moreover,
\(\operatorname{rad}(h_B)=\ker B=\mathbb C e_0\). On
\(H=\operatorname{span}\{e_1,\ldots,e_{d-1}\}\), the induced form is
nondegenerate and
\[
h_B(e_j,e_k)=0\quad(j+k<d),
\quad
h_B(e_j,e_{d-j})=\sqrt{-1}(-1)^{j-1}c.
\]

If \(\dim H\geq2\), its first and last basis vectors span a
\(2\times2\) Gram block with negative determinant, hence a hyperbolic
plane. The first vector pairs trivially with all intermediate vectors.
Subtracting suitable multiples of it from the intermediate vectors makes
them orthogonal to the last vector, without changing their mutual
pairings. We can therefore split off the hyperbolic plane and repeat the
argument on the smaller anti-triangular matrix. It follows that, if
\(d=2\ell+1\), the form on \(H\) is a sum of \(\ell\) hyperbolic planes and
\(\operatorname{sign}(h_B)=0\). If \(d=2\ell\), it is a sum of
\(\ell-1\) hyperbolic planes and one line whose sign is the sign of
\[
h_B(e_\ell,e_\ell)
=\sqrt{-1}(-1)^{\ell-1}c
=\Omega\bigl((\sqrt{-1}B)^{2\ell-1}e_{2\ell-1},e_{2\ell-1}\bigr).
\]
The last expression is exactly the primitive Hermitian form induced on
\(E/BE\) in \cite[Definition~4.16]{KPWUnitary}. The block formula
\cite[Theorem~4.17]{KPWUnitary} is zero for odd blocks and the negative of
this primitive sign for even blocks. This proves
\eqref{eq:unipotent-rho-finite-signature}.
\end{proof}

\begin{lemma}
\label{lem:rho-defect-parabolic-extension}
Let \(L\in P_F\), and let
\(
L^{\mathrm{gr}}
=
\lim_{t\to 0}a_tLa_t^{-1}
\)
be its Levi degeneration. Then
\begin{equation}
\label{eq:rho-defect-parabolic-extension}
\left|
\rho(L)-\rho(L^{\mathrm{gr}})
\right|
\leq 2s.
\end{equation}
\end{lemma}

\begin{proof}
Decompose \(L\) according to its eigenvalues: those off the unit circle,
those on the unit circle other than \(1\), and the eigenvalue \(1\).
These give the canonical hyperbolic-unipotent, elliptic-unipotent, and
unipotent summands of \cite[Proposition~4.6]{KPWUnitary}.
By Lemma~\ref{lem:Levi-Jordan-compatibility}, the Levi projection preserves the
multiplicative Jordan decomposition and the primary decomposition. The
hyperbolic-unipotent summand contributes zero to the rho invariant, see
\cite[Lemma~4.7]{KPWUnitary}. The elliptic-unipotent contributions agree by
Lemma~\ref{lem:Levi-Jordan-compatibility}(3) and
\cite[Theorem~2(3)]{KPWUnitary}. It remains only to control the unipotent
summand.

Let \(E_1=\ker(L-I)^N\) for \(N\gg0\). By
Lemma~\ref{lem:Levi-Jordan-compatibility}(3), \(E_1\) is nondegenerate. Put
\[
F_1:=F\cap E_1,
\qquad
F_1^{\perp,1}:=F^\perp\cap E_1,
\qquad
s_1:=\dim F_1\leq s.
\]
The proof of that lemma shows that
\(F_1^{\perp,1}\) is the orthogonal complement of \(F_1\) inside \(E_1\).
Choose an internal Witt decomposition
\[
E_1=F_1\oplus E_{1,0}\oplus F_1^{-}.
\]
By Lemma~\ref{lem:Levi-Jordan-compatibility}(2), the eigenvalue-one primary
component of \(L^{\mathrm{gr}}\) is the associated graded of \(L|_{E_1}\).
Since the logarithm of a unipotent operator is a finite polynomial in
\(L-I\), it commutes with passage to the associated graded. Hence
\[
L_1:=L|_{E_1}=\exp(2\pi B),
\qquad
L_1^{\mathrm{gr}}
=
\exp(2\pi B^{\mathrm{gr}}),
\]
where \(B=(2\pi)^{-1}\log(L|_{E_1})\), and \(B^{\mathrm{gr}}\) is its
associated graded in the internal Witt decomposition above.

With respect to the induced Witt decomposition, \(B\) and its associated
graded have the form
\[
B=
\begin{pmatrix}
B_{F_1}&X&Z\\
0&B_0&Y\\
0&0&-B_{F_1}^*
\end{pmatrix},
\qquad
B^{\mathrm{gr}}
=
\begin{pmatrix}
B_{F_1}&0&0\\
0&B_0&0\\
0&0&-B_{F_1}^*
\end{pmatrix}.
\]
Here \(B_{F_1}=B|_{F_1}\), while \(B_0\) is the endomorphism induced by
\(B\) on the quotient \(F_1^\perp/F_1\). Since \(B\) is nilpotent and
preserves this filtration, both \(B_{F_1}\) and \(B_0\) are nilpotent;
moreover, \(-B_{F_1}^*\) is nilpotent. Thus \(B^{\mathrm{gr}}\) is nilpotent.

For \(t>0\), define the internal Witt dilation
\[
b_t:=\operatorname{diag}
\bigl(tI_{F_1},I_{E_{1,0}},t^{-1}I_{F_1^-}\bigr)
\in\mathrm U(E_1,\Omega|_{E_1})
\]
and set \(B_t=b_tBb_t^{-1}\). Then
\begin{equation}
\label{eq:nilpotent-Levi-difference}
B_t-B^{\mathrm{gr}}
=
\begin{pmatrix}
0&tX&t^2Z\\
0&0&tY\\
0&0&0
\end{pmatrix}.
\end{equation}
Since \(b_t\) preserves \(\Omega|_{E_1}\), the Hermitian forms
\(h_{B_t}\) and \(h_B\) are congruent. In particular, they have the same
signature and the same rank. Moreover,
\(h_{B_t}\to h_{B^{\mathrm{gr}}}\) as \(t\to0\).

Let \(n_\pm(h)\) denote the positive and negative indices of a Hermitian
form \(h\). By semicontinuity of inertia,
\[
n_\pm\bigl(h_{B^{\mathrm{gr}}}\bigr)
\leq
n_\pm(h_{B_t})
\]
for all sufficiently small \(t>0\). Therefore
\begin{align*}
\left|
\operatorname{sign}(h_B)
-
\operatorname{sign}\bigl(h_{B^{\mathrm{gr}}}\bigr)
\right|
&=
\left|
\operatorname{sign}(h_{B_t})
-
\operatorname{sign}\bigl(h_{B^{\mathrm{gr}}}\bigr)
\right|\\
&\leq
\operatorname{rank}(h_{B_t})
-
\operatorname{rank}\bigl(h_{B^{\mathrm{gr}}}\bigr)\\
&\leq
\operatorname{rank}
\bigl(h_{B_t}-h_{B^{\mathrm{gr}}}\bigr).
\end{align*}
The first inequality uses the fact that both inertia indices can only
decrease at the limit. The second follows from
\(\operatorname{rank}(A)\leq
\operatorname{rank}(A-C)+\operatorname{rank}(C)\).
Because \(\Omega\) is nondegenerate,
\[
\operatorname{rank}
\bigl(h_{B_t}-h_{B^{\mathrm{gr}}}\bigr)
=
\operatorname{rank}
\bigl(B_t-B^{\mathrm{gr}}\bigr).
\]
The image of the matrix in
\eqref{eq:nilpotent-Levi-difference} is contained in
\(F_1\oplus\operatorname{Im}(Y)\). Since
\(\dim F_1=s_1\) and \(\operatorname{rank}(Y)\leq s_1\), it follows that
\[
\operatorname{rank}
\bigl(B_t-B^{\mathrm{gr}}\bigr)
\leq 2s_1\leq2s.
\]
Lemma~\ref{lem:unipotent-rho-finite-signature} now yields
\[
\left|
\rho(L_1)
-
\rho(L_1^{\mathrm{gr}})
\right|
\leq 2s.
\]
The hyperbolic-unipotent and elliptic-unipotent contributions agree on the
two sides, so \eqref{eq:rho-defect-parabolic-extension} follows.
\end{proof}

\begin{corollary}
\label{cor:controlled-isotropic-reduction}
With the notation above, one has
\begin{equation}
\label{eq:Toledo-controlled-reduction}
T(\Sigma,\phi)=T(\Sigma,\phi_0)
\end{equation}
and
\begin{equation}
\label{eq:signature-controlled-reduction}
\left|
\operatorname{sign}(\phi)
-
\operatorname{sign}(\phi_0)
\right|
\leq 2sn.
\end{equation}
\end{corollary}

\begin{proof}
For every \(t>0\), the representation \(\phi_t\) is conjugate to \(\phi\),
so \(T(\Sigma,\phi_t)=T(\Sigma,\phi)\). The action of
\(\mathrm U(p,q)\) on its Hermitian symmetric space, and hence its bounded
K\"ahler class, factors through \(\mathrm{PU}(p,q)\). The Toledo invariant is
continuous on the latter representation variety by
\cite[Theorem~1(2)]{BIW}; its pullback to
\(\operatorname{Hom}(\pi_1(\Sigma),\mathrm U(p,q))\) is therefore
continuous. Since \(\phi_t\to\phi^{\mathrm{gr}}\), we obtain
\[
T(\Sigma,\phi)
=
T(\Sigma,\phi^{\mathrm{gr}})
=
T(\Sigma,\phi_0),
\]
where the second equality follows from
Lemma~\ref{lem:split-hyperbolic-levi-factor}.

Using the signature--Toledo formula
\cite[Theorem~1]{KPWUnitary}, together with
\(\operatorname{sign}(\phi^{\mathrm{gr}})
=\operatorname{sign}(\phi_0)\), we find
\[
\operatorname{sign}(\phi)-\operatorname{sign}(\phi_0)
=
\sum_{j=1}^{n}
\left[
\rho\bigl(\phi(c_j)\bigr)
-
\rho\bigl(\phi^{\mathrm{gr}}(c_j)\bigr)
\right].
\]
Applying Lemma~\ref{lem:rho-defect-parabolic-extension} to each boundary
holonomy gives
\[
\left|
\operatorname{sign}(\phi)
-
\operatorname{sign}(\phi_0)
\right|
\leq
\sum_{j=1}^{n}2s
=
2sn.
\]
\end{proof}

\subsection{The sharp upper bound for arbitrary representations}
\label{subsec:upper-bound-reduction}

We can now combine Corollary~\ref{cor:global-residue-compensation} with
controlled Levi degeneration. The boundary error is at most \(2sn\),
whereas reducing the two ranks by \(s\) decreases the proposed upper bound
by \(2sn+4s(g-1)\). Since \(g\geq1\), this is enough to close an
induction on the smaller rank.

For \(0\leq a\leq b\), put
\[
M_{g,n}(a,b):=2a(2g-2)+n(a+b)-\delta(n,b-a).
\]
We include \((a,b)=(0,0)\): the group \(\mathrm U(0,0)\) is trivial,
its zero-dimensional local system has signature zero, and
\(M_{g,n}(0,0)=0\).

\begin{proposition}
\label{prop:extension-to-arbitrary-representations}
Fix \(g,n\geq1\). Assume that
\(
\operatorname{sign}(\psi)\leq M_{g,n}(a,b)
\)
has been proved for every reductive representation into
\(\mathrm{U}(a,b)\), for all \(1\leq a\leq b\), and assume the known compact
case \(a=0\). Then the same estimate holds for every representation into
\(\mathrm{U}(p,q)\), with \(q\geq p\).
\end{proposition}

\begin{proof}
We argue by induction on \(p\). The case \(p=0\) is the compact-unitary
case, including the zero-dimensional convention above. Let \(p>0\).
If \(\phi\) is reductive, the asserted inequality is
exactly the hypothesis at rank \((p,q)\). Suppose therefore that \(\phi\)
is nonreductive. Choose a nonzero invariant isotropic subspace of dimension
\(s\), and let
\(
\phi_0\colon\pi_1(\Sigma_{g,n})\to\mathrm{U}(p-s,q-s)
\)
be the middle factor of the Levi degeneration. By
Corollary~\ref{cor:controlled-isotropic-reduction},
\[
\operatorname{sign}(\phi)
\leq
\operatorname{sign}(\phi_0)+2sn.
\]
The induction hypothesis gives
\(
\operatorname{sign}(\phi_0)\leq M_{g,n}(p-s,q-s)
\).
Since the rank difference is unchanged,
\[
M_{g,n}(p-s,q-s)+2sn
=M_{g,n}(p,q)-4s(g-1)\leq M_{g,n}(p,q).
\]
This proves the assertion.
\end{proof}

\begin{theorem}
\label{thm:sharp-upper-all-representations}
Let \(g,n\geq1\) and \(q\geq p>0\). Every representation
\(
\phi\colon\pi_1(\Sigma_{g,n})\to\mathrm{U}(p,q)
\)
satisfies
\begin{equation}
\label{eq:sharp-upper-all-representations}
|\operatorname{sign}(\phi)|\leq M_{g,n}(p,q).
\end{equation}
\end{theorem}

\begin{proof}
If \(q=p\), the signature Milnor--Wood inequality
\eqref{eq:mw-signature} gives
\[
|\operatorname{sign}(\phi)|
\leq2p|\chi(\Sigma_{g,n})|
=2p(2g-2+n)
=M_{g,n}(p,p),
\]
because \(\delta(n,0)=0\). Assume henceforth that \(q>p\). For reductive
representations the upper bound follows from
Corollary~\ref{cor:global-residue-compensation}; the balanced cases in the
hypotheses of Proposition~\ref{prop:extension-to-arbitrary-representations}
are covered by \eqref{eq:mw-signature}, and the compact case is
Theorem~\ref{thm:complete-Up-range}. That proposition
therefore extends the upper bound to every representation.

To obtain the lower bound, choose a basis in which
\(\Omega=\operatorname{diag}(I_p,-I_q)\). Complex conjugation then
preserves \(\mathrm U(p,q)\) and induces an antilinear map on twisted
cohomology satisfying
\[
iQ_{\overline\phi}(\overline x,\overline y)
=-\overline{iQ_\phi(x,y)}.
\]
Thus \(\operatorname{sign}(\overline\phi)=-\operatorname{sign}(\phi)\).
Applying the upper bound to \(\overline\phi\) gives the lower bound.
\end{proof}
\section{Realization of all values}
\label{subsec:realization-positive-genus}

We now construct a representation for each value in the interval of
Theorem~\ref{thm:positive-genus-realization-bounds}. The idea is to split the
coefficient space into \(p\) pieces of signature \((1,1)\) and a
negative-definite piece of dimension \(q-p\). We choose a representation on
each piece and add their signatures.

\begin{proposition}
\label{prop:realization-positive-genus}
Let \(g\geq1\), \(n\geq1\), and \(q\geq p>0\). Then every integer in
\begin{equation}
\label{eq:realized-positive-genus-interval}
[-M_{g,n}(p,q),M_{g,n}(p,q)]\cap\mathbb Z
\end{equation}
is the signature of a representation
\(
\pi_1(\Sigma_{g,n})\to\mathrm{U}(p,q).
\)
\end{proposition}

\begin{proof}
Set
\(
d:=q-p.
\)
We first construct the balanced part of the representation. By
Corollary~\ref{cor:signature-values-U11}, one has
\begin{equation}
\label{eq:rank-one-realization-positive-genus}
\mathscr{S}_{1,1}(\Sigma_{g,n})
=
[-2|\chi(\Sigma_{g,n})|,2|\chi(\Sigma_{g,n})|]
\cap\mathbb Z.
\end{equation}
Since \(g,n\geq1\), 
\(
|\chi(\Sigma_{g,n})|=2g-2+n.
\)
Taking orthogonal direct sums of \(p\) rank-one representations and using
the block-diagonal inclusion
\(
\mathrm{U}(1,1)^p
\hookrightarrow
\mathrm{U}(p,p),
\)
we obtain representations realizing every integer in
\begin{equation}
\label{eq:balanced-realization-range}
[-A_{g,n,p},A_{g,n,p}]\cap\mathbb Z,
\qquad
A_{g,n,p}:=2p(2g-2+n).
\end{equation}
Indeed, every integer in this interval can be written as a sum of \(p\)
integers belonging to the rank-one interval in
\eqref{eq:rank-one-realization-positive-genus}. The signature is additive
under orthogonal direct sums; see, for example,
\cite[Subsection~1.4]{KPWValues}.

If \(d=0\), then \(q=p\) and \(\delta(n,d)=0\). Consequently,
\[
A_{g,n,p}
=
2p(2g-2+n)
=
2p(2g-2)+n(p+q)
=
M_{g,n}(p,p).
\]
Thus \eqref{eq:balanced-realization-range} already gives the full interval
\eqref{eq:realized-positive-genus-interval}.

Assume now that \(d>0\), and put
\(B_{n,d}:=nd-\delta(n,d)=\max\{0,nd-2\}\).
If \(nd\geq2\), Proposition~\ref{prop:positive-genus-Up} provides
representations into \(\mathrm{U}(d)\) with every integer signature between
\(-B_{n,d}\) and \(B_{n,d}\). If \(nd=1\), then \(n=d=1\) and
\(B_{n,d}=0\); the trivial character realizes the required value \(0\).
Regard these representations as acting on a negative-definite Hermitian
space. Changing the sign of the coefficient form changes the sign of the
intersection form, and hence reverses the signature. Because the interval
is symmetric, representations into \(\mathrm{U}(0,d)\) therefore realize
every integer in
\begin{equation}
\label{eq:residual-compact-range}
[-B_{n,d},B_{n,d}]\cap\mathbb Z,
\qquad
B_{n,d}=nd-\delta(n,d).
\end{equation}
Consider the orthogonal block-diagonal inclusion
\[
\mathrm{U}(p,p)\times\mathrm{U}(0,d)
\hookrightarrow
\mathrm{U}(p,p+d)
=
\mathrm{U}(p,q).
\]
Let \(u\in[-A_{g,n,p},A_{g,n,p}]\cap\mathbb Z\) and
\(v\in[-B_{n,d},B_{n,d}]\cap\mathbb Z\). Choose representations
\(
\phi_u\colon
\pi_1(\Sigma_{g,n})\to\mathrm{U}(p,p),
\) \(
\psi_v\colon
\pi_1(\Sigma_{g,n})\to\mathrm{U}(0,d),
\)
such that
\(
\operatorname{sign}(\phi_u)=u,
\) \(
\operatorname{sign}(\psi_v)=v.
\)
Their orthogonal direct sum satisfies
\[
\operatorname{sign}(\phi_u\oplus\psi_v)=u+v.
\]

Adding an integer from \([-A_{g,n,p},A_{g,n,p}]\) to one from
\([-B_{n,d},B_{n,d}]\) gives every integer in
\[
[-(A_{g,n,p}+B_{n,d}),A_{g,n,p}+B_{n,d}]
\cap\mathbb Z.
\]
Finally,
\begin{align*}
A_{g,n,p}+B_{n,d}
&=
2p(2g-2+n)+nd-\delta(n,d)\\
&=
2p(2g-2)+2pn+n(q-p)-\delta(n,q-p)\\
&=
2p(2g-2)+n(p+q)-\delta(n,q-p)\\
&=
M_{g,n}(p,q).
\end{align*}
Therefore every integer in
\eqref{eq:realized-positive-genus-interval} is realized.
\end{proof}

Combining Proposition~\ref{prop:realization-positive-genus} with
Theorem~\ref{thm:sharp-upper-all-representations} proves
Theorem~\ref{thm:positive-genus-realization-bounds}. Together with the
rank-one, genus-zero, compact, and closed-surface cases proved earlier,
this completes the proof of
Theorem~\ref{thm:introduction-main}.

\appendix
\section{The \texorpdfstring{\(L^2\)}{L2} minimal-extension comparison}
\label{app:L2-minimal-extension}

We use the notation of Subsection~\ref{subsec:L2-Dolbeault-complex}.
In particular, \(\mathscr E=V\oplus W\) is the normalized logarithmic
lattice, and \(\mathcal A\), \(\mathcal B\), and
\(\mathcal C_{L^2}^\bullet\) are defined in
\eqref{eq:def-L2-lattice-A}, \eqref{eq:def-L2-lattice-B}, and
\eqref{eq:L2-total-Dolbeault-complex}.
We first identify the holomorphic \(L^2\) graph domain explicitly.
We then compare it with the coherent complex used in the proof, and finally
pass to global harmonic representatives and intersection cohomology.

Write \(\mathcal G_{(2),\mathrm{hol}}^\bullet(D^0)\) for the holomorphic
graph-domain complex. On an open set \(U\subset\overline\Sigma\), its
sections in degree zero are holomorphic sections \(u\) on
\(U\cap\Sigma^\circ\) for which both \(u\) and \(\Phi u\) are locally
\(L^2\), including near the marked points. Its sections in degree one are
locally \(L^2\) holomorphic one-forms with values in the bundle. The
differential is \(\Phi\). The following lemma shows, in particular, that
these sections belong to the normalized logarithmic lattice.

\begin{lemma}
\label{lem:filtered-holomorphic-graph-domain-strictness}
Fix a marked point \(x_j\). On the weight-zero graded fiber, write
\[
\operatorname{GrRes}_{x_j}(\Phi)=S_j+N_j,
\qquad [S_j,N_j]=0,
\qquad
\operatorname{Gr}_{0,j}(\mathscr E)=\bigoplus_{\lambda}E_{j,\lambda},
\]
where \(S_j\) acts on \(E_{j,\lambda}\) as \(\lambda\operatorname{id}\).
Let \(W_\bullet E_{j,\lambda}\) be the weight filtration of
\(N_j|_{E_{j,\lambda}}\), centered at zero, and let
\(q_{j,\lambda}\) be evaluation followed by the canonical projections to
the weight-zero graded fiber and then to \(E_{j,\lambda}\).
Define local subsheaves of \(\mathscr E\) by
\begin{equation}
\label{eq:filtered-graph-domain-associated-graded}
\begin{aligned}
\mathcal A^{\mathrm{hol}}_j
&=\{u:q_{j,0}(u)\in W_0E_{j,0},\quad
q_{j,\lambda}(u)\in W_{-2}E_{j,\lambda}
\text{ for }\lambda\ne0\},\\
\mathcal B^{\mathrm{hol}}_j
&=\{v:q_{j,\lambda}(v)\in W_{-2}E_{j,\lambda}
\text{ for every }\lambda\}.
\end{aligned}
\end{equation}
A condition on an absent primary summand is understood to be vacuous.
Then, on a sufficiently small coordinate disc \(\Delta_j\),
\[
\mathcal G_{(2),\mathrm{hol}}^\bullet(D^0)
=
\left[\mathcal A^{\mathrm{hol}}_j
\xrightarrow{\Phi}
\mathcal B^{\mathrm{hol}}_j\otimes
K_{\overline\Sigma}(D)\right].
\]
These descriptions are intrinsic: the projections are taken on the
weight-zero graded fiber, and no holomorphic splitting preserved by
\(\Phi\) is required.
\end{lemma}

\begin{proof}
Choose a coordinate \(z\) centered at \(x_j\), and put
\(r=|z|\) and \(t=-\log r\). The tame harmonic metric is uniformly
comparable with its local model. In a holomorphic frame adapted to the
parabolic weights and the nilpotent weight filtrations, this gives
\begin{equation}
\label{eq:adapted-splitting-quasi-orthogonal}
C^{-1}\sum_a |u_a|^2r^{2\alpha_a}t^{\ell_a}
\leq
|\sum_a u_ae_a|_h^2
\leq
C\sum_a |u_a|^2r^{2\alpha_a}t^{\ell_a},
\end{equation}
where \(0\leq\alpha_a<1\) and \(\ell_a\in\mathbb Z\).
Here \(\ell_a\) is the weight-filtration index of the corresponding
basis vector. These are the local-model estimates in
\cite[Section~5.1, especially (5.8), and Theorem~5.1]{BGPMiR20},
expressed in the normalized lattice.
The same comparison applies to bundle-valued forms.

For the Poincar\'e metric,
\[
d\operatorname{vol}_{\mathrm P}\asymp
\frac{dr\,d\theta}{rt^2},
\qquad
|\frac{dz}{z}|_{\mathrm P}^2\asymp t^2.
\]
Consequently, the radial integrals for a constant adapted vector of
weight \(\alpha\) and index \(\ell\), as a section and as the coefficient
of \(dz/z\), are respectively
\[
\int^\infty e^{-2\alpha t}t^{\ell-2}\,dt,
\qquad
\int^\infty e^{-2\alpha t}t^\ell\,dt.
\]
Both converge when \(\alpha>0\). When \(\alpha=0\), they converge exactly
when \(\ell\leq0\) and \(\ell\leq-2\), respectively.
A negative Laurent term has an exponentially growing radial factor,
since \(\alpha<1\), and cannot be \(L^2\). This also excludes an essential
singularity: angular integration of the squared Laurent series detects
each negative coefficient separately. Thus an \(L^2\) holomorphic
section extends to \(\mathscr E\), and an \(L^2\) holomorphic one-form
extends to \(\mathscr E\otimes K_{\overline\Sigma}(D)\).
Terms whose coefficients vanish at \(z=0\) are always integrable.
It follows that
\begin{equation}
\label{eq:full-versus-graded-graph-domain}
\begin{aligned}
u\in L^2
&\Longleftrightarrow
q_{j,\lambda}(u)\in W_0E_{j,\lambda}
\quad\text{for every }\lambda,\\
v\,dz/z\in L^2
&\Longleftrightarrow
q_{j,\lambda}(v)\in W_{-2}E_{j,\lambda}
\quad\text{for every }\lambda.
\end{aligned}
\end{equation}

We explain why the full Higgs field gives the same conditions.
In an adapted frame write
\[
\Phi=(R_0+R_{>}+zB(z))\frac{dz}{z},
\]
where \(R_0\) preserves the parabolic weights, \(R_{>}\) strictly
increases them, and \(B(z)\) is holomorphic. A block of \(R_{>}\)
from weight \(\alpha\) to weight \(\beta>\alpha\), including the
one-form factor, has operator norm bounded by
\(Cr^{\beta-\alpha}t^N\) for some integer \(N\). A block of the regular
term has norm bounded by \(Cr^{1+\beta-\alpha}t^N\). Both are bounded
near the puncture; for the second bound, the essential point is
\(1+\beta-\alpha>0\). Thus these terms send \(L^2\) sections to
\(L^2\) one-forms. The same estimates apply to changes of lifts: two
lifts of the same fiber vector differ by a term divisible by \(z\),
and changing a lift of a weight-zero graded vector may also add a
positive-weight vector. Neither change affects the integrability
conditions in \eqref{eq:full-versus-graded-graph-domain}.

On a positive-weight block, \(R_0\) also sends holomorphic lattice
sections to \(L^2\) one-forms: the exponential decay dominates every
logarithmic power. The only remaining condition is therefore on the
weight-zero graded residue. There it acts on \(E_{j,\lambda}\) as
\(A_\lambda=\lambda\operatorname{id}+N_j\). If \(\lambda=0\), then
\(N_jW_0\subseteq W_{-2}\), so the condition \(u\in L^2\) already ensures
\(\Phi u\in L^2\). If \(\lambda\ne0\), both \(A_\lambda\) and its inverse
preserve every weight-filtration step. Hence
\[
q_{j,\lambda}(u)\in W_0,\quad
A_\lambda q_{j,\lambda}(u)\in W_{-2}
\quad\Longleftrightarrow\quad
q_{j,\lambda}(u)\in W_{-2}.
\]
This proves the asserted descriptions of both terms of the complex.
\end{proof}

\begin{lemma}
\label{lem:maximal-holomorphic-L2-comparison}
The natural inclusion
\[
\mathcal G_{(2),\mathrm{hol}}^\bullet(D^0)
\longrightarrow
\mathscr L_{(2)}^\bullet(D^0)
\]
is a quasi-isomorphism of complexes of sheaves on \(\overline\Sigma\).
\end{lemma}

\begin{proof}
This is the punctured-curve holomorphic \(L^2\) comparison for tame
harmonic bundles; see
\cite[Section~1.4.1 and Remark~1.4]{MochizukiL2Complexes}.
At parameter \(\lambda=0\), the holomorphic structure is
\(\overline\partial_{\mathscr E}\) and the differential on holomorphic
sections is \(\Phi\). Thus its holomorphic complex is exactly the one
defined above. Its analytic complex uses the distributional maximal
graph domain of \(D^0=\overline\partial_{\mathscr E}+\Phi\).
In total degree one, this is the condition that the sum
\(\overline\partial u^{1,0}+\Phi u^{0,1}\) be \(L^2\); the two summands
are not required to be separately \(L^2\). This is the domain convention
used here. Away from the marked points, the assertion is the usual
Dolbeault resolution of the Higgs complex.
\end{proof}

\begin{proposition}
\label{prop:appendix-general-residue-L2-model}
The full Higgs field satisfies
\(\Phi(\mathcal A)\subseteq
\mathcal B\otimes K_{\overline\Sigma}(D)\).
There is a natural zig-zag of quasi-isomorphisms
\[
\mathcal C_{L^2}^{\bullet}
\xleftarrow{\ \simeq\ }
\mathcal G_{(2),\mathrm{hol}}^\bullet(D^0)
\xrightarrow{\ \simeq\ }
\mathscr L_{(2)}^\bullet(D^0).
\]
For the fixed adapted harmonic metric, common \(L^2\)-harmonic
representatives give isomorphisms
\[
\mathbb H^k(\overline\Sigma,\mathcal C_{L^2}^{\bullet})
\cong H_{(2)}^k(D^0)
\cong H_{(2)}^k(D^1)
\cong IH^k(\overline\Sigma;\mathcal E_\phi).
\]
The middle isomorphism is an identification of global cohomology through
harmonic representatives. In degree one, these spaces identify with
\[
\widehat H^1(\Sigma;\mathcal E_\phi)
:=\operatorname{Im}\left(
H^1(\Sigma,\partial\Sigma;\mathcal E_\phi)
\longrightarrow H^1(\Sigma;\mathcal E_\phi)\right).
\]
\end{proposition}

\begin{proof}
First, the lattice conditions are preserved by the full Higgs field.
For a local section \(s\in\mathcal A\), projection of its residue to
\(G_j=E_{j,0}\) gives
\[
q_j\bigl(\operatorname{Res}_{x_j}(\Phi)s\bigr)
=N_jq_j(s)\in W_{-2}G_j.
\]
The strictly parabolic residue has zero projection to the weight-zero
quotient, and every regular term has coefficient divisible by \(z\).
Thus \(\Phi s\) satisfies the defining condition of
\(\mathcal B\otimes K_{\overline\Sigma}(D)\).

The local graph-domain lemma gives termwise inclusions
\(\mathcal A^{\mathrm{hol}}_j\subseteq\mathcal A\) and
\(\mathcal B^{\mathrm{hol}}_j\subseteq\mathcal B\), hence a natural
inclusion of complexes
\[
\iota:\mathcal G_{(2),\mathrm{hol}}^\bullet(D^0)
\longrightarrow\mathcal C_{L^2}^\bullet.
\]
We compute its cokernel directly at each marked point. The lattice
conditions agree on positive parabolic weights and on the zero-primary
part of the weight-zero fiber. They differ only on \(E_{j,\lambda}\)
with \(\lambda\ne0\): the coherent model imposes no condition there,
whereas both holomorphic graph lattices require the fiber value to lie
in \(W_{-2}E_{j,\lambda}\). The cokernel is therefore the skyscraper
complex
\[
\bigoplus_{\lambda\ne0}
\left[
E_{j,\lambda}/W_{-2}E_{j,\lambda}
\xrightarrow{\ \overline A_\lambda\ }
E_{j,\lambda}/W_{-2}E_{j,\lambda}
\right],
\]
where we have used \(dz/z\) to identify the logarithmic fiber in degree
one with \(\mathbb C\). Regular terms vanish in this quotient, and
strictly parabolic residue terms disappear upon projection to the
weight-zero fiber. Thus the displayed differential is precisely the one
induced by \(A_\lambda=\lambda\operatorname{id}+N_j\).
It is invertible, with inverse induced by the finite sum
\[
A_\lambda^{-1}
=\lambda^{-1}\sum_{m=0}^{\nu-1}(-\lambda^{-1}N_j)^m,
\qquad N_j^\nu=0.
\]
The cokernel complex is acyclic. Since the two complexes agree away from
\(D\), \(\iota\) is a quasi-isomorphism everywhere.
On \(G_j\), the conditions are exactly the nilpotent
\(W_0/W_{-2}\) model of \cite[Section~3.2, equation~(2)]{DPS}.
This argument also explains why the individual coherent lattices need
not themselves consist entirely of \(L^2\) sections: the extra
nonzero-primary terms form acyclic quotient complexes.
The other arrow in the zig-zag is
Lemma~\ref{lem:maximal-holomorphic-L2-comparison}.

We now pass to global cohomology. The maximal-domain sheaves are fine.
Indeed, take a smooth partition of unity on the compactification
\(\overline\Sigma\). For each of its functions \(\chi\), both \(\chi\)
and \(d_\lambda\chi=\overline\partial\chi+\lambda\partial\chi\) are
bounded in the Poincar\'e metric. Near a puncture, the latter norm is
\(O(r|\log r|)\). Therefore
\[
D^\lambda(\chi u)
=\chi D^\lambda u+d_\lambda\chi\wedge u
\]
belongs to \(L^2\) whenever \(u\) is in the local maximal graph domain.
These partitions of unity preserve that domain, so global sections
compute its hypercohomology.

The \(L^2\)-Hodge theorem and the harmonic-bundle K\"ahler identities give
\[
H_{(2)}^k(D^\lambda)
\cong\operatorname{Harm}_{(2)}^k(D^\lambda),
\qquad
\Delta_\lambda=(1+|\lambda|^2)\Delta_0.
\]
The comparison with harmonic representatives, including finite
dimensionality, is supplied by
\cite[Corollary~6.8]{MochizukiL2Complexes}; the cutoff argument is
\cite[Proposition~6.5 and Corollary~6.6]{MochizukiL2Complexes}.
In particular, \(D^0\) and \(D^1\) have the same harmonic forms.
At \(\lambda=1\), the minimal-extension comparison of
\cite[Theorem~1.3]{MochizukiL2Complexes}, followed by the
Riemann--Hilbert correspondence, identifies the flat complex with the
intersection complex.

For clarity, our indexing convention is unshifted. If
\(j:\Sigma^\circ\hookrightarrow\overline\Sigma\), we set
\[
\operatorname{IC}_{\overline\Sigma}^\bullet(\mathcal E_\phi)
:=j_{!*}(\mathcal E_\phi[1])[-1],
\qquad
IH^k(\overline\Sigma;\mathcal E_\phi)
:=\mathbb H^k\bigl(\overline\Sigma,
\operatorname{IC}_{\overline\Sigma}^\bullet(\mathcal E_\phi)\bigr).
\]
Combining the preceding comparisons yields
\[
\begin{aligned}
\mathbb H^k(\overline\Sigma,\mathcal C_{L^2}^\bullet)
&\cong H_{(2)}^k(D^0)\\
&\cong\operatorname{Harm}_{(2)}^k(D^0)
=\operatorname{Harm}_{(2)}^k(D^1)\\
&\cong H_{(2)}^k(D^1)
\cong IH^k(\overline\Sigma;\mathcal E_\phi).
\end{aligned}
\]
Only the flat complex at \(\lambda=1\) is identified with the
intersection complex at the level of sheaves. The passage from
\(\lambda=0\) to \(\lambda=1\) uses global harmonic representatives.

Finally, on a punctured curve the unshifted intersection complex is
\(j_*\mathcal E_\phi\), placed in degree zero. The low-degree sequence
for \(Rj_*\) identifies its first cohomology with the kernel of
restriction to the local punctured discs, equivalently to
\(\partial\Sigma\). The long exact sequence of the pair
\((\Sigma,\partial\Sigma)\) then gives
\[
\begin{aligned}
IH^1(\overline\Sigma;\mathcal E_\phi)
&\cong\ker\left(H^1(\Sigma;\mathcal E_\phi)
\longrightarrow H^1(\partial\Sigma;\mathcal E_\phi)\right)\\
&=\operatorname{Im}\left(H^1(\Sigma,\partial\Sigma;\mathcal E_\phi)
\longrightarrow H^1(\Sigma;\mathcal E_\phi)\right).
\end{aligned}
\]
This is the claimed parabolic cohomology group.
\end{proof}

\section{Disclosure of Delegation to Generative AI}

The authors declare the use of generative AI in the research and writing process. According to the GAIDeT taxonomy (2025), the following tasks were delegated to GAI tools under full human supervision:

\begin{itemize}
\item Idea generation
\item Literature search and systematization
\item Selection of research methods
\item Text generation
\end{itemize}

The GAI tool used was: ChatGPT-5.6.

Responsibility for the final manuscript lies entirely with the authors.

GAI tools are not listed as authors and do not bear responsibility for the final outcomes.

Declaration submitted by: Xueyuan WAN

\emph{Additional note}: X. Wan suggested AI use the Higgs bundle method, and after much communication with it, AI gave the proof, and then X. Wan understood the proof with the AI's help. I. Kim and P. Pansu reorganized the crux of the proof in order to remove unnecessary definitions and notations and make the verification of the proof easier.

\bibliographystyle{amsalpha}
\bibliography{possible_values_v9}

\end{document}